\documentclass[a4paper]{amsart}

\usepackage{amsmath,amsthm,amssymb,enumerate}
\usepackage{epsfig}
\usepackage{amssymb}
\usepackage{mathtools}
\usepackage{amsmath}
\usepackage{amssymb}
\usepackage{amsmath,amsthm}
\usepackage[latin1]{inputenc}
\usepackage[T1]{fontenc}
\usepackage{path}
\usepackage{ae,aecompl}
\usepackage{amsfonts}
\usepackage{amsxtra}
\usepackage{euscript,mathrsfs}
\usepackage{color}
\usepackage[left=3.4cm,right=3.4cm,top=4cm,bottom=4cm]{geometry}
\usepackage[citecolor=blue,colorlinks=true]{hyperref}
\allowdisplaybreaks
\usepackage{enumitem}
\setenumerate{label={\rm (\alph{*})}}

\usepackage{amsfonts}
\usepackage{amsxtra}

\usepackage[frak,hyperref,theorem,section]{paper_diening}

\usepackage{float}          
\usepackage{algorithm}      
\usepackage{algpseudocode}  
\usepackage{xcolor}
\usepackage{mdframed}
\usepackage[most]{tcolorbox}

\newmdenv[
  linewidth=0.8pt,
  linecolor=black,
  backgroundcolor=gray!5,
  roundcorner=4pt,
  skipabove=10pt,
  skipbelow=10pt
]{scheme}

\tcbset{
  authorcomment/.style={
    enhanced,
    colback=blue!3,
    colframe=blue!50!black,
    coltitle=blue!20!black,
    boxrule=0.8pt,
    arc=4pt,
    left=6pt,
    right=6pt,
    top=6pt,
    bottom=6pt,
    fonttitle=\bfseries,
    title={Author's Comment},
  }
}

\newcommand{\Div}{\divergence}

\newcommand{\R}{\mathbb R}
\newcommand{\N}{\mathbb N}

\newcommand{\E}{\mathbb E}
\newcommand{\p}{\mathbb P}

\newcommand{\F}{\mathfrak F}

\newcommand{\uu}{\mathbf{u}}
\newcommand{\bfvarrho}{\boldsymbol{\varrho}}

\newcommand{\A}{\Delta}

\newcommand{\dd}{\mathrm d}
\newcommand{\dx}{\, \mathrm{d}x}

\newcommand{\dt}{\, \mathrm{d}t}
\newcommand{\dxt}{\,\mathrm{d}x\, \mathrm{d}t}
\newcommand{\dxs}{\,\mathrm{d}x\, \mathrm{d}\sigma}
\newcommand{\dif}{\mathrm{d}}

\newcommand{\mf}{\mathfrak{F}}

\newcommand{\prst}{\mathbb{P}}

\newcommand{\mn}{\mathbb{N}}

\newcommand{\mt}{\mathbb{T}^3}
\newcommand{\tor}{\mathbb{T}^3}

\DeclareMathOperator{\diver}{div}

\newcommand{\Vspace}{\mathbb V}
\newcommand{\Vdual}{\Vspace^{-1}}
\newcommand{\Qspace}{{\mathbb{L}^2}}

\usepackage{caption}
\usepackage{booktabs}
\usepackage{amsmath,amssymb}

\begin{document}

\title[Strong Rate of convergence for 3D stochastic NSEs]
{Error analysis for 3D Navier--Stokes equations with additive noise}

\author{Soumya Ranjan Behera}\address{Faculty of Mathematics, University of Duisburg-Essen, Thea-Leymann-Str. 9, 45127 Essen}\email{soumya.behera@uni-due.de}

\author{Dominic Breit}\address{Faculty of Mathematics, University of Duisburg-Essen, Thea-Leymann-Str. 9, 45127 Essen}\email{dominic.breit@uni-due.de}

\author{Abhishek Chaudhary}
\address{Mathematisches Institut der Universit\"at T\"ubingen,
                Auf der Morgenstelle 10,
                D-72076 T\"ubingen, Germany.}
\email{chaudhary@na.uni-tuebingen.de}
\begin{abstract}
We consider the three-dimensional stochastic Navier--Stokes equations with additive stochastic forcing. We study temporal discretizations based on a semi-implicit Euler as well as a Crank-Nicolson scheme. We prove that locally in time the former converges with rate 1 while the latter converges with rate 3/2.  These theoretical results are confirmed by extensive numerical simulations.
To the best of our knowledge this is the first time that numerical simulations for the three-dimensional stochastic Navier--Stokes equations have been performed.
\end{abstract}


\date{\today}

\maketitle

%
%
%
%
%
%
%
%
%
%

\section{Introduction}
Let $(\Omega,\F,(\F_t)_{t\geq0},\prst)$ be a stochastic basis with a complete, right-continuous filtration. We are interested in the numerical approximation of the  following three-dimensional stochastic Navier--Stokes equations 
\begin{align}\label{eq:SNS}
\left\{\begin{array}{rc}
\dd\uu=\nu\Delta\uu\dt-(\uu\cdot\nabla)\uu\dt-\nabla p\dt+\Phi\dd W
& \mbox{in $\mathcal O_T$,}\\
\Div \uu=0\qquad\qquad\qquad\qquad\qquad\,\,\,\,& \mbox{in $\mathcal O_T$,}\\
\uu(0)=\uu_0\,\qquad\qquad\qquad\qquad\qquad&\mbox{ \,in $\mt$,}\end{array}\right.
\end{align}
$\p$-a.s. in $\mathcal O_T:=(0,T)\times\mt$ equipped with periodic boundary conditions. Here $\mt$ is the three-dimensional, $T>0$, $\nu>0$ is the viscosity and $\uu_0$ is a given initial datum. The momentum equation is driven by a cylindrical Wiener process $W$ and the diffusion coefficient $\Phi$ takes values in the space of Hilbert-Schmidt operators; see Section \ref{sec:prob} for details.
There are various physical motivations to add stochastic components to the equations of fluid mechanics. 
The most important one is probably the mathematical description of phenomena of turbulence \cite{Bi,BE,Kr}.
From an analytical point of view most deterministic results have found their stochastic counterpart, see \cite{FL2,Ro} for an overview.
Also, some remarkable regularization effects have been obtained by a carefully chosen noise \cite{FL}. 

The primary objective of this work is to develop and analyse the semi-implicit Euler and Crank--Nicolson discretisations schemes for \eqref{eq:SNS} where the noise is additive in nature. Thereby enabling reliable and efficient simulations of stochastic fluid flows and facilitating a quantitative investigation of the influence of stochastic forcing on deterministic flow structures (see, Figures \ref{fig:lidcavity_IE_mean_T020_slices} \& \ref{fig:lidcavity_CN_mean_T020_slices} for the lid-driven cavity problem in 3D). A major challenge in the numerical approximation of \eqref{eq:SNS} arises from the driving Wiener process, whose sample paths are only H\"older continuous in time. Consequently, the limited temporal regularity of the noise restricts the attainable convergence order and necessitates sufficiently small time-step sizes to achieve accurate numerical simulations.
\begin{figure}[htbp]
    \centering
\includegraphics[width=\textwidth]{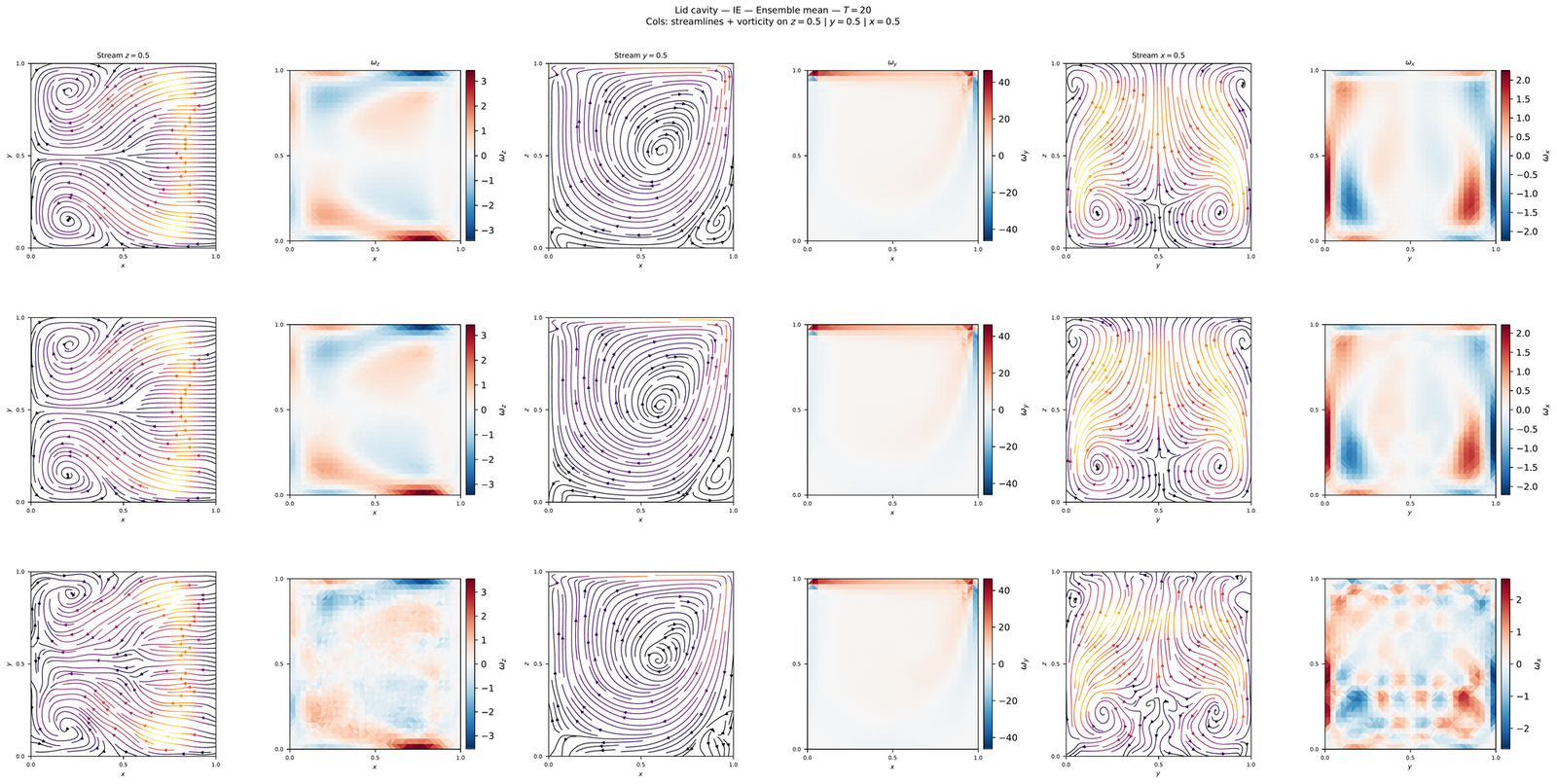}
\captionsetup{width=1\textwidth}
    \caption{Lid-driven cavity in 3D (see Section \ref{sec:Lid-3d}): Streamlines and vorticity (on the XY-plane 
$z{=}0.5$ in columns~1-2, XZ-plane 
$y{=}0.5$ in columns~3-4, YZ-plane 
$x{=}0.5$ in columns~5-6) of the deterministic solution ($\mu=0$) in row 1, streamlines and vorticity of the expected value of the solution with smaller noise ($\mu=10$) in row 2, and larger noise ($\mu=40$) in row 3 at $T=20$ for semi-implicit Euler Scheme.}
\label{fig:lidcavity_IE_mean_T020_slices}
\end{figure}

\begin{figure}[htbp]
    \centering
    \includegraphics[width=\textwidth]{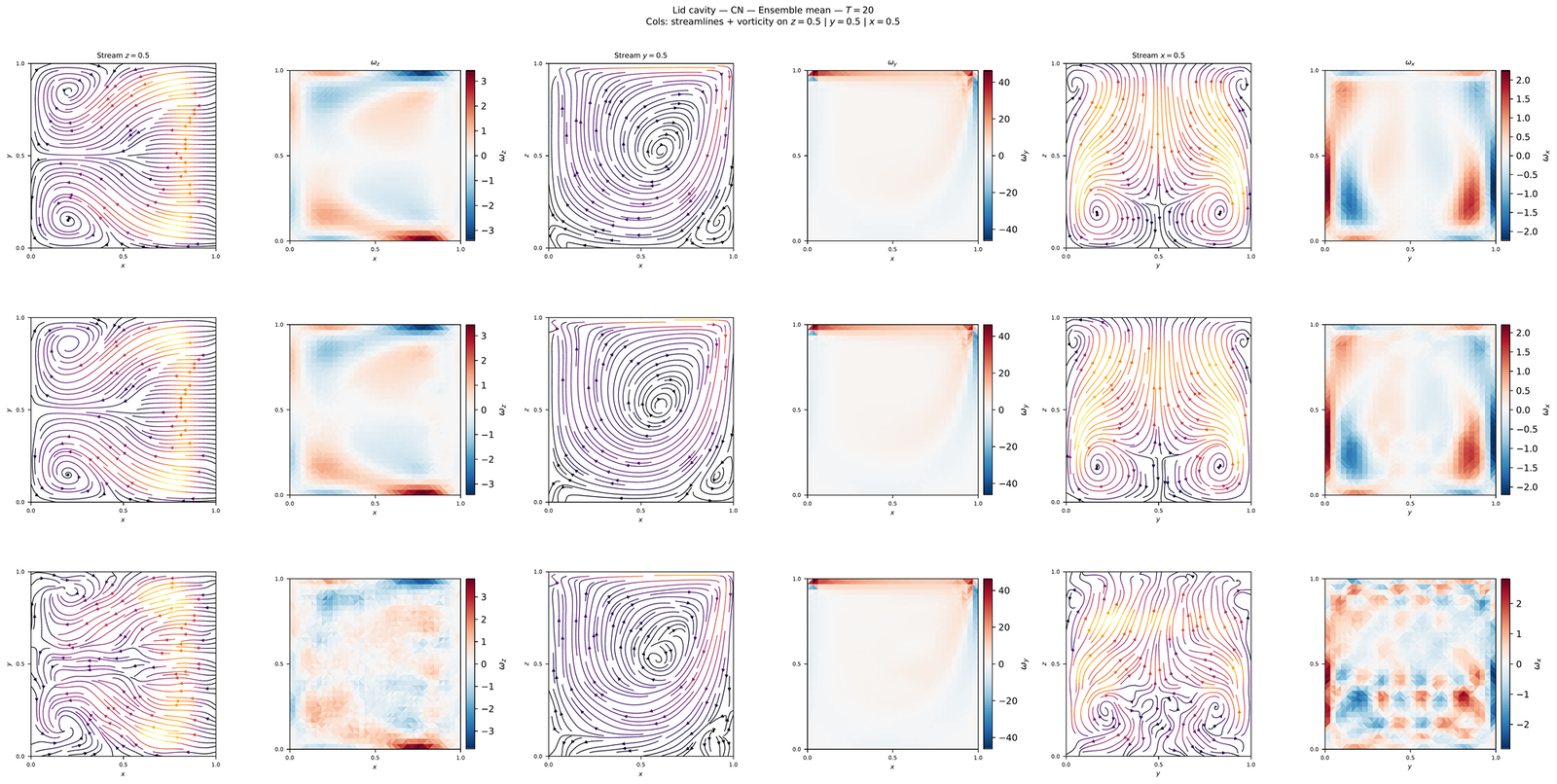}
    \caption{Lid-driven cavity in 3D: Streamlines and vorticity of solution at $T=20$ for Crank-Nicolson scheme}\label{fig:lidcavity_CN_mean_T020_slices}
\end{figure}
\subsection{The 2D problem}
For the two-dimensional stochastic Navier--Stokes equations there exists also a growing literature 
on the numerical approximation. For a (semi)-implicit Euler discretization in time it is shown in \cite{BrDo} and \cite{CP}
that for any $\xi>0$, and $m\in\{1,\dots,M\}$
\begin{align}\label{eq:perror}
&\mathbb P\bigg[\|\bfu(t_m)-\bfu_{m}\|_{L^2_x}^2+\sum_{n=1}^m \tau\|\nabla\bfu(t_n)-\nabla\bfu_{n}\|_{L^2_x}^2>\xi\,\tau^{2\alpha}\bigg]\rightarrow0
\end{align}
as $\tau\rightarrow0$, where $\alpha<\frac{1}{2}$. This means we have 
\begin{align}\label{1/2}
\textbf{convergence of order (almost) 1/2 for the Euler scheme}
\end{align}
 for the convergence in probability.\footnote{Under additional structural assumptions on the noise
also convergence in mean square is proved, that is rates for
the expectation of the error in the squared energy norm, see \cite{BeMi1,BeMi2} and, for the optimal rate \cite{FeVo}.} Here $\bfu$ is the solution to \eqref{eq:SNS} and $\bfu_{m}$ the approximation of $\bfu(t_m)$ with discretisation parameter $\tau=T/M$. These papers also contain results on the space-time approximation. But in this work we concentrate on the temporal error. It is more delicate for stochastic problems due to the limited time regularity which the solution inherits from the driving Wiener process.

In view of that the convergence rate 1/2 in \eqref{eq:perror} is optimal for general multiplicative noise. For additive noise, {\em i.e.}, noise which does not depend on the solution more can be said. This is based on the transformation
\begin{align}\label{zwei}
{\bf y}(t) = {\bf u}(t) - \int_0^t \Phi \, {\rm d}W(s) = {\bf u}(t) - \Phi W(t) \qquad (0 \leq t \leq T)\,,
\end{align}
which solves the {\em random} PDE
\begin{align}\nonumber
\partial_t {\bf y} & =  \nu \Delta{\bf y} - P_{\mathrm{HL}}\bigl[ ({\bf y}\cdot \nabla) {\bf y} \bigr]
+ \nu\Delta[\Phi W] - P_{\mathrm{HL}}[{\mathcal L}^W({\bf y})]\,, \\ \label{req:stochastic-NS}
{\rm div}\, {\bf y} & =  0\,, \\ \nonumber
{\bf y}(0) & = {\bf u}_0.
\end{align}
Here $P_{\mathrm{HL}}$ is the Helmholtz-Leray projection onto divergence free functions, and ${\mathcal L}^W :=
\sum_{i=1}^3 {\mathcal L}_i^W$ couples the transform ${\bf y}$ with the Wiener process $W$, with
$${\mathcal L}_1^W({\bf y}) = ({\bf y}\cdot \nabla)[\Phi W] \,, \qquad
{\mathcal L}_2^W({\bf y}) = ([\Phi W]\cdot\nabla) {\bf y}\,, \qquad
{\mathcal L}_3^W  = ([\Phi W]\cdot\nabla )[\Phi W]\, .
 $$
For sufficiently regular data $\Phi$ and $\bfu_0$
the time derivative of $\bfy$ exists as a measurable function. Using this fact
\begin{align}\label{1}
\textbf{convergence of order (almost) 1 for the Euler scheme}
\end{align}
 has been shown in \cite{BrPr2}. It is worth to mention that the difference between $\bfu$ and $\bfy$ is $\Phi W$ which is known. Hence approximating $\bfy$ is equivalent to approximating $\bfu$ and confirmed by numerical experiments in \cite{BrPrWi}. 

Even that can still be improved: Giving sufficient smoothness of the data, one can deduce from \eqref{req:stochastic-NS} that $\partial_t\bfy$ has the same regularity as the Wiener process $W$, {\em i.e.}, is $\alpha$-H\"older continuous in time for any $\alpha<1/2$. Thus $\bfy$ belongs to the class $C^{1,\alpha}_t$. These considerations motivated \cite{BBCP} (building up on ideas from \cite{CP1}), where even a discretization of strong order ${\mathcal O}(\tau^{1.5})$ is constructed. The starting point in \cite{BBCP} is again the random PDE (\ref{req:stochastic-NS}) for the transform ${\bf y}$ via (\ref{zwei}) . The  higher-order implicit scheme in \cite{BBCP} uses  a Crank-Nicolson discretisation for the time differentiable term, which appears in (\ref{req:stochastic-NS}) as $\nu \Delta{\bf y}$ and 
a mesh of order ${\mathcal O}(\tau^2)$ and an Euler-type discretisation
of the H\"older continuous term, which appears in (\ref{req:stochastic-NS}) as  $\Delta[\Phi W]$.
Summarising this,
\begin{align}\label{3/2}
\textbf{convergence of order (almost) 3/2 for the Crank-Nicolson scheme}
\end{align}
has been shown in \cite{BBCP} and confirmed by numerical experiments. 

\subsection{The 3D problem}
All the numerical results mentioned so far consider the two-dimensional problem. As far as the three-dimensional problem is concerned, regularity and uniqueness of solutions is a major open problem as in the deterministic case. Globally in time only the existence of weak martingale solutions is known \cite{FlGa}. Their numerical approximation has been studied in \cite{BCP}, where weak convergence in law of the approximate solution up to a subsequence is shown. Strong pathwise solutions only exist locally in time, where the existence time is a random variable from which only its positivity is known \cite{BeFr,GlZi,Kim,Mi}. Counterparts of \eqref{1/2}, \eqref{1} and \eqref{3/2} for the three-dimensional problem can only be expected locally in time up to the hypothetical blow-up time of the solution. Indeed, such a version of \eqref{1/2} has been verified in \cite{BrAd}. However, results in the spirit of \eqref{1} and \eqref{3/2} are still missing. Also, to the best of our knowledge, no numerical experiments have been performed yet for the three-dimensional problem. This is due to the fact that a Monte-Carlo simulation for \eqref{eq:SNS} is extremely costly. The aim of the present paper is to close these gaps. In particular,
\begin{itemize}
\item[1.] We prove locally in time \eqref{1}, {\em i.e.}, the Euler scheme converges with rate (almost) 1, see Theorem \ref{thm:main}. The proof follows a different strategy from
that of \cite{BrPr2}, where a sufficiently small time-step size \(\tau\) is required.  By adapting the approach of \cite{BBCP}, we avoid this smallness condition.  In addition, the present proof provides a structurally different and simpler error analysis.
\item[2.] We prove locally in time \eqref{3/2}, {\em i.e.}, the Crank-Nicolson scheme converges with rate (almost) 3/2, see Theorem \ref{thm:main-local}. The proof is inspired by the method of \cite{BBCP}.  However, in \cite{BBCP} the error analysis is carried out in a mean-square framework with high-probability set, whereas in the present 3D setting, the solution is available only locally up to a stopping time. Therefore, the argument cannot be transferred directly.  We instead prove the higher-order rate \(3/2\) on the stopped discrete velocity and discrete pressure by combining new pathwise-in-time Brownian quadrature bounds established in Lemmas~\ref{lem:BM-quadrature-sup} and \ref{lem:triple-BM-sup}. We also obtain a convergence rate for the discrete pressure; see Theorem~\ref{thm:pressure-local-3d}.
\item[3.] 
 Although the convergence analysis in theory is confined to periodic boundary condition (i.e. in $\mathbb{T}^3$) exploiting the regularity results, the numerical simulations carried out in Section \ref{sec: simulations} are considerably more general. In particular, the proposed semi-implicit Euler (IE) and Crank--Nicolson (CN) time discretisations are formulated within a mixed finite element framework. Indeed, the numerical experiments presented are performed using inf-sup stable Taylor--Hood $\mathbb{P}_2/\mathbb{P}_1$ finite elements together with homogeneous Dirichlet boundary conditions on the three-dimensional unit cube $(0,1)^3$. Moreover, owing to their variational structure, the proposed algorithms readily admit extensions to more sophisticated flow models, including multiphase, variable-density, and non-Newtonian fluids, without resorting to highly specialized temporal integration techniques.

 We first examine the numerical accuracy of the proposed semi-implicit Euler (section \ref{sec:IE}) and Crank-Nicolson schemes (Section \ref{theo-2}) through academic example, where experimental convergence rates for both the velocity and pressure are computed and validated. Subsequently, we investigate the three-dimensional stochastic lid-driven cavity problem, which serves as a canonical benchmark for assessing the influence of additive stochastic forcing on incompressible flow dynamics. We examine the effects of increasing noise intensity on both individual sample realisations and ensemble-averaged flow behaviour over long integration times.

 To the best of our knowledge, comparable ensemble-averaged visualizations
for the stochastic lid-driven cavity problem in three spatial dimensions
have not previously been reported; existing studies have primarily focused
on two-dimensional configurations or deterministic flows. Indeed, the computation of statistically reliable ensemble averages
as well as representative sample paths for both the IE and CN schemes in three-dimensional are done for the first time in stochastic
setup. The results are visualised through vorticity contours (see, Figures \ref{fig:lidcavity_IE_mean_T020_slices} \& \ref{fig:lidcavity_CN_mean_T020_slices}), volumetric
streamlines (see, Figures \ref{fig:lidcavity_IE_mean_T020_streamlines_mu0_mu40_back_oblique} \& \ref{fig:lidcavity_CN_mean_T020_streamlines_mu0_mu40_back_oblique}), and velocity-magnitude iso-surfaces (see, Figures \ref{fig:lidcavity_IE_mean_T020_isosurface_mu0_mu40_back_oblique} \& \ref{fig:lidcavity_CN_mean_T020_isosurface_mu0_mu40_back_oblique}), collectively providing
new insight into the manner in which additive stochastic forcing modifies
coherent vortical structures and recirculation patterns in three-dimensional incompressible flows.
\begin{figure}[htbp]
\centering
\begin{minipage}[t]{0.48\textwidth}
    \centering
    \includegraphics[width=\linewidth]{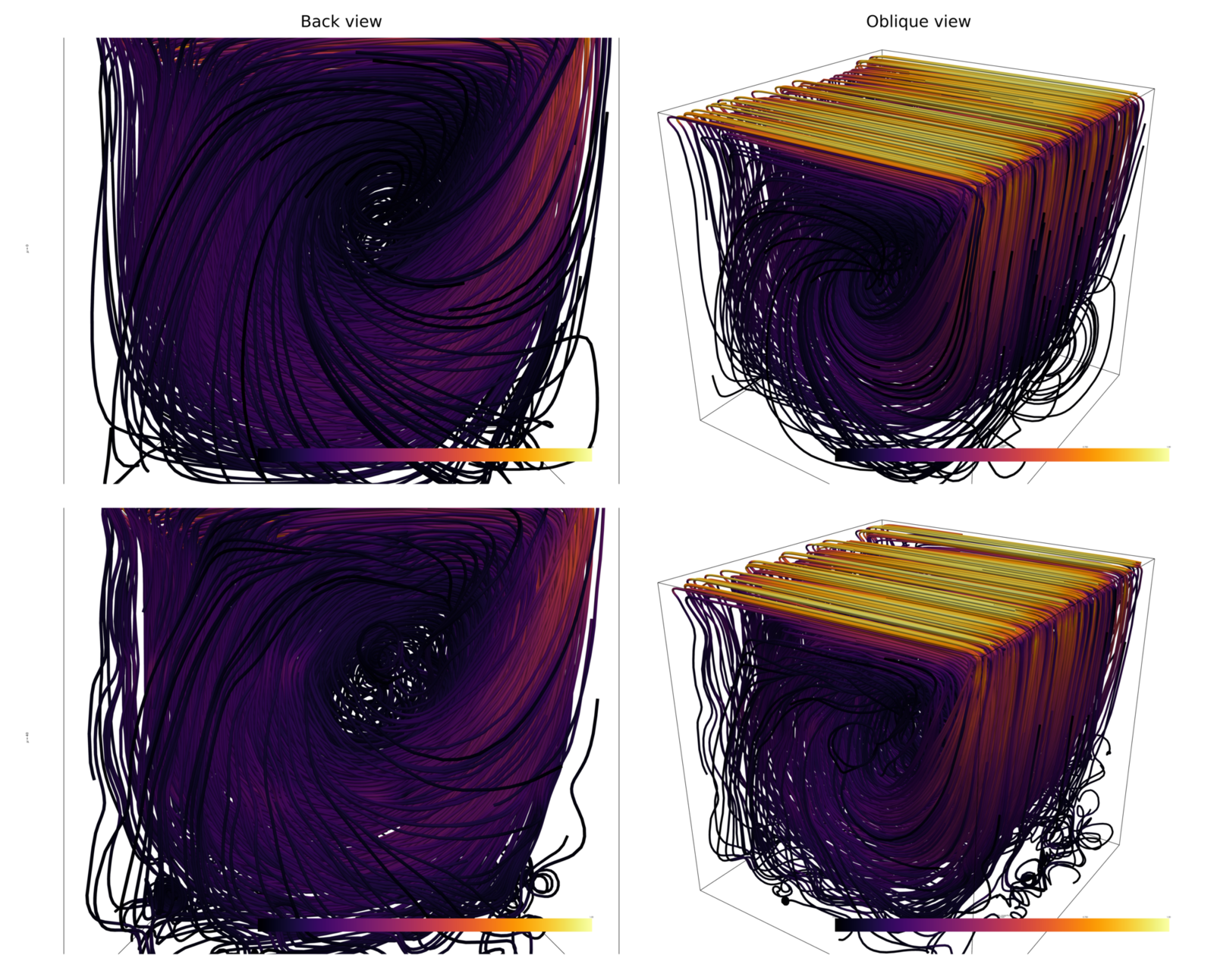}
    \captionsetup{width=1\textwidth}
    \caption{Lid-driven cavity (see Section \ref{sec:Lid-3d}): volumetric
streamlines for deterministic solution $(\mu = 0)$ in row 1 and volumetric
streamlines of the mean of solution with larger noise  $(\mu = 40)$ in row 2 at $T=20$ and $Re=500$ for IE scheme.}
\label{fig:lidcavity_IE_mean_T020_streamlines_mu0_mu40_back_oblique}
\end{minipage}
\hfill
\begin{minipage}[t]{0.48\textwidth}
    \centering
    \includegraphics[width=\linewidth]{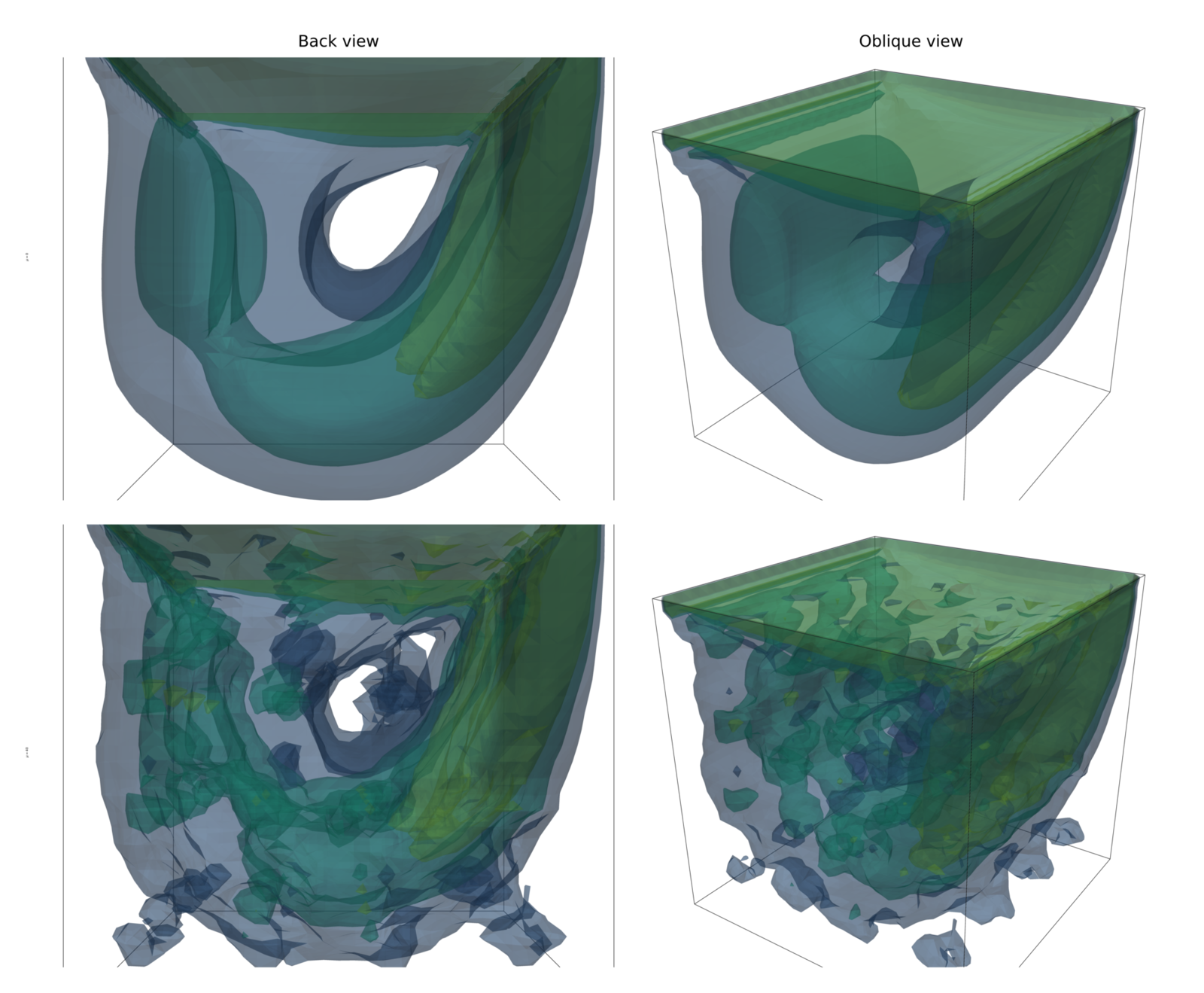}
    \captionsetup{width=1\textwidth}
    \caption{Lid-driven cavity: velocity-magnitude iso-surfaces for IE scheme}
\label{fig:lidcavity_IE_mean_T020_isosurface_mu0_mu40_back_oblique}
\end{minipage}
\end{figure}
\begin{figure}[htbp]
\centering
\begin{minipage}[t]{0.48\textwidth}
    \centering
    \includegraphics[width=\linewidth]{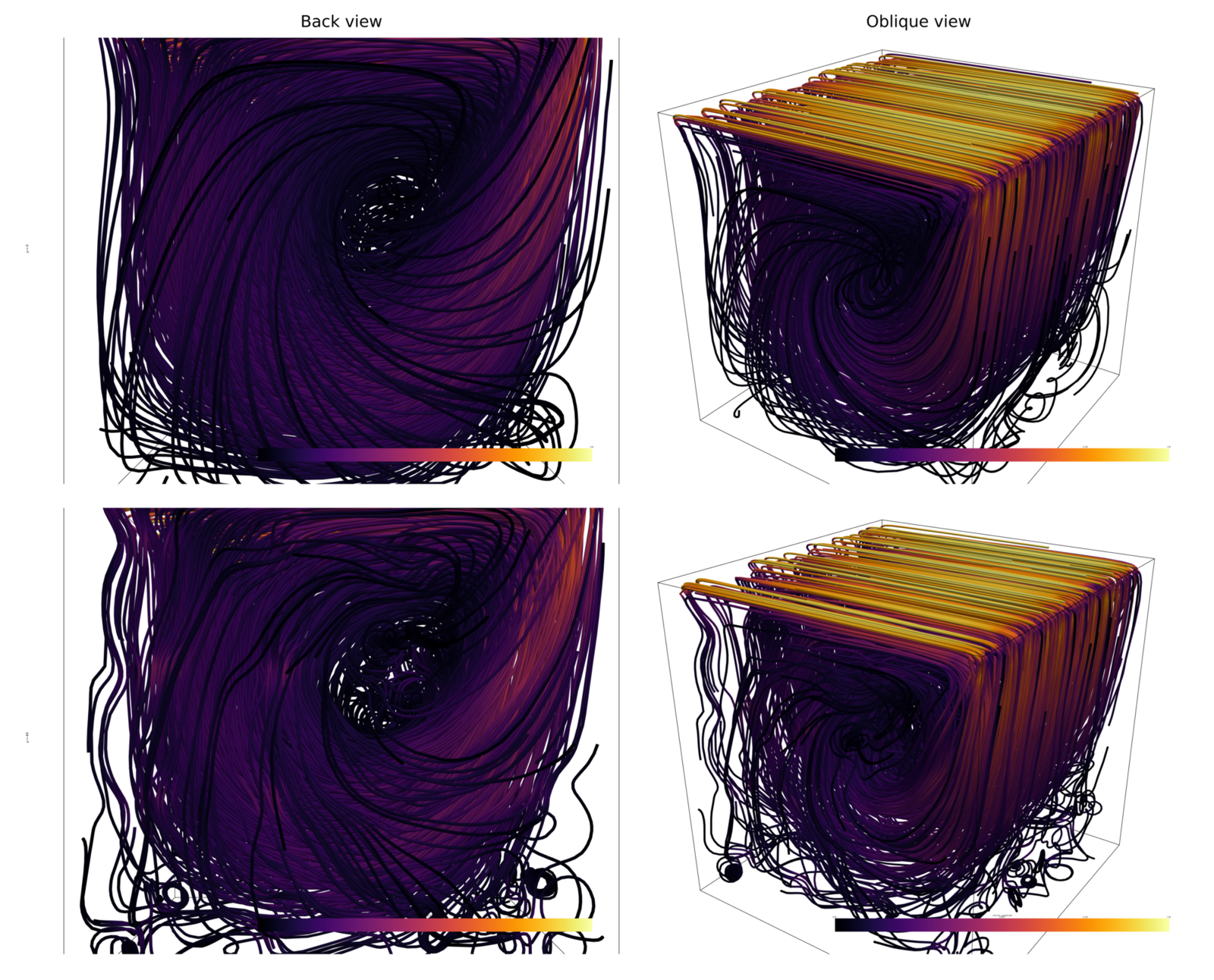}
    \captionsetup{width=1\textwidth}
    \caption{Lid-driven cavity: volumetric
streamlines for CN scheme}
\label{fig:lidcavity_CN_mean_T020_streamlines_mu0_mu40_back_oblique}
\end{minipage}
\hfill
\begin{minipage}[t]{0.48\textwidth}
    \centering
    \includegraphics[width=\linewidth]{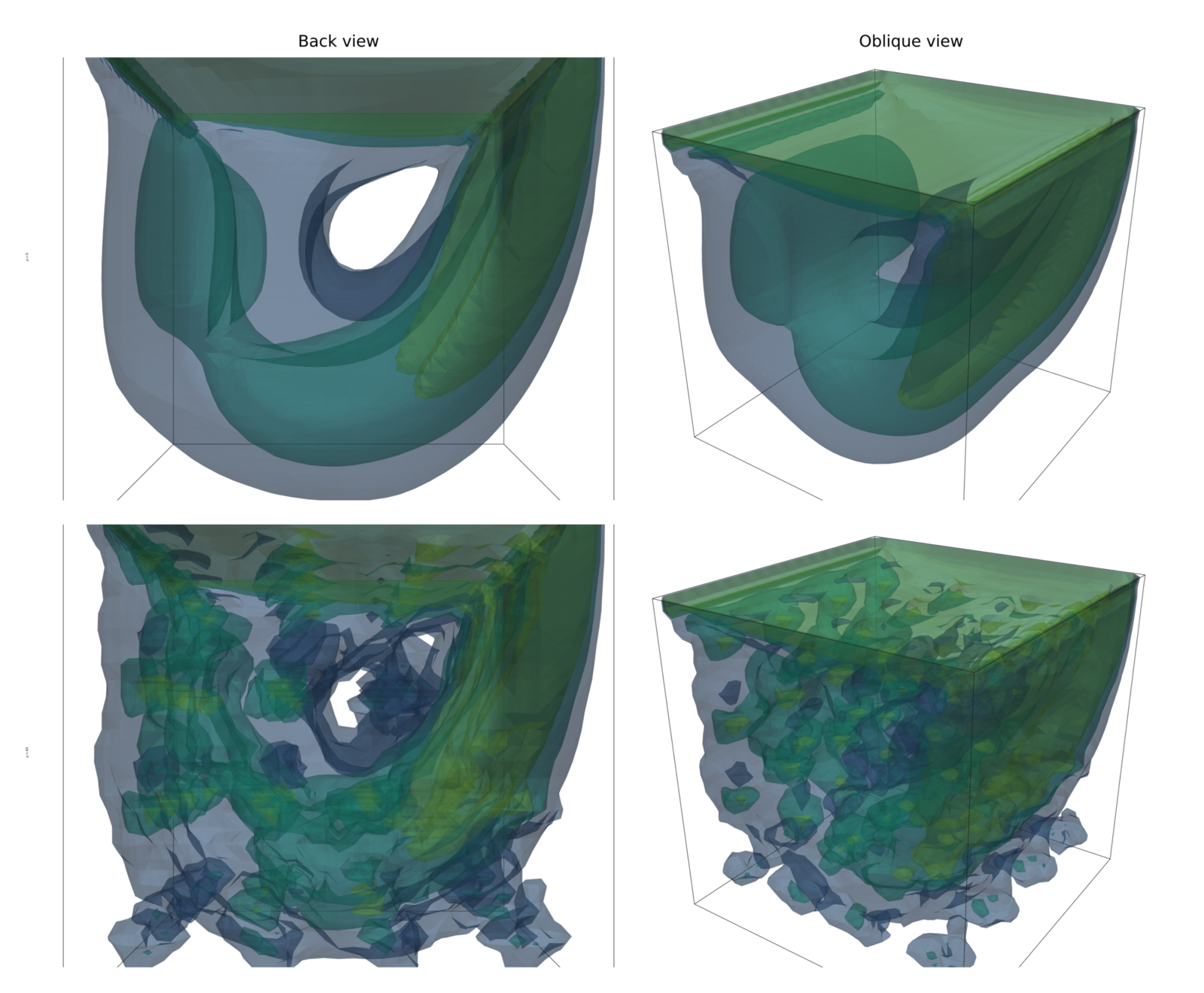}
    \captionsetup{width=1\textwidth}
    \caption{Lid-driven cavity: velocity-magnitude iso-surfaces for CN scheme}
\label{fig:lidcavity_CN_mean_T020_isosurface_mu0_mu40_back_oblique}
\end{minipage}
\end{figure}The associated computational complexity is efficiently managed through
an MPI-based ensemble parallelization strategy in a high performance computer (HPC), whereby each independent
realization is assigned to a dedicated MPI process. Since Monte-Carlo samples are statistically independent, no inter-process
communication is required during time-stepping, yielding an embarrassingly
parallel workload; see Section \ref{sec:CC}.
\end{itemize}
\begin{figure}[htbp]
    \centering
\includegraphics[width=\textwidth]{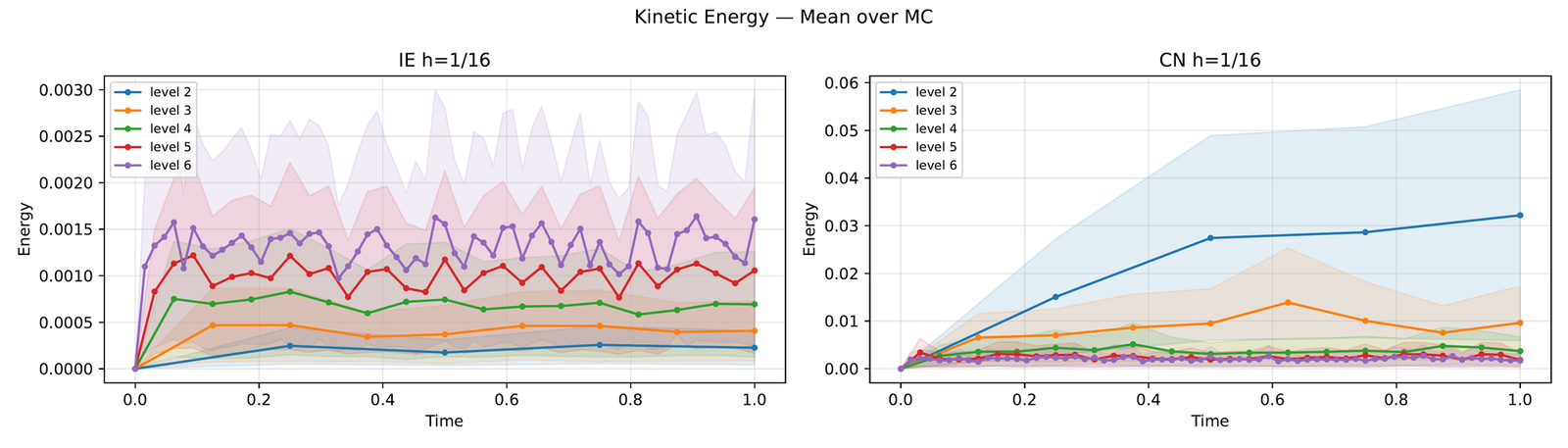}
    \caption{Temporal evolution of the kinetic energy obtained using IE and CN schemes.}\label{fig:kinetic_energy}
\end{figure}
Figures~\ref{fig:kinetic_energy} and~\ref{fig:potential_energy} present the evolution of the mean
kinetic and potential energies obtained using the IE and CN schemes for fixed  spatial discretization $h = 1/16$, while only the temporal step size is 
refined according to $
    \tau_{\ell} = 2^{-\ell}, \quad \ell \in \{2, 3, 4, 5, 6\}$. Here, Level~$\ell$ denotes the $\ell$-th temporal
refinement, obtained by successively halving the time-step size, so that
larger values of $\ell$ correspond to finer temporal meshes.

A noteworthy feature of the experimental findings is the opposite directional
convergence exhibited by the two schemes. For the IE scheme, the
energies increase monotonically with temporal refinement:
the coarsest temporal discretization (Level~2) yields the smallest energy values, while finer
levels (Levels~3--6) produce progressively larger energies. Conversely,
the CN discretisation displays the reverse behaviour, with the coarsest
mesh producing the largest energy values and successive refinements
yielding progressively smaller energies. In both cases, the difference
between consecutive refinement levels decreases as the time step is
refined. The numerical evidence therefore suggests that the IE
approximations approach the limiting energy from below, whereas the CN
approximations approach it from above. We emphasize that this is an
empirical observation based solely on the numerical experiments and is not
established theoretically in the present work.
\begin{figure}[htbp]
\includegraphics[width=\textwidth]{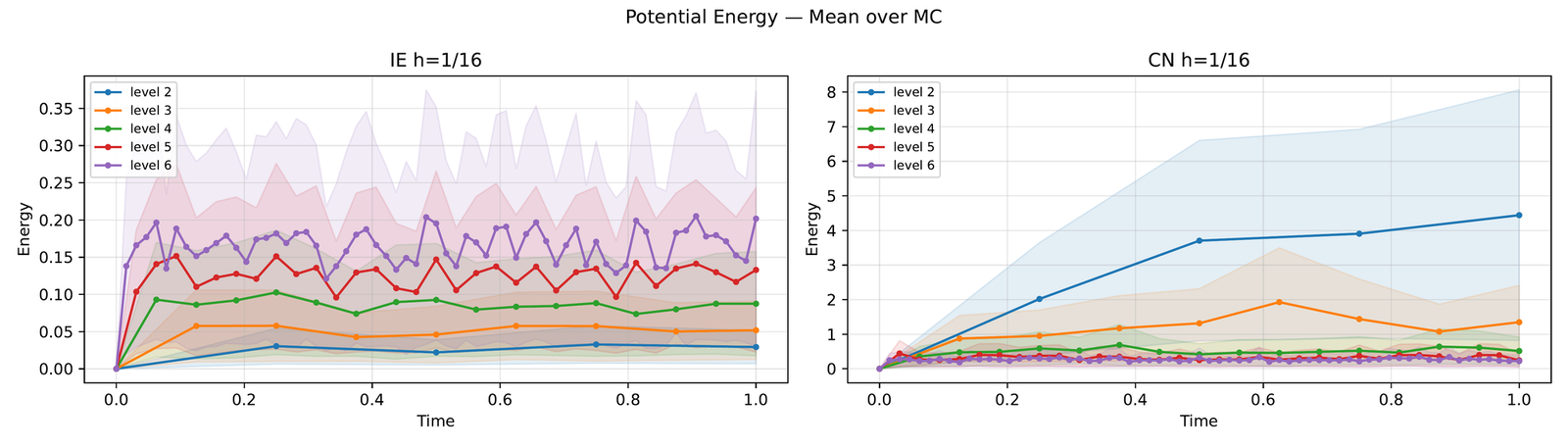}
    \caption{Temporal evolution of the potential energy}\label{fig:potential_energy}
\end{figure}
The observed behaviour reflects the different numerical dissipation inherent in the two temporal discretizations. The IE scheme introduces greater artificial damping, while the CN method exhibits superior energy preservation owing to its second-order, less dissipative nature. As the time step is refined, these discretization effects become progressively weaker, leading both methods to converge towards an identical limiting energy profile. 

The remaining part of the paper is organized as follows: Section \ref{sec:framework} contains the relevant notations, recalls the solution framework for \eqref{eq:SNS} and improved time regularity results. The semi-implicit Euler scheme and the higher order Crank-Nicolson scheme are proposed and the corresponding strong rate of convergence are studied in Sections \ref{sec:IE} and \ref{theo-2}, respectively. Computational studies are assembled in Section \ref{sec: simulations} and the Appendix \ref{appx} contains technical results for the approximation of Brownian integrals.
\section{Mathematical framework}
\label{sec:framework}

\subsection{Function spaces}
In this subsection we present the function space used in the main body of the paper. All function spaces over $\mt$ are considered subject to periodicity of the function and zero mean value.
By $\mathbb L^q:=\mathbb L^q(\mt)$ with $1\leq q\leq\infty$ we denote the standard Lebesgue spaces over $\mt$ of periodic functions.
In particular, we write $(\cdot,\cdot)$ for the $\mathbb L^2({\mathbb{T}^3})$ inner
product. Also, we write
\[
\mathbb H:=\mathbb H(\mt) :=\bigg\{{\bf v}\in\mathbb{L}^2(\mathbb{T}^3) : \int_{\mathbb{T}^2}\mathbf{v}\cdot\nabla q(x)\,{\rm d}x=0\,\,\,\forall q\in\mathbb H^1(\mt)\bigg\}.
\]
Note that we do not distinguish the notation between scalar- and vector-valued function spaces. However, vector-valued function are written in bold face.
By $\mathbb H^k:=\mathbb{H}^k({\mathbb{T}^3})$ for $k\in\N$ we denote Sobolev spaces with differentiability $k$ (and integrability $q=2$). The dual of $\mathbb{H}^k({\mathbb{T}^3})$ is denoted by $\mathbb{H}^{-k}({\mathbb{T}^3})$. We also set $\mathbb W^{k,p}:=\mathbb W^{k,p}(\mt)$
for Sobolev spaces with differentiability $k$ and integrability $q\in[1,\infty]$, such that $\mathbb H^k(\mt)=\mathbb W^{k,2}(\mt)$.  We will also use $\mathbb H^{s,2}:=\mathbb H^{s,2}(\mt)$ for $s \in \R$ to denote the space of distributions $v$ defined on $\mt$ with the finite norm
\begin{equation}
\left\| v \right\|^2_{\mathbb H^{s,2}(\mt)} = \sum_{k \in \mathbb{Z}^3} (1 + |k|^s)^{2} |c_k(v)|^2 < \infty.\notag
\end{equation}
Here, $c_k(v)$ are the Fourier coefficients of $v$ with respect to the standard trigonometric basis $\{ \exp(ik\cdot x) \}_{k \in\mathbb{Z}^3}$.

We use the standard solenoidal velocity space
\begin{align*}
\Vspace:=\mathbb V(\mt) := \{{\bf v}\in \mathbb{H}^1({\mathbb{T}^3}):\ \diver {\bf v}=0\}
\end{align*}
with dual space $\mathbb V^{-1}:=(\mathbb V(\mt))'$.
The transport-form trilinear form is defined by
\[
\mathcal C({\bf a},{\bf b},{\bf c})
:
= \int_{\mathbb{T}^3}({\bf a}(x)\!\cdot\!\nabla){\bf b}(x)\cdot {\bf c}(x)\,\mathrm{d}x,
\qquad {\bf a}\in \Vspace,{\bf b},{\bf c}\in\mathbb{H}^1({\mathbb{T}^3}).
\]
Finally, we consider smooth functions $C^\infty(\mt)=\bigcap_{k\in\N} \mathbb H^k(\mt)$ and smooth solenoidal functions $C^\infty_{\mathrm{div}}:=C^\infty(\mt)\cap\mathbb V$.

For a separable Banach space $(\mathbb X,\|\cdot\|_{\mathbb X})$, we denote by $L^p(I;X)$, the set of (Bochner-) measurable functions $u:I\rightarrow X$ such that the mapping $t\mapsto \|u(t)\|_{\mathbb X}$ belongs to $L^p(I)$. 
The set $C(\overline{I};\mathbb X)$ denotes the space of functions $u:\overline{I}\rightarrow \mathbb X$ which are continuous with respect to the norm topology on $(\mathbb X,\|\cdot\|_{\mathbb X})$. For $\alpha\in(0,1]$ we write
$C^{0,\alpha}(\overline{I};\mathbb X)$ for the space of H\"older-continuous functions with values in $X$. 
Similarly, for a probability space $\mathbb F:=(\Omega,\mathfrak F,\p)$ and a separable Banach space $(\mathbb X,\|\cdot\|_{\mathbb X})$ and $p\in[1,\infty]$
we write $L^p(\Omega,\mathfrak F,\p;\mathbb X)$ or short
$L^p(\Omega;\mathbb X)$ for the set of (Bochner-) measurable functions $v:\Omega\rightarrow \mathbb X$ such that the mapping $\omega\mapsto \|v(\omega)\|_{\mathbb X}$ belongs to $L^p(\Omega,\mathfrak F,\p)$.

\subsection{Probability setup}\label{sec:prob}
Let $(\Omega,\F,(\F_t)_{t\geq0},\prst)$ be a stochastic basis with a complete, right-continuous filtration. The process $W$ is a cylindrical Wiener process, that is, $W(t)=\sum_{k\geq1}W_j(t) \mathfrak e_j$ with $(W_j)_{j\geq1}$ being mutually independent real-valued standard Wiener processes relative to $(\F_t)_{t\geq0}$, and $(\mathfrak e_j)_{j\geq1}$ a complete orthonormal system in a separable Hilbert space $\mathfrak{U}$.
We assume that the diffusion coefficient $\varPhi$ belongs to the set of Hilbert-Schmidt operators $L_2(\mathfrak U;\mathbb X)$, where
$\mathbb X$ can take the role of various Hilbert spaces. Of particular importance are spaces of solenoidal vector fields, such as $\mathbb H(\mt)$, $\mathbb V(\mt)$ or $\mathbb H^{k}(\mt)$.

For a diffusion coefficient $\Phi\in L_2(\mathfrak U;\mathbb X)$ the stochastic integral
 \begin{align*}
\int_0^t \varPhi\,\dd W=\sum_{k\geq 1} \int_0^t\Phi\mathfrak e_k\dd W_k
\end{align*}
is well-defined with values in $\mathbb X$.

\subsection{The concept of solutions}

We give the definition of a strong pathwise solution to \eqref{eq:SNS} which exists up to a stopping time $\mathfrak t$. The velocity field belongs $\p$-a.s. to $C([0,\mathfrak t];\mathbb H^2(\mt))$.

\begin{definition}[Local strong pathwise solution]\label{def:strsol}
Let $(\Omega,\mf,(\mf_t)_{t\geq0},\prst)$ be a given stochastic basis with a complete right-continuous filtration and an $(\mf_t)$-cylindrical Wiener process $W$. Let $\uu_0$ be an $\mf_0$-measurable random variable with values in $\mathbb H\cap\mathbb H^2(\mt)$. The tuple $(\uu,\mathfrak t)$ is called a \emph{local strong pathwise solution} \index{incompressible Navier--Stokes system!weak pathwise solution} to \eqref{eq:SNS} with the initial condition $\uu_0$ provided
\begin{enumerate}
\item $\mathfrak t$ is a $\p$-a.s. strictly positive $(\mathfrak F_t)$-stopping time;
\item the velocity field $\uu$ is $(\mf_t)$-adapted and
$$\uu(\cdot\wedge \mathfrak t) \in C([0,T];\mathbb H^2)\cap L^2(0,T;\mathbb H^3) \quad\text{$\p$-a.s.},$$
\item the momentum equation
\begin{align}\label{eq:mom}
\begin{aligned}
&\int_{\mt}\uu(t\wedge \mathfrak t)\cdot\bfvarrho\dx-\int_{\mt}\uu_0\cdot\bfvarrho\dx
\\&=-\int_0^{t\wedge\mathfrak t}\int_{\mt}(\uu\cdot\nabla)\uu\cdot\bfvarrho\dx\,\dif s+\nu\int_0^{t\wedge\mathfrak t}\int_{\mt}\Delta\uu\cdot\bfvarrho\dx\,\dif s+\int_{\mt}\int_0^{t\wedge\mathfrak t}\Phi \, \dif W  \cdot\bfvarrho\,\dx\notag
\end{aligned}
\end{align}
holds $\p$-a.s. for all $\bfvarrho\in C^{\infty}_{\diver}(\mt)$ and all $t\geq0$.
\end{enumerate}
\end{definition}
We finally define a maximal strong pathwise solution.
\begin{definition}[Maximal strong pathwise solution]\label{def:maxsol}
Fix a stochastic basis with a cylindrical Wiener process and an initial condition as in Definition \ref{def:strsol} a triplet $$(\uu,(\mathfrak{t}_R)_{R\in\N},\mathfrak{t})$$ is a maximal strong pathwise solution to system \eqref{eq:SNS} provided

\begin{enumerate}
\item $\mathfrak{t}$ is a $\p$-a.s. strictly positive $(\mathfrak{F}_t)$-stopping time;
\item $(\mathfrak{t}_R)_{R\in\mn}$ is an increasing sequence of $(\mathfrak{F}_t)$-stopping times such that
$\mathfrak{t}_R<\mathfrak{t}$ on the set $[\mathfrak{t}<\infty]$,
$\lim_{R\to\infty}\mathfrak{t}_R=\mathfrak t$ $\p$-a.s., and
\begin{equation}\label{eq:blowup}
\mathfrak t_R:=\inf \big\{t\in[0,\infty):\,\,\|\uu(t)\|_{\mathbb H^2}\geq R\big\}\quad \text{on}\quad [\mathfrak{t}<\infty] ,\notag
\end{equation}
with the convention that $\mathfrak{t}_R=\infty$ if the set above is empty;
\item each triplet $(\uu,\mathfrak{t}_R)$, $R\in\mn$,  is a local strong pathwise solution in the sense  of Definition \ref{def:strsol}.
\end{enumerate}
\end{definition}
We finally present an existence theorem.
The proof follows the same strategy as in \cite[Proposition 3.2]{Kim}, where the stochastic Navier--Stokes equations are studied on the whole space $\mathbb{R}^3$ under the assumption of fractional regularity $\sigma \in (3/2,2)$. As noted on \cite[p. 2]{Kim}, the case $\sigma = 2$, which is the setting considered here, can be treated straightforwardly.
\begin{theorem}\label{thm:inc2d}
Given that $\varPhi \in L_2(\mathfrak U; \mathbb H\cap \mathbb H^ 2(\mathbb{T}^3))$ and
$\uu_0\in L^2(\Omega,\mathbb H\cap \mathbb H^2(\mt))$. Then there is a
unique maximal global strong pathwise solution to \eqref{eq:SNS} in the sense  of Definition \ref{def:maxsol}.
\end{theorem}

\subsection{Estimates for the continuous solution}
In this section we present the crucial estimates for the
continuous solution, which hold up to the stopping time $\mathfrak t_R$ from Definition \ref{def:maxsol} for some $R\gg1$.
\begin{lemma}\label{lem:reg}
 \begin{enumerate}
Let $(\Omega,\mf,(\mf_t)_{t\geq0},\prst)$ be a given stochastic basis with a complete right-continuous filtration and an $(\mf_t)$-cylindrical Wiener process $W$. 
Let $(\uu,(\mathfrak t_R)_{R\in\N},\mathfrak t)$ be the maximal strong pathwise solution to \eqref{eq:SNS}, cf. Definition \ref{def:maxsol}.
\item[(a)] Assume 
that $\uu_0\in L^r(\Omega,\mathbb H(\mt))$ for some $r\geq2$ and 
that $\Phi$ $\in L_2(\mathfrak{U}; \mathbb H)$. Then we have
\begin{align}
\E\bigg[ \bigg(\sup_{0\leq t\leq T}\int_{\mt}|\uu(t\wedge\mathfrak t_R)|^2\dx+\int_0^{T\wedge t_R}\int_{\mt}|\nabla\uu|^2\dxt\bigg)^{\frac{r}{2}}\bigg]\leq\,c\,\E\Big[1+\|\uu_0\|_{\mathbb H}^r\Big].\notag
\end{align}
\item[(b)] Assume 
that $\uu_0\in L^r(\Omega,\mathbb V)$ for some $r\geq2$ and 
that $\Phi \in L_2(\mathfrak{U}; \mathbb V)$. Then we have
\begin{align}
\begin{aligned}
\E\bigg[\bigg(\sup_{0\leq t\leq T}\int_{\mt}|\nabla\uu(t\wedge \mathfrak t_R)|^2\dx&+\int_0^{T\wedge \mathfrak t_R}\int_{\mt}|\nabla^2\uu|^2\dxt\bigg)^{\frac{r}{2}}\bigg]\\&\leq\,cR^{3r}\,\E\Big[1+\|\uu_0\|_{\mathbb V}^r\Big].\notag
\end{aligned}
\end{align}
\item[(c)] Let $m\in\N$ with $m\geq 2$. Assume that ${\bf u}_0\in L^r(\Omega,\mathbb H^{m}(\mt))\cap  L^{(2m+1)r}(\Omega,\mathbb V(\mt))$ for some $r\geq2$ and and $\Phi\in L_2(\mathfrak U;\mathbb V\cap\mathbb H^{m}(\mt))$. Then we have
\begin{align}
\begin{aligned}
\E\bigg[\sup_{0\leq t\leq T}\|{\bf u}(t\wedge\mathfrak t_R)\|_{\mathbb H^{m}}^2\dx&+\int_0^{T\wedge\mathfrak t_R}\|{\bf u}\|_{\mathbb H^{m+1}}^2\dt\bigg]^{\frac{r}{2}}\\&\leq\,cR^{ \frac{2m+1}{3}r}\,\E\Big[1+\|{\bf u}_0\|_{\mathbb H^m}^2\Big]^{\frac{r}{2}}.\notag
\end{aligned}
\end{align}
Here $c=c(r,T)$ is independent of $R$.
\end{enumerate}
\end{lemma}

\begin{proof}The proof of (a), (b) and (c) with $m=2$ is given in \cite[Lemma 3.1]{BrAd}. The proof for (c) with $m\geq 3$ follows the same strategy similar to \cite[Corollary 2.4.13]{KukShi}.
Formerly,\footnote{The proof can be made rigorous by working with a Galerkin-type approximation and show that the following estimates are uniform with respect to the dimension of the ansatz space.} one applies It\^{o}'s formula to the function $f^{\alpha}(\bfu):=\tfrac{1}{2}\|\partial^\alpha\bfu\|_{\mathbb L^2}^2$ where $\alpha\in \mathbb N_0^3$ is a multi-index of length $m$.
One can estimate the contribution of the convective term for $t\leq\mathfrak t_R$ using the embedding $\mathbb W^{3/4,2}(\mt)\hookrightarrow \mathbb L^4(\mt)$ and interpolation as follows:
\begin{align*}
\bigg|\int_{\mt}\partial^\alpha\mathrm{div}({\bf u}\otimes{\bf u})\cdot\partial^\alpha{\bf u}\dx\bigg|
&\leq\,c\sum_{0\leq |\beta|<m}\|D^{m-|\beta|}\bfu\|_{\mathbb L^{4}}\|D^{|\beta|}\bfu\|_{\mathbb H^{1}}\|D^m\bfu\|_{\mathbb L^{4}}\\
&\leq\,c\sum_{0\leq |\beta|<m}\|\bfu\|_{\mathbb H^{3/4+m-|\beta|}}\|\bfu\|_{\mathbb H^{1+|\beta|}}\|\bfu\|_{\mathbb H^{m+3/4}}\\
&\leq\,c\sum_{1\leq |\beta|<m-1}\|\bfu\|_{\mathbb H^{3/4+m-|\beta|}}\|\bfu\|_{\mathbb H^{1+|\beta|}(\mt)}\|\bfu\|_{\mathbb H^{m+3/4}}
\\&+cR\|\bfu\|_{\mathbb H^{m+3/4}}^2\\
&\leq\,c\sum_{1\leq |\beta|<m-1}\|\bfu\|_{\mathbb H^2}^{1-\frac{3/4+m-|\beta|-2}{m-1}}\|\bfu\|_{\mathbb H^{m+1}}^{\frac{3/4+m-|\beta|-2}{m-1}}\|\bfu\|_{\mathbb H^{2}}^{1-\frac{|\beta|-1}{m-1}}\|\bfu\|_{\mathbb H^{m+1}}^{\frac{|\beta|-1}{m-1}}\\
&\qquad\qquad\times\|\bfu\|_{\mathbb H^{2}(\mt)}^{1-\frac{m+3/4-2}{m-1}}\|\bfu\|_{\mathbb H^{m+1}}^{\frac{m+3/4-2}{m-1}}
+cR\|\bfu\|_{\mathbb H^{m+3/4}}^2\\
&\leq\,c\|\bfu\|_{\mathbb H^{2}}^{1+\frac{3}{2(m-1)}}\|\bfu\|_{\mathbb H^{m+1}}^{\frac{2m+3/2-5}{m-1}}
+cR\|\bfu\|_{\mathbb H^{2}}^{1-\frac{m+3/4-2}{m-1}}\|\bfu\|_{\mathbb H^{m+1}}^{\frac{m+3/4-2}{m-1}}\\
&\leq\,cR^{1+\frac{3}{2(m-1)}}\|\bfu\|_{\mathbb H^{m+1}}^{2-\frac{3/2}{m-1}}
+cR\|\bfu\|_{\mathbb H^{m+1}}\\
&\leq\,\delta \|\bfu\|_{\mathbb H^{m+1}}^{2}+c(\delta)R^{\frac{4m+2}{3}},
\end{align*}

where $\delta>0$ is arbitrary.
\end{proof}

\subsection{The random PDE}
 If $\bfu$ is the maximal strong pathwise solution defined on $(\Omega,\mf,(\mf_t)_{t\geq0},\prst)$ with an $(\mf_t)$-cylindrical Wiener process $W$ (recall Definition \ref{def:maxsol} and Theorem \ref{thm:inc2d}), 
 we may consider the transform
\begin{equation}\label{transform} {\bf y}(t \wedge \mathfrak t_R) = {\bf u}(t \wedge \mathfrak t_R ) - \int_0^{t\wedge \mathfrak t_R}\Phi \, {\rm d}W(s)=\bfu(t \wedge \mathfrak t_R)-\Phi W(t\wedge \mathfrak t_R) \quad \forall\, t \geq 0.
\end{equation}
We denote
\begin{align*}
\mathcal L_1^W(\bfy)&:= 
 \bigl( \nabla[\Phi W] \bigr){\bf y}  ,\quad
\mathcal L_2^W(\bfy):= 
\bigl(\nabla {\bf y}\bigr) [\Phi W]  ,\\
\mathcal L_3^W(t) &:=  \bigl(\nabla [\Phi W]\bigr)[\Phi W],\quad \mathcal L^W:=\mathcal L_1^W+\mathcal L_2^W+\mathcal L^W_3 .
 \end{align*}
Then ${\bf y}: \Omega \times {\mathcal O}_T \rightarrow {\mathbb R}^3$ solves the  following random PDE locally up to the stopping time $ \mathfrak t_R$.
\begin{align}\label{eq:SNSy}
\left\{\begin{array}{rc}
\partial_t {\bf y}  =\nu {\Delta} {\bf y}-{\mathcal P} _{\mathrm{HL}}\bigl[( {\bf y}\cdot)\nabla{\bf y}\bigr]  +\nu {\Delta} [\Phi W]
-{\mathcal P} \bigl[\mathcal L^W(\bfy)\bigr]
& \mbox{in $\mathcal O_T$,}\\
\Div {\bf y}=0\qquad\qquad\qquad\qquad\qquad\,\,\,\,& \mbox{in $\mathcal O_T$,}\\
{\bf y}(0)=\bfu_0\,\qquad\qquad\qquad\qquad\qquad&\mbox{ \,in $\mt$.}\end{array}\right.
\end{align}
Note that for $\mathbb P$-a.a. $\omega\in\Omega$
the triplet $\big(\bfy(\omega,\cdot), \mathfrak t_R, t\big)$ is a solution to the Navier--Stokes equations with right-hand side
\begin{align*}
\bff:=\nu {\Delta} [\Phi W]
-{\mathcal P} \bigl[\mathcal L^W(\bfy)\bigr].
\end{align*} 
Standard regularity results apply provided $\Phi$ is sufficiently regular.
In particular, $\partial_t {\bf y}(\cdot \wedge \mathfrak t_R) \in L^2\bigl( 0,T; \mathbb H(\mt) \bigr)$ holds ${\mathbb P}$-a.s.
\begin{lemma}\label{lem:regadditive}
Let $\bfu_0\in L^r(\Omega;\mathbb H(\mt))$ for some $r>2$, $\Phi\in L_2(\mathfrak U;\mathbb V(\mt))$ and $\bfu$ be the local strong pathwise solution to \eqref{eq:SNS}.
\begin{enumerate} 
\item[(a)] Additionally assume that $\bfu_0\in W^{1,2}(\mathbb V)$ $\mathbb P$-a.s.~and $\Phi\in L_2(\mathfrak U;\mathbb H^2(\mt))$. Then $\partial_t {\bf y} (\cdot \wedge \mathfrak t_R) \in L^2\bigl( 0,T; \mathbb H(\mt) \bigr)$  $\mathbb P$-a.s.~and for a.a. $t\in(0,T)$
\begin{align}
\begin{aligned}
\int_{\mt}|\partial_t\bfy (\cdot \wedge \mathfrak t_R)|^2\dx&\leq\,c\,\Big[\|\Phi W\|_{\mathbb H^2}^2+ \|\bfu(\cdot \wedge \mathfrak t_R)\|_{\mathbb H^2}^2\Big]\\
&+c\,\Big[\|\Phi W\|_{\mathbb V}^2+\|\bfu (\cdot \wedge \mathfrak t_R)\|_{\mathbb V}^2\Big].\notag
\end{aligned}
\end{align}
\item[(b)] Assume additionally $\bfu_0\in \mathbb H^{2}(\mt)$ $\mathbb P$-a.s.~and $\Phi\in L_2(\mathfrak U;\mathbb H^3(\mt))$. Then $\partial_t {\bf y} (\cdot \wedge \mathfrak t_R) \in L^2\bigl( 0,T;\mathbb V(\mt) \bigr)$  $\mathbb P$-a.s.~and
\begin{align}
&\int_0^{T \wedge \mathfrak t_R}\int_{\mt}|\partial_t\nabla\bfy (\cdot \wedge \mathfrak t_R)|^2\dxt \notag \\ &\leq\,c\,\int_0^T\Big[\|\Phi W\|_{\mathbb H^3}^2+\|\bfu(\cdot \wedge \mathfrak t_R)\|_{\mathbb H^3}^2+\|\Phi W\|_{\mathbb H^2}^2+\|\bfu(\cdot \wedge \mathfrak t_R)\|_{\mathbb H^2}^2\Big]\dt.\notag
\end{align}
\item[(c)] Assume additionally ${\bf u}_0\in L^r(\Omega,\mathbb H^4(\mt))$ for some $r\geq2$~and $\Phi\in L_2(\mathfrak U;\mathbb H^{4}(\mt))$. Then we have
\begin{align}
\partial_t{\bf y}(\cdot\wedge\mathfrak t_R)\in C^{1/2}([0,T];L^{r/2}(\Omega);\mathbb V).\notag
\end{align}
\end{enumerate}
\end{lemma}
\begin{proof}
In view of the transformation \eqref{transform} and estimations in Lemma \ref{lem:reg} in hand, the proof of (a) and (b) follows from \eqref{eq:SNSy} by estimating upto a stopping time, in the right-hand side in $\mathbb H$ and $\mathbb V$, respectively.

 The proof of (c) can be obtained exactly as for the 2D problem in \cite[Lemma 2.3]{BBCP}. It is based on the embedding $\mathbb H^2(\mt)\hookrightarrow\mathbb W^{1,4}(\mt)$ which also holds in three dimensions.
\end{proof}
In view of Lemmas \ref{lem:reg} and \ref{lem:regadditive} we obtain the following.
\begin{corollary}\label{cor:add}
Assume that $\bfu_0\in L^r(\Omega,\mathbb H^2(\mt))\cap L^{5r}(\Omega,\mathbb V(\mt))$ for some $r\geq2$ and that $\Phi\in L_2(\mathfrak U;\mathbb H^{3}(\mt))\cap L_2(\mathfrak U;\mathbb V(\mt))$. Let $\bfu$ be the local strong pathwise solution to \eqref{eq:SNS}.
Then we have for any $R>0$
\begin{align*}
\E\bigg[\Bigl(\sup_{0\leq t\leq T}\int_{\mt}|\partial_t\bfy(t\wedge\mathfrak t_R)|^2\dx\Bigr)^{\frac{r}{2}}\bigg]&\leq c(T,\Phi,\bfu_0)R^{6r},\\\E\bigg[\Bigl(\int_0^{T\wedge\mathfrak t_R}\int_{\mt}|\partial_t\nabla\bfy|^2\dxt\Bigr)^{\frac{r}{2}}\bigg]&\leq\,c(T,\Phi,\bfu_0)R^{6r},\\ \E\bigg[\Bigl(\sup_{0\leq t\leq T}\int_{\mt}|\nabla\bfy(t\wedge\mathfrak t_R)|^2\dx\Bigr)^{\frac{r}{2}}\bigg]&\leq c(T,\Phi,\bfu_0)R^{3r}.
\end{align*}
\end{corollary}

\section{The semi-implicit Euler scheme}\label{sec:IE}
Let us consider the temporal discretization of \eqref{eq:SNS} on an equidistance partition of $[0,T]$ with mesh size $\tau = T/M$ and set $t_m = m\tau$, for $ 1 \le m \le M$. Let $u_0$ be an $\mathfrak{F}_0$- measurable random variable with values in $\mathbb H(\mt)$. Our goal is to iteratively construct a  sequence of $\mathfrak{F}_{t_m}$-measurable random variables $\bfu_m $ $\in \mathbb V(\mt)$ such that it holds $\mathbb{P}$-a.s. 
\begin{align}\label{eq:temporal-discrete}
    & \int_{\mt} \uu_m \cdot \bfvarrho\, \dif x - \tau\int_{\mt} \uu_m \otimes \uu_{m-1}: \nabla \bfvarrho \, \dif x \\ & \notag \vspace{2cm}= -\nu\tau\int_{\mt}\nabla \uu : \nabla \bfvarrho\,\dif x + \int_{\mt} \uu_{m-1}\cdot \bfvarrho \, \dif x + \int_{\mt} \varPhi \Delta_m W \cdot \bfvarrho \,\dif x
\end{align}
for every $\bfvarrho \in \mathbb V(\mt)$,
where $\Delta_m W = W(t_m) - W(t_{m-1})$. For given $\uu_{m-1}$ and $\Delta_m W$, the existence of unique $\uu_m$ to \eqref{eq:temporal-discrete} is a routine to check thanks to the linearity in $\uu_m$.
Thus, \eqref{eq:temporal-discrete} can be re-written as
\begin{align}
    & \uu_m + \tau\mathcal{P}_{\mathrm{HL}}(\uu_m\cdot\nabla)\uu_{m-1} = \nu\tau\Delta\uu_m + \uu_{m-1} + \varPhi\Delta_m W \qquad (1 \le m \le M)\,,\notag
\end{align}
 which holds $\mathbb{P}$-a.s. in $\mathbb H(\mt)$. Here, $\mathcal{P}_{\mathrm{HL}} : L^2(\mt) \rightarrow \mathbb H(\mt)$ denotes the Helmholtz-Leray projection.
\subsection{Estimates for time-discrete solution}
Now, we recall the following result concerning the estimation of $\uu_m$, whose proof follows from \cite[Lemma 3.1]{BCP}. Although the proof in \cite{BCP} is given for a bounded domain with zero boundary data, the same arguments apply to the space-periodic case. Moreover, we employ a semi-implicit scheme; this does not affect the proof, since the convective term still cancels when tested against $\uu_m$. Therefore, we state the lemma without proof.
\begin{lemma}\label{lem:regularity-discrte-1}Let $u_0 \in L^{2^p}(\Omega,\mathbb H(\mt))$ and $\varPhi \in L_2(\mathfrak{U}; \mathbb H(\mt))$ for $p \in \N$. Then the iterates $(\uu_m)_{m=1}^M$ given by \eqref{eq:temporal-discrete} satisfy the following, uniformly in M;
\begin{align}
    &\E\bigg[\underset{1 \le m \le M}{\max}\|\uu_m\|_{\mathbb L^2}^{2^p} + \tau \sum_{m=1}^M \|\uu_m\|_{\mathbb L^2}^{2^p-2}\|\nabla\uu _m\|_{\mathbb L^2}^2\bigg] \le C,\notag
\end{align}
where $C= C(p,T, \varPhi, \uu_0) > 0.$
\end{lemma}
In order to obtain ``local-in-time'' estimates we consider a truncated variant similar to \cite{Mi}. For $R>1$ and $\zeta\in C_c^\infty([0,2))$ with $0\leq \zeta\leq 1$ and $\zeta=1$ in $[0,1]$, we set $\zeta_R:=\zeta(R^{-1}\cdot)$ . 

Let $\uu_m^R\in \mathbb V$ be such that
 for every $\varrho\in \mathbb V(\mt)$, it holds true $\p$-a.s. that
\begin{align}\label{tdiscrR}
\begin{aligned}
\int_{\mt}\uu_{m}^R\cdot\bfvarrho \dx &-\tau\int_{\mt}\zeta_R(\|\uu_{m-1}^R\|_{\mathbb H^2})\uu^R_{m}\otimes\uu_{m-1}^R :\nabla\bfvarrho\dx\\
&=-\nu\tau\int_{\mt}\nabla\uu_{m}^R:\nabla\bfvarrho\dx+\int_{\mt}\uu_{m-1}^R\cdot\bfvarrho \dx\\&+\int_{\mt}\zeta_R(\|\uu_{m-1}^R\|_{\mathbb H^{2}}\Phi\,\Delta_mW\cdot \bfvarrho\dx
\end{aligned}
\end{align}
with initial data $\uu_0^R=\uu_0$. Using arguments analogous to those in Lemma \ref{lem:regularity-discrte-1},
and observing that the convective term in \eqref{tdiscrR} disappears when
tested against $\uu_m^R$, we deduce that for any $p \in \mathbb{N}$
\begin{align}
\E\bigg[\max_{1\leq m\leq M}\|\uu_m^R\|^{2^p}_{\mathbb L^{2}}+\tau\sum_{m=1}^M\|\uu_m^R\|_{\mathbb L^{2}}^{2^p-2}\|\nabla\uu^R_m\|^2_{\mathbb L^2}\bigg]&\leq\,c\,,\notag
\end{align}
where $c=c(p,T,\Phi,\uu_0)>0$ is independent of $R$. We begin by recalling $R$--dependent estimates for $(\uu_m^R)_{m=1}^M$ from \cite{BrAd}.
Subsequently, these bounds will be carried over to $(\uu_m)_{m=1}^M$
through the introduction of a suitable discrete stopping time.
\begin{lemma}\label{lemma:3.1} 
Let us assume that $\uu_0\in L^{2^p}(\Omega,\mathbb V(\mt))$ for some $p\in\N$ and $\varPhi \in  L_2(\mathfrak{U}; \mathbb V(\mt))$. Then the iterates $(\uu_m^R)_{m=1}^M$ given by \eqref{tdiscrR}
satisfy the following estimates uniformly in $M$:
\begin{align}
\label{lem:3.1b}
\begin{aligned}
\E\bigg[\max_{1\leq m\leq M}\|\uu_m^R\|^{2^p}_{\mathbb V}&+\sum_{m=1}^{M}\tau\|\uu_m^R\|_{\mathbb V}^{2^p-2}\|\nabla^2\uu_m^R\|^2_{\mathbb L^2}\bigg]\\&+\E\bigg[\sum_{m=1}^{M}\|\uu_m^R\|_{\mathbb V}^{2^p-2}\|\nabla(\uu_m^R-\uu_{m-1}^R)\|^2_{\mathbb L^2}\bigg]\leq\,ce^{cR^{2}},
\end{aligned}
\end{align}
where $c=c(p,T,\Phi,\uu_0)>0$ is independent of $R$.
\end{lemma}

\begin{lemma}\label{lemma:3.1B} 
Assume that $\uu_0\in L^{2^p}(\Omega,\mathbb V\cap\mathbb H^{2}(\mt))$ for some $q\in\N$ and suppose that  $\Phi$ $\in L_2(\mathfrak{U};  \mathbb V \cap \mathbb H^{2}(\mt))$ hold. Then the iterates $(\uu_m^R)_{m=1}^M$ given by \eqref{tdiscrR}
satisfy the following estimates uniformly in $M$:
\begin{align}
\label{lem:3.1B}
\begin{aligned}
\E\bigg[\max_{1\leq m\leq M}\|\uu_m^R\|^{2^p}_{\mathbb H^2}&+\sum_{m=1}^{M}\tau\|\uu_m^R\|_{\mathbb H^2(\mt)}^{2^p-2}\|\nabla^3\uu_m^R\|^2_{\mathbb L^2}\bigg]\\
&+\sum_{m=1}^{M}\|\uu_m^R\|_{\mathbb H^2}^{2^p-2}\|\nabla^2(\uu_m^R-\uu_{m-1}^R)\|^2_{\mathbb L^2}\bigg]
\leq\,ce^{cR^{2}},
\end{aligned}
\end{align}
where $c=c(p,T,\Phi,\uu_0)>0$ is independent of $R$.
\end{lemma}
For $R>0$, we define the (discrete) $(\mathfrak{F}_{t_{m}})$-stopping time
\begin{align} 
{\mathfrak s}_{R}^{\tt d} &:= \min_{0 \leq m \leq M} \biggl\{ t_m:\ \max_{0\leq n\leq m} \Vert  {\bf u}_{n}\Vert_{\mathbb H^2(\mt)} \geq R\biggr\},\notag
\end{align}
where we set ${\mathfrak s}_{R}^{\tt d}=t_M$ if the set above is empty.
 Note that ${\mathfrak s}_{R}^{\tt d} \in \{t_m\}_{m=0}^M$, 
 with random index 
 ${\mathfrak j}_{R} \in {\mathbb N}_0 \cap [0,M]$, such that ${\mathfrak s}_{R}^{\tt d} = t_{{\mathfrak j}_{R}}$. The crucial point is now to obtain a counterpart of the strict positivity
 of the stopping time from the continuous solution, cf. Definition
 \ref{def:strsol}. This is the content of the next lemma, which states that ${\mathfrak s}_{R}^{\tt d}\geq \tau$ with high probability and the prove follows similar lines of argument as in \cite[Lemma 4.4]{BrAd}.
 \begin{lemma}\label{lem:sgeqtau}
 Suppose that the assumptions from Lemma  \ref{lemma:3.1B}
 (with $q=1$) hold and let $R=R(\tau)$ be chosen such that
 $\tau e^{cR^2}\rightarrow0$ as $\tau\rightarrow0$. Then we have for any $\ell\in\N$
 \begin{align}\label{eq:sgeqtau}
 \lim_{\tau\rightarrow0}\p\big([{\mathfrak s}_{R}^{\tt d}\leq\tau\ell]\big)=0.
 \end{align}
 \end{lemma}
With relation \eqref{eq:sgeqtau} at hand it is now meaningful to transfer the estimates for $(\uu_m^R)_{m=1}^M$ from Lemma  \ref{lem:regularity-discrte-1} and  \ref{lemma:3.1B} to 
$(\uu_m)_{m=1}^M$. Noticing that $\uu_m=\uu_m^R$
in $[{\mathfrak s}_{R}^{\tt d}\geq t_m]$ we obtain the following corollary.
\begin{corollary}\label{cor:3.1} 
Assume that $q\in\N$. Then the iterates $(\uu_m)_{m=1}^M$ given by \eqref{eq:temporal-discrete}
satisfy the following estimates uniformly in $M$:
\begin{enumerate}
\item Suppose that $\uu_0\in L^{2^q}(\Omega,\mathbb V(\mt))$ and that additionally $\varPhi \in  L_2(\mathfrak{U}; \mathbb V(\mt))$. Then we have
\begin{align*}
\E\bigg[\max_{1\leq m\leq {\mathfrak j}_{R} }\|\uu_m\|^{2^q}_{\mathbb V}&+\sum_{m=1}^{{\mathfrak j}_{R} }\tau\|\uu_m\|_{\mathbb V}^{2^q-2}\|\nabla^2\uu_m\|^2_{\mathbb L^2}\bigg]\\
&+\sum_{m=1}^{\mathfrak j_R}\|\uu_m^R\|_{\mathbb H^2}^{2^q-2}\|\nabla(\uu_m-\uu_{m-1})\|^2_{\mathbb L^2}\bigg]\leq\,ce^{cR^{2}}.
\end{align*}
\item Suppose that $\uu_0\in L^{2^q}(\Omega,\mathbb V\cap\mathbb H^2(\mt))$  and that additionally $L_2(\mathfrak{U}; \mathbb V\cap\mathbb H^2(\mt))$ hold. Then we have
\begin{align*}
\E\bigg[\max_{1\leq m\leq {\mathfrak j}_{R} }\|\uu_m\|^{2^q}_{\mathbb H^2}&+\sum_{m=1}^{{\mathfrak j}_{R} }\tau\|\uu_m\|_{\mathbb H^2(\mt)}^{2^q-2}\|\nabla^3\uu_m\|^2_{\mathbb L^2}\bigg]\\
&+\sum_{m=1}^{\mathfrak j_R}\|\uu_m^R\|_{\mathbb H^2}^{2^q-2}\|\nabla^2(\uu_m-\uu_{m-1})\|^2_{\mathbb L^2}\bigg]\leq\,ce^{cR^{2}}.
\end{align*}
\end{enumerate}
Here $c=c(q,T,\Phi,\uu_0)>0$ is independent of $R$.
\end{corollary}

\subsection{A stable time-discretization}
\label{sec:3.1}
Owing to the temporal regularity of $\bfy$ as in Lemma~\ref{lem:regadditive}, we anticipate (up to) first-order convergence in probability for the semi-implicit time discretization of \eqref{eq:SNS}. To this end, consider a uniform partition of the interval $[0,T]$ with step size $\tau = T/M$ and nodes $t_m = m\tau$. Let $\bfu_0$ be an $\mathfrak F_0$-measurable random variable taking values in $\mathbb V(\mt)$, and assume that $\Phi \in L_2(\mathfrak U; \mathbb V(\mt))$. Starting from $\bfy_0 = \bfu_0$, we construct a sequence of $\mf_{t_m}$-measurable random variables $\bfy_m$ ($1 \leq m \leq M$) upto a stopping time $\mathfrak t_R$, with values in $\mathbb V(\mt)$ such that
\begin{align}\label{tdiscr-add-rephrased}
\begin{aligned}
\frac{\bfy_m - \bfy_{m-1}}{\tau} - \nu \Delta \bfy_m 
&= \nu \Delta[\Phi W(t_m)] 
- \mathcal P_{\mathrm{HL}}\bigl[(\bfy_{m-1}\cdot\nabla )\bfy_m\bigr] 
+ \mathcal P_{\mathrm{HL}}\bigl[\mathcal L^m(\bfy_{m-1}, \bfy_m)\bigr], \\
\text{where}\quad 
\mathcal L^m(\bfy_{m-1}, \bfy_m)
&= \mathcal L_1^{W(t_m)}(\bfy_{m-1}) 
+ \mathcal L_2^{W(t_{m-1})}(\bfy_m) 
+ \mathcal L_1^{W(t_m)}\bigl(\Phi W(t_{m-1})\bigr).
\end{aligned}
\end{align}
For $\mathbb P$-almost every $\omega \in \Omega$, the formulation \eqref{tdiscr-add-rephrased} can be interpreted as a stationary Navier--Stokes problem with forcing term $\nu \Delta[\Phi W(t_m)]$ and perturbed by $\mathcal P_{\mathrm{HL}}\bigl[\mathcal L^m(\bfy_{m-1}, \bfy_m)\bigr]$. The mapping depends continuously on $\bfy_{m-1}$, $W(t_{m-1})$, and $W(t_m)$, which ensures the required measurability. If, in addition, $\Phi \in L_2(\mathfrak U; \mathbb H^2(\mt))$, then enhanced regularity follows and $\bfy_m \in \mathbb H^2(\mt)$ holds $\mathbb P$-almost surely. An analogous statement remains valid when $\mathbb H^3(\mt)$ is considered instead of $\mathbb H^2(\mt)$.
Setting 
\begin{equation}\label{identd1}
{\bf u}_m := {\bf y}_m  +\Phi W(t_m) \qquad (1 \leq m \leq M)\,,
\end{equation}
accordingly gives
\begin{equation}\label{help-1}({\bf u}_m - {\bf u}_{m-1}\bigr) - \nu \tau {\Delta}{\bf u}_m + \tau {\mathcal P}_{\mathrm{HL}}(\nabla{\bf u}_m ){\bf u}_{m-1} = \Phi\Delta_m W \qquad (1 \leq m \leq M)
\end{equation}
and ${\bf u}_0 = {\bf y}_0$. Thus, approximating $\bfy$ by \eqref{tdiscr-add-rephrased} and approximating $\bfu$ by \eqref{help-1} is equivalent.
To proceed further, we define
\begin{eqnarray}\label{eq:tRd-1} 
\mathfrak t_{R}^{\tt d} &:=& \min_{0 \leq m \leq M} \biggl\{ t_m\leq \mathfrak t_{R}:\ \sup_{t\in[0,t_m]}\|\Phi (W_{t})\|_{\mathbb H^2} \geq R\biggr\}\, \wedge {\mathfrak s}_{R}^{\tt d} ,
\end{eqnarray}
where the minimum of the empty set is defined as $t_M$ and $R$ is chosen in accordance with \eqref{eq:2008} below. To this end, $\mathfrak m_{R}$ denotes the unique index in $\{1,\dots,M\}$ with
$t_{\mathfrak m_{R}}=\mathfrak t_{R}^{\tt d}$. 
Since $\Phi\in L_2(\mathfrak U;\mathbb H^3(\mt))$ we clearly have
$$\E\bigg[\sup_{0\leq t\leq t_M}\|\Phi W(t)\|_{\mathbb H^2}\bigg]\leq\,C.$$
Hence one can control the size of \{$\mathfrak t_{R}^{\tt d} < T\} $ by
 \begin{align}\label{eq:2008} \begin{aligned}{\mathbb P}\bigl[\{\mathfrak t_{R}^{\tt d} < T\}\bigr] &\leq {\mathbb P} \bigl[ \{ \mathfrak s_{R}^{\tt d} < T\}\big]+\mathbb P\bigg[\sup_{0\leq t\leq t_M}\|\Phi W(t)\|_{\mathbb H^{2}}\geq R\bigg] \rightarrow 0,
 \end{aligned}
 \end{align}
provided we choose $R$ such that  $R=R(\tau)$ with
$\tau e^{cR^2}\rightarrow0$ as $\tau\rightarrow0$ (as in Lemma  \ref{lem:sgeqtau}). Note that, thanks to the transformation \eqref{identd1}, $\mathbb{P}$-a.s., 
$$\underset{m\le \mathfrak m_R}{\sup} \|\bfy_m\|_{\mathbb H^2} \le \underset{m\le \mathfrak m_R}{\sup}\|\bfu_m\|_{\mathbb H^2} + \underset{0 \leq t \le t_M}{\sup}\|\Phi W(t)\|_{\mathbb H^2} \le R\,.$$

.

\subsection{Temporal Error analysis}
For the purpose of the error analysis
we define for $R>0$ and $m\in\{1,\dots M\}$
\begin{align}\label{eq:tmR}
\mathfrak t_{m}^R:=t_m\wedge \mathfrak t_{R}^{\tt d},\quad \tau_{n}^{R}&=\mathfrak t_{n}^{R}-\mathfrak t_{n-1}^{R}.
\end{align} 
We consider the stopped process
\begin{align}\label{tdiscrtilde}
\bfy_m^{R}:=\begin{cases} \bfy_m,\quad \text{in}\quad \{t_m={\mathfrak t}_m^{R}\},\\
\bfy_{{\mathfrak m}_{R}},\quad \text{in}\quad \{t_m> {\mathfrak{t}}_m^{R}\},
\end{cases}
\end{align}
where $\mathfrak m_{R}$ is defined after \eqref{eq:tRd-1}. The main effort of this section is to prove the following error estimate.
\begin{theorem}\label{thm:main}
Assume that $\bfu_0\in L^r(\Omega,\mathbb H^3(\mt))\cap L^{5r}(\Omega,\mathbb V(\mt))$ for some $r\geq8$ and that $\Phi\in L_2(\mathfrak U;\mathbb V\cap \mathbb H^3(\mt))$. Let $\bfy$ be the solution
to \eqref{eq:SNSy}, and $(\bfy_m)_{m=1}^M$ be the solution to \eqref{tdiscr-add-rephrased}. Then we have the error estimate
\begin{align}\label{eq:thm:4}
\begin{aligned}
\max_{1\leq m\leq M}\E\bigg[\|\bfy({\mathfrak{t}}^R_m)-\bfy^{R}_{m}\|^2_{\mathbb L^2}&+\sum_{n=1}^m \tau_n^R \|\nabla\bfy(\mathfrak{t}^R_n)-\nabla\bfy^R_{n}\|^2_{\mathbb L^2}\bigg]\leq C_{\varepsilon, R, \Phi, T}\tau^{2-\varepsilon}\,,\notag
\end{aligned}
\end{align}
for every $\varepsilon > 0.$
\end{theorem}
\subsubsection{Consistency residual} We introduce the shorthand notation for all $m\in\{1,\dots,M\}$ and $\quad  \ell \in \{m, m-1\}$
\[ {{\bf Y}}^R_{\ell}
:=
{\bf y}(\mathfrak{t}_\ell^R) + \Phi W(\mathfrak{t}_{\ell}^R)\,,\]
and 
\[\bar{{\bf Y}}_{m-1}^R:=\frac{1}{\tau}\int^{\mathfrak{t}_m^R}_{\mathfrak{t}^R_{m-1}}{\bf y}(s)\,\mathrm{d}s + \frac{1}{\tau}\int^{\mathfrak{t}_m^R}_{\mathfrak{t}^R_{m-1}} \Phi W(s)\,\mathrm{d}s =: \bar{\bf y}_{m-1}^R + \Phi\mathcal{Q}_{m-1, R}^W
\]
For each time level \(m=1,\dots,M\), we define the \emph{consistency residual} \(\mathcal{R}_{m}\in (\mathbb{V}(\mt))'\) by
\begin{align}\label{eq:R-def-1} \langle \mathcal{R}^R_{m},\boldsymbol{\varphi}\rangle &:=\frac{1}{\tau}\!\int_{\mathfrak{t}_{m-1}^R}^{\mathfrak{t}_{m}^R}\! \Big(\mathcal C({\bf {Y}}(t),{\bf {Y}}(t),\boldsymbol{\varphi})-\mathcal C(\bar{{\bf Y}}^R_{m-1},{{\bf Y}}^R_{m},\boldsymbol{\varphi})\Big)\,\mathrm{d}t\notag\\ 
&\quad- \frac{1}{\tau}\!\int_{\mathfrak{t}_{m-1}^R}^{\mathfrak{t}_{m}^R}\! \mathcal C(\bar{\bf {Y}}_{m-1}^R, \Phi( W(t) -W(\mathfrak{t}_{m}^R)),\boldsymbol{\varphi})\,\mathrm{dt} -
\mathcal C\!\bigl({{\bf y}}(\mathfrak{t}_{m-1}^R)-\bar{{\bf y}}^R_{m-1},\, {{\bf Y}}^R_{m},\, \boldsymbol{\varphi}\bigr)\notag\\
&\quad+ \nu\,\frac{1}{\tau}\!\int_{\mathfrak{t}_{m-1}^R}^{\mathfrak{t}_{m}^R}\! \big(\nabla ({\bf {y}}(t)-{\bf y}(\mathfrak{t}_m^R)),\nabla\boldsymbol{\varphi}\big)\,\mathrm{d}t
\end{align} 
for any \(\boldsymbol{\varphi}\in \mathbb{V}(\mt)\).
By using the bilinearity of $\mathcal{C}$, we obtain
\begin{align}\nonumber \langle \mathcal{R}_{m}^R,\boldsymbol{\varphi}\rangle &:=\frac{1}{\tau}\!\int_{\mathfrak{t}_{m-1}^R}^{\mathfrak{t}_{m}^R}\! \Big(\mathcal C({\bf {Y}}(t)-\bar{{\bf Y}}^R_{m-1},{\bf {Y}}(t)-\bar{{\bf Y}}^R_{m-1},\boldsymbol{\varphi})\,{\rm d}t\\&\quad+\underbrace{\frac{1}{\tau}\!\int_{\mathfrak{t}_{m-1}^R}^{\mathfrak{t}_{m}^R}\! \Big(\mathcal C( {\bf Y}(t)-\bar{{\bf Y}}^R_{m-1},\bar{{\bf Y}}^R_{m-1},\boldsymbol{\varphi})\Big)\,\mathrm{d}t}_{=0}\notag\\ &\quad +\frac{1}{\tau}\!\int_{\mathfrak{t}_{m-1}^R}^{\mathfrak{t}_{m}^R}\! \mathcal C\big(\bar{{\bf Y}}^R_{m-1},\,{\bf {Y}}(t)-{{\bf Y}}^R_{m},\,\boldsymbol{\varphi}\big)\,\mathrm{d}t- \frac{1}{\tau}\!\int_{\mathfrak{t}_{m-1}^R}^{\mathfrak{t}_{m}^R}\! \mathcal C(\bar{\bf {Y}}_{m-1}^R, \Phi( W(t) -W(\mathfrak{t}_{m}^R)),\boldsymbol{\varphi})\,\mathrm{dt}\notag \\ &\quad -
\mathcal C\!\bigl({{\bf y}}(\mathfrak{t}_{m-1}^R)-\bar{{\bf y}}^R_{m-1},\, {{\bf Y}}^R_{m},\, \boldsymbol{\varphi}\bigr) + \nu\,\frac{1}{\tau}\!\int_{\mathfrak{t}_{m-1}^R}^{\mathfrak{t}_{m}^R}\! \big(\nabla ({\bf {y}}(t)-{\bf y}^R_{m}),\nabla\boldsymbol{\varphi}\big)\,\mathrm{d}t.\notag \\
&= \frac{1}{\tau}\!\int_{\mathfrak{t}_{m-1}^R}^{\mathfrak{t}_{m}^R}\! \Big(\mathcal C({\bf {Y}}(t)-\bar{{\bf Y}}^R_{m-1},{\bf {Y}}(t)-\bar{{\bf Y}}^R_{m-1},\boldsymbol{\varphi})\,{\rm d}t  +\frac{1}{\tau}\!\int_{\mathfrak{t}_{m-1}^R}^{\mathfrak{t}_{m}^R}\! \mathcal C\big(\bar{{\bf Y}}^R_{m-1},\,{\bf y}(t)-{\bf y}^R_{m},\,\boldsymbol{\varphi}\big)\,\mathrm{d}t\notag \\ &\quad  -
\mathcal C\!\bigl({{\bf y}}(\mathfrak{t}_{m-1}^R)-\bar{{\bf y}}^R_{m-1},\, {{\bf Y}}^R_{m},\, \boldsymbol{\varphi}\bigr) + \nu\,\frac{1}{\tau}\!\int_{\mathfrak{t}_{m-1}^R}^{\mathfrak{t}_{m}^R}\! \big(\nabla ({\bf {y}}(t)-{\bf y}(\mathfrak{t}_m^R),\nabla\boldsymbol{\varphi}\big)\,\mathrm{d}t.
\end{align}
Thanks to the decomposition on $[\mathfrak{t}_{m-1}^R,\mathfrak{t}_{m}^R],$ 
\[ {\bf {Y}}(t)-\bar {{\bf Y}}^R_{m-1} =\big({\bf y}(t)-\bar {{\bf y}}^R_{m-1}\big)+\Phi\big(W(t)-\mathcal Q_{m-1,R}^W\big), \] 
a straightforward expansion yields 
\[ \langle\mathcal R^R_{m},\boldsymbol{\varphi}\rangle =\sum_{i=1}^{7}\langle\mathcal{T}_{i,R,m},\boldsymbol{\varphi}\rangle, \]
where 
\begin{align*} \big\langle\mathcal{T}_{1,R,m},\boldsymbol{\varphi}\big\rangle &:=\frac{1}{\tau}\!\int_{\mathfrak{t}_{m-1}^R}^{\mathfrak{t}_{m}^R}\! \mathcal C\big({\bf y}(t)-\bar{{\bf y}}^R_{m-1},\,{\bf y}(t)-\bar{{\bf y}}^R_{m-1},\,\boldsymbol{\varphi}\big)\,\mathrm{d}t,\\[1mm]
\big\langle\mathcal{T}_{2,R,m},\boldsymbol{\varphi}\big\rangle &:=\frac{1}{\tau}\!\int_{\mathfrak{t}_{m-1}^R}^{\mathfrak{t}_{m}^R}\! \mathcal C\big(\Phi W(t)-\Phi\mathcal{Q}^{W}_{m-1,R},\, \Phi W(t)-\Phi\mathcal{Q}^{W}_{m-1,R},\, \boldsymbol{\varphi}\big)\,\mathrm{d}t,\\[1mm]
\big\langle\mathcal{T}_{3,R,m},\boldsymbol{\varphi}\big\rangle &:=\frac{1}{\tau}\!\int_{\mathfrak{t}_{m-1}^R}^{\mathfrak{t}_{m}^R}\! \mathcal C\big({\bf y}(t)-\bar{{\bf y}}^R_{m-1},\, \Phi W(t)-\Phi\mathcal{Q}^{W}_{m-1,R},\, \boldsymbol{\varphi}\big)\,\mathrm{d}t,\\[1mm]
\big\langle\mathcal{T}_{4,R,m},\boldsymbol{\varphi}\big\rangle &:=\frac{1}{\tau}\!\int_{\mathfrak{t}_{m-1}^R}^{\mathfrak{t}_{m}^R}\! \mathcal C\big(\Phi W(t)-\Phi\mathcal{Q}^{W}_{m-1,R},\, {\bf y}(t)-\bar{{\bf y}}^R_{m-1},\, \boldsymbol{\varphi}\big)\,\mathrm{d}t,\\[1mm]
\big\langle\mathcal{T}_{5,R,m},\boldsymbol{\varphi}\big\rangle &:=\frac{1}{\tau}\!\int_{\mathfrak{t}_{m-1}^R}^{\mathfrak{t}_{m}^R}\! \mathcal C\big(\bar{{\bf Y}}^R_{m-1},\,{\bf y}(t)-{{\bf y}}_{m}^R,\,\boldsymbol{\varphi}\big)\,\mathrm{d}t,\\[1mm]
\big\langle\mathcal{T}_{6,R,m},\boldsymbol{\varphi}\big\rangle &:=- \mathcal C\!\bigl({{\bf y}}(\mathfrak{t}_{m-1}^R)-\bar{{\bf y}}^R_{m-1},\, {{\bf Y}}^R_{m},\, \boldsymbol{\varphi}\bigr),\\[1mm] \big\langle\mathcal{T}_{7,R,m},\boldsymbol{\varphi}\big\rangle &:=\nu\,\frac{1}{\tau}\!\int_{\mathfrak{t}_{m-1}^R}^{\mathfrak{t}_{m}^R}\! \big(\nabla {\bf y}(t)-\nabla{\bf y}(\mathfrak{t}_{m}^R),\nabla\boldsymbol{\varphi}\big)\,\mathrm{d}t. \end{align*}
We start with the following estimate.
\begin{lemma}[Residual bound]\label{lem:R-bound-1}
For each small $\varepsilon>0$, there exists a constant \(C_{\varepsilon,R}>0\), independent of
\(\tau\) and \(N\), such that
\begin{equation}
\mathbb E\left[
\tau\sum_{m=1}^{M}
\left\|
\mathcal R_{m}^R
\right\|_{\mathbb V^{-1}}^{2}
\right]
\le
C_{\varepsilon,R}\tau^{2-\varepsilon}.\notag
\end{equation}
\end{lemma}

\begin{proof}
Recall that
\[
\left\|
\mathcal T_{i,R,m}
\right\|_{\mathbb V^{-1}}
=
\sup_{\substack{
\boldsymbol{\varphi}\in\mathbb V\\
\|\nabla\boldsymbol{\varphi}\|_{\mathbb L^2}\le1}}
\left|
\left<
\mathcal T_{i,R,m},\boldsymbol{\varphi}
\right>
\right|.
\]
We first consider \(\mathcal T_{1,R,m}\). By the continuity estimate
\[
\left|
\left<
\mathcal T_{1,R,m},\boldsymbol{\varphi}
\right>
\right|
\le
\frac{C}{\tau}
\int_{\mathfrak t_{m-1}^R}^{\mathfrak t_{m}^R}
\left\|
{\bf y}(t)-\bar{\bf y}^R_{m-1}
\right\|_{\mathbb V}^{2}
\,{\rm d}t\,
\|\nabla\boldsymbol{\varphi}\|_{\mathbb L^2}.
\]
Thanks to Corollary \ref{cor:add}, 
\({\bf y}(\cdot\wedge\mathfrak t_R)\in
L^4(\Omega;C^1([0,T];\mathbb V))\), we have, for
\(t\in[\mathfrak t_{m-1}^R,\mathfrak t_{m}^R]\), $\mathbb{P}$-a.s.
\[
\left\|
{\bf y}(t)-\bar{\bf y}^R_{m-1}
\right\|_{\mathbb V}
\le
C_R\tau
\|{\bf y}(\cdot\wedge\mathfrak t_R)\|_{C^1([0,T];\mathbb V)}.
\]
Therefore by taking expectation, we obtain 
\begin{equation}
\mathbb E\left[
\sup_{1\le m \le M}
\left\|
\mathcal T_{1,R,m}
\right\|_{\mathbb V^{-1}}^2
\right]
\le
C_R\tau^4
\le
C_R\tau^{3-\varepsilon},\notag
\end{equation}
where we have used the H\"older inequality in the probability variable together with Corollary \ref{cor:add}. For $\mathcal{T}_{2, R, m}$, we again invoke  continuity estimates to get
\[
\left|
\left<
\mathcal T_{2,R,m},\boldsymbol{\varphi}
\right>
\right|
\le
\frac{C}{\tau}
\int_{\mathfrak t_{m-1}^R}^{\mathfrak t_{m}^R}
\left\|\Phi(W(t) -\mathcal{Q}_{m-1,R}^W) \right\|_{\mathbb V}^2
\,{\rm d}t\,
\|\nabla\boldsymbol{\varphi}\|_{\mathbb L^2}.
\]
Therefore,
\begin{align}
    \label{eq:T1-sup-1}
    \begin{aligned}
\mathbb E\left[
\sup_{1\le m \le M}
\left\|
\mathcal T_{2,R,m}
\right\|_{\mathbb V^{-1}}^2
\right]
&\le 
C\mathbb E\left[
\sup_{1\le m \le M}\sup_{t,s \in [t_{m-1}^R, \mathfrak{t}_{m}^R]}
\left\|\Phi(W(t)-W(s))\right\|_{\mathbb V}^{4}
\right] \\&\le 
C_{\varepsilon, R, \Phi}\tau^{2-\varepsilon},
\end{aligned}
\end{align}
thanks to Doob's maximum inequality and the fact that $\Phi \in L^2(\mathfrak{U}, \mathbb V).$

\noindent
We next estimate \(\mathcal T_{3,R,m}\) and \(\mathcal T_{4,R,m}\).
Again by the continuity estimate for \(\mathcal C\),
\[
\left\|
\mathcal T_{3,R,m}
\right\|_{\mathbb V^{-1}}
+
\left\|
\mathcal T_{4,R,m}
\right\|_{\mathbb V^{-1}}
\]
is bounded by
\[
C
\left(
\frac1{\tau}
\int_{\mathfrak t_{m-1}^R}^{\mathfrak t_{m}^R}
\|{\bf y}(t)-\bar{\bf y}^R_{m-1}\|_{\mathbb L^2}^{2}
\,{\rm d}t
\right)^{1/2}
\left(
\frac1{\tau}
\int_{\mathfrak t_{m-1}^R}^{\mathfrak t_{m}^R}
\|\Phi(W(t)-\mathcal Q_{m-1,R}^{W})\|_{\mathbb H^2}^{2}
\,{\rm d}t
\right)^{1/2}.
\]
The first factor is bounded by
$
C_R\tau
\|{\bf y}(\cdot\wedge\mathfrak t_R)\|_{C^1([0,T];\mathbb L^2)}$ where as invoking  an argument similar to \eqref{eq:T1-sup-1} along with the fact that $\Phi \in L^2(\mathfrak{U};\mathbb {H}^2(\mt))$, we bound the second factor as
\[
\mathbb E\left[
\sup_{1\le m \le M}
\frac1{\tau}
\int_{\mathfrak t_{m-1}^R}^{\mathfrak t_{m}^R}
\|\Phi(W(t)-\mathcal Q_{m-1, R}^{W})\|_{\mathbb H^2}^{2}
\,{\rm d}t
\right]
\le
C_{\varepsilon,\Phi}\tau^{1-\varepsilon}.
\]
By Combining the above estimates, we get
\begin{equation}\label{eq:T3T4-sup-1}
\mathbb E\left[
\sup_{0\le m \le M}
\left(
\left\|
\mathcal T_{3,R,m}
\right\|_{\mathbb V^{-1}}
+
\left\|
\mathcal T_{4,R,m}
\right\|_{\mathbb V^{-1}}
\right)^2
\right]
\le
C_{\varepsilon,R, \Phi}\tau^{3-\varepsilon}.
\end{equation}
For $\mathcal{T}_{5,R,m}$, we use H\"older's inequality and the Sobolev embedding $\mathbb{H}^2(\mt) \hookrightarrow \mathbb L^\infty(\mt)$ along with localisation bound
$\|\bar{\bf Y}^R_m\|_{\mathbb{H}^2(\mt)} \le CR$  to obtain
\[
\|\mathcal{T}_{5,R,m}\|_{\mathbb{V}^{-1}}
\le
C_R
\left\|
\frac{1}{\tau}
\int_{\mathfrak{t}_{m-1}^R}^{\mathfrak{t}_m^R}
\bigl({\bf y}(t)-{\bf y}^R_{m}\bigr)
\,\mathrm{d}t
\right\|_{\mathbb{L}^2}.
\]
By using Jensen inequality and  Corollary \ref{cor:add}, we have
\begin{align}\label{eq:T5-sup-1}
\begin{aligned}
\mathbb{E}\left[
\sup_{1\le m \le M}
\|\mathcal{T}_{5,R,m}\|^2_{\mathbb{V}^{-1}}
\right] &\le C_R\,\tau^2\mathbb{E}\left[
\sup_{0\le m \le M}\sup_{t \in [\mathfrak{t}_{m-1}^R, \mathfrak{t}_m^R]}\|\partial_t{\bf y}(\cdot \wedge \mathfrak{t}_m^R)\|_{\mathbb L^2}^2\right] 
\\&\le
C_{\varepsilon,R}\,\tau^{2-\varepsilon}.
\end{aligned}
\end{align}
Similarly, for $\mathcal{T}_{6,R,m}$, we use the localization bound
$\|{\bf Y}^R_{m}\|_{\mathbb{H}^2(\mt)} \le CR$ to obtain 
\begin{equation}\label{eq:T6-sup-1}
\mathbb{E}\left[
\sup_{1 \le m \le M}
\|\mathcal{T}_{6,R,m}\|^2_{\mathbb{V}^{-1}}
\right]
\le
C_{\varepsilon,R}\,\tau^{2-\varepsilon}.
\end{equation}

\noindent
For \(m=1\),
\({\bf y}(0)={\bf y}(t_0)\), and hence only the first-order estimate
\[
\left\|
{\bf y}_1^R-\bar{\bf y}_1^R
\right\|_{\mathbb L^2_x}
\le
C\tau
\|{\bf y}(\cdot\wedge\mathfrak t_R)\|_{C^1([0,T];\mathbb L^2)}
\]
is available. Therefore,
\begin{equation}
\mathbb E\left[
\left\|
\mathcal T_{6,R,1}
\right\|_{\mathbb V^{-1}}^2
\right]
\le
C_R\tau^2.\notag
\end{equation}
Finally, we consider \(\mathcal T_{7,R,m}\). By definition,
\[
\left\|
\mathcal T_{7,R,m}
\right\|_{\mathbb V^{-1}}
\le
\nu
\left\|
\frac1{\tau}
\int_{\mathfrak t_{m-1}^R}^{\mathfrak t_{m}^R}
\left(
\nabla{\bf y}(t)-\nabla{\bf y}(\mathfrak{t}_m^R)
\right)
\,{\rm d}t
\right\|_{\mathbb L^2}.
\]
Thus, again by invoking Jensen inequality and  Corollary \ref{cor:add}, we get
\begin{align}\label{eq:T7-sup-1}
\begin{aligned}
\mathbb E\left[
\sup_{1 < m \le M}
\left\|
\mathcal T_{7,R,m} 
\right\|_{\mathbb V^{-1}}^2
\right]
&\le
C_R\,\tau^2\mathbb{E}\left[
\sup_{1 < m \le M}\sup_{t \in [\mathfrak{t}_{m-1}^R, \mathfrak{t}_m^R]}\|\partial_t\nabla{\bf y}(\cdot \wedge \mathfrak{t}_m^R)\|_{\mathbb L^2}^2\right] \\
&\le
C_{\varepsilon,R}\,\tau^{2-\varepsilon}.
\end{aligned}
\end{align}
By combining
\eqref{eq:T1-sup-1},
\eqref{eq:T3T4-sup-1},
\eqref{eq:T5-sup-1},
\eqref{eq:T6-sup-1}, and
\eqref{eq:T7-sup-1},
the claim follows.
\end{proof}
\noindent\textbf{Proof of Theorem \ref{thm:main}:}
Fix \(m\in\{1,\dots,M\}\). By integrating the continuous
equation on the stopped interval
\([\mathfrak t_{m-1}^R,\mathfrak t_{m}^R]\), writing it in terms of
\[
{\bf Y}(t)={\bf y}(t)+\Phi W(t),
\]
and testing with \(\boldsymbol{\varphi}\in\mathbb V\), we obtain
\begin{align}\label{eq:cont-avg2-1}
\begin{aligned}
\Big(
\frac{{\bf y}(\mathfrak t_{m}^R)-{\bf y}(\mathfrak t_{m-1}^R)}{\tau},
\boldsymbol{\varphi}
\Big)
&+
\frac1{\tau}
\int_{\mathfrak t_{m-1}^R}^{\mathfrak t_{m}^R}
\mathcal C({\bf Y}(t),{\bf Y}(t),\boldsymbol{\varphi})\,{\rm d}t\\&+
\nu\,\frac1{\tau}
\int_{\mathfrak t_{m-1}^R}^{\mathfrak t_{m}^R}
(\nabla{\bf Y}(t),\nabla\boldsymbol{\varphi})\,{\rm d}t
=0 .
\end{aligned}
\end{align}
The corresponding stopped discrete step reads
\begin{align}\label{eq:disc-step22-1}
\begin{aligned}
\Big(
\frac{{\bf y}_{m}^R-{\bf y}_{m-1}^R}{\tau},
\boldsymbol{\varphi}
\Big)
&+
\mathcal C\big(
{\bf y}_{m-1}^R+\Phi W(\mathfrak{t}_{m-1}^R),\,
{\bf y}_{m}^R+\Phi W(\mathfrak{t}_{m}^R),\,
\boldsymbol{\varphi}
\big)\\
&+
\nu\,
\big(
\nabla({\bf y}_{m}^R+\Phi W(\mathfrak{t}_m^R)),
\nabla\boldsymbol{\varphi}
\big)
=0 .
\end{aligned}
\end{align}
By subtracting \eqref{eq:disc-step22-1} from \eqref{eq:cont-avg2-1}, multiplying with $\tau$ and using
$
{\bf e}_{\ell}^R
=
{\bf y}(\mathfrak t_{\ell}^R)-{\bf y}_{\ell}^R$
for $\ell \in \{m, m-1\}$, 
and choosing $\boldsymbol{\varphi}= {\bf e}_m^R $, we have
\begin{align}\label{eq:err-local-raw-1}
&\Big(
{\bf e}_{m}^R-{\bf e}_{m-1}^R,
{\bf e}_{m}^R
\Big)
+
\tau\nu
\|\nabla{\bf e}_{m}^R\|_{\mathbb L^2}^2 = 
-\!\int_{\mathfrak{t}_{m-1}^R}^{\mathfrak{t}_{m}^R}\! \mathcal C({\bf {Y}}(t),{\bf {Y}}(t),{\bf e}_m^R)\,\mathrm{d}t \\&\quad-\nu\!\int_{\mathfrak{t}_{m-1}^R}^{\mathfrak{t}_{m}^R}\! \big(\nabla ({\bf {Y}}(t)-{\bf Y}_m^R),\nabla{\bf e}_m^R\big)\,\mathrm{d}t\notag\\&\quad+\tau\mathcal C\big(
{\bf y}_{m-1}^R+\Phi W(\mathfrak{t}_{m-1}^R),\,
{\bf y}_{m}^R+\Phi W(\mathfrak{t}_{m}^R),\,
{\bf e}_m^R
\big) \notag \\
&\quad= -\!\int_{\mathfrak{t}_{m-1}^R}^{\mathfrak{t}_{m}^R}\! \Big(\mathcal C({\bf {Y}}(t),{\bf {Y}}(t),{\bf e}_m^R)-\mathcal C(\bar{{\bf Y}}^R_{m-1},{{\bf Y}}^R_{m},{\bf e}_m^R)\Big)\,\mathrm{d}t\notag\\ 
&\quad +
\tau\mathcal C\!\bigl({{\bf y}}(\mathfrak{t}_{m-1}^R)-\bar{{\bf y}}^R_{m-1},\, {{\bf Y}}^R_{m},\, {\bf e}_m^R\bigr)
- \nu\!\int_{\mathfrak{t}_{m-1}^R}^{\mathfrak{t}_{m}^R}\! \big(\nabla ({\bf {Y}}(t)-{\bf Y}^R_{m}),\nabla{\bf e}_m^R\big)\,\mathrm{d}t \notag \\ &\quad -\tau\mathcal C\big(\Phi (Q_{m-1,R}^W-W(\mathfrak{t}_{m-1}^R)) 
,\,
{\bf Y}_{m}^R,\,
{\bf e}_{m}^R
\big) -\tau\mathcal C({\bf Y}^R_{m-1}, {\bf Y}^R_{m}, {\bf e}^R_{m}) \notag\\ &\quad+\mathcal \tau C\big(
{\bf y}_{m-1}^R+\Phi W(\mathfrak{t}_{m-1}^R),\,
{\bf y}_{m}^R+\Phi W(\mathfrak{t}_{m}^R),\,
{\bf e}_m^R
\big)\notag\\ & \quad=
-\tau\big\langle
{\mathcal R}_{m}^R,
{\bf e}_{m}^R
\big\rangle -\mathcal \tau C\big({\bf e}_{m-1}^R
,\,
{\bf Y}_{m}^R,\,
{\bf e}_{m}^R
\big)\, - \tau\mathcal C\big({\bf y}_{m-1}^R + \Phi W(\mathfrak{t}_{m-1}^R)
,\,
{\bf e}_{m}^R,\,
{\bf e}_{m}^R
\big) 
\notag \\&\quad-\tau\mathcal C\big(\Phi (Q_{m-1,R}^W-W(\mathfrak{t}_{m-1}^R))
,\,
{\bf Y}_{m}^R,\,
{\bf e}_{m}^R
\big) -\!\int_{\mathfrak{t}_{m-1}^R}^{\mathfrak{t}_{m}^R}\! \mathcal C(\bar{\bf {Y}}_{m-1}^R, \Phi( W(t) -W(\mathfrak{t}_{m}^R)), {\bf e}_m^R)\,\mathrm{dt}\, \notag\\
&\quad- \nu\!\int_{\mathfrak{t}_{m-1}^R}^{\mathfrak{t}_{m}^R}\! \big(\nabla \Phi(W(t)-W(\mathfrak t^R_{m})), \nabla{\bf e}_m^R\big)\,\mathrm{d}t\,,\notag
\end{align}
where $\langle
{\mathcal R}_{m}^R,
{\bf e}_{m}^R
\big\rangle$ is defined in \ref{eq:R-def-1}. The left-hand side  of \eqref{eq:err-local-raw-1} becomes
\begin{equation}\label{erro-1}\tfrac{1}{2} \big(\Vert {\bf e}_m^R\Vert^2_{\mathbb L^2} - \Vert {\bf e}_{m-1}^R\Vert^2 _{\mathbb L^2}+ \Vert {\bf e}_m^R - {\bf e}_{m-1}^R\Vert^2_{\mathbb L^2}\big) + \tau\nu\Vert \nabla {\bf e}_m^R\Vert^2_{\mathbb L^2} \, .
\end{equation}
First, by duality and Young's inequality, for every
\(\delta\in(0,\nu)\),
\begin{align}\label{eq:Rtilde-bound-1}
\big|
\big\langle
\mathcal R_{m}^R,
{\bf e}_{m}^R
\big\rangle
\big|
&\le
\|{\mathcal R}_{m}^R\|_{\mathbb V^{-1}}
\,
\|\nabla{\bf e}_{m}^R\|_{\mathbb L^2}
\nonumber\\
&\le
\frac{\delta}{4}
\|\nabla{\bf e}_{m}^R\|_{\mathbb L^2}^2
+
C(\delta)
\|{\mathcal R}_{m}^R\|_{\mathbb V^{-1}}^2.
\end{align}
By using the continuity estimate for the trilinear form and the localization
bound $
\|{\bf Y}_{m}^R\|_{\mathbb H^2(\mt)}\le CR$,
we obtain
\begin{align}\label{eq:Econv-bound-local-1}
\big|
\mathcal C\big(
{\bf e}_{m-1}^R,\,
{\bf Y}_{m}^R,\,
{\bf e}_{m}^R
\big)
\big|
&\le
C\,
\|{\bf e}_{m-1}^R\|_{\mathbb L^2}
\,
\|{\bf Y}_{m-1}^R\|_{\mathbb H^2}
\,
\|\nabla{\bf e}_{m}^R\|_{\mathbb L^2(\mt)}
\nonumber\\
&\le
C_R\,\|{\bf e}_{m-1}^R\|_{\mathbb L^2}
\,
\|\nabla{\bf e}_{m}^R\|_{\mathbb L^2}
\,.
\end{align}
Hence, again by Young's inequality,
\begin{equation}\label{eq:Econv-Young-1}
\big|
\mathcal C\big(
{\bf e}_{m-1}^R,\,
{\bf Y}_{m}^R,\,
{\bf e}_{m}^R
\big)
\big|
\le
\frac{\delta}{4}
\|\nabla{\bf e}_{m}^R\|_{\mathbb L^2}^2
+
C_R(\delta)\|{\bf e}_{m-1}^R\|_{\mathbb L^2}^2\,.
\end{equation}
Thanks to the property of trilinear form
\begin{align}\label{err-1}
    \mathcal C\big({\bf y}_{m-1}^R + \Phi W(\mathfrak{t}_{m-1}^R),\,
{\bf e}_{m}^R,\,
{\bf e}_{m}^R
\big) = 0\,.
\end{align}
Thanks to the independent increment of the Wiener process, Doob's maximal inequality along with the fact that $\Phi\in L_2(\mathfrak U;\mathbb H^{2}(\mt))$, we follow the estimations as in  \cite{BrPr2} to get
\begin{align}\label{err-3}
    &\mathbb E \Big[ \nu\!\int_{\mathfrak{t}_{m-1}^R}^{\mathfrak{t}_{m}^R}\! \big(\Delta \Phi(W(t)-W(\mathfrak t^R_{m})),{\bf e}_m^R\big)\,\mathrm{d}t\Big] \le \frac{1}{2} {\mathbb E}\bigl[\Vert {\bf e}_m^R - {\bf e}_{m-1}^R\Vert^2_{\mathbb L^2}\bigr] + C\tau^3\,.
\end{align}
We re-write 
\begin{align}
&\!\int_{\mathfrak{t}_{m-1}^R}^{\mathfrak{t}_{m}^R}\! \mathcal C(\bar{\bf {Y}}_{m-1}^R, \Phi( W(t) -W(\mathfrak{t}_m^R)), {\bf e}_m^R)\,\mathrm{dt} \notag \\ & =  \!\int_{\mathfrak{t}_{m-1}^R}^{\mathfrak{t}_{m}^R}\! \mathcal C(\bar{\bf {Y}}_{m-1}^R, \Phi( W(t) -W(\mathfrak{t}_m^R)), {\bf e}_m^R - {\bf e}_{m-1}^R)\,\mathrm{dt} \notag \\
& \quad + \!\int_{\mathfrak{t}_{m-1}^R}^{\mathfrak{t}_{m}^R}\! \mathcal C(\bar{\bf {Y}}_{m-1}^R - \big({\bf y}(\mathfrak{t}_{m-1}^R) + \Phi W(\mathfrak{t}_{m-1}^R)\big), \Phi( W(t) -W(\mathfrak{t}_m^R)), {\bf e}_{m-1}^R)\,\mathrm{dt} \notag \\
& +\!\int_{\mathfrak{t}_{m-1}^R}^{\mathfrak{t}_{m}^R}\! \mathcal C({\bf y}(\mathfrak{t}_{m-1}^R) + \Phi W(\mathfrak{t}_{m-1}^R), \Phi( W(t) -W(\mathfrak{t}_m^R)), {\bf e}_{m-1}^R)\,\mathrm{dt} \notag=: \sum_{i=1}^3\mathcal{E}_i^{m}\,.
\end{align}
Observe that $\sum_{n = 1}^m \mathcal{E}_3^{n}$ can be written as composition of the $\mathfrak{F}_{t_m}$-martingale
$$-\sum_{n=1}^m \Big \langle \int_{t_{n-1}}^{t_n} ({\bf y}(\mathfrak{t}_{m-1}^R) + \Phi W(\mathfrak{t}_{m-1}^R)) \otimes \Phi( W(t) -W(\mathfrak{t}_m^R)), \nabla{\bf e}_{m-1}^R)\Big \rangle$$
with the $(\mathfrak F_{t_m})$-stopping time $\mathfrak m_R$ defined after \eqref{tdiscrtilde}. Hence it is an $(\mathfrak F_{t_m})$-martingale in its own right. Thus, $\mathbb{E}\Big[\sum_{n = 1}^m \mathcal{E}_3^{n}\Big] = 0\,.$

\noindent
By using H\"older's inequality along with Sobolev's embedding $\mathbb{H}^{2}(\mt) \hookrightarrow \mathbb L^\infty(\mt)$, the localization bound of $\|\bar{\bf Y}_{m-1}^R\|_{\mathbb H^2(\mt)} \le R$ and Young's inequality, we infer
\begin{align}
    \mathbb{E}\big[\mathcal{E}_1\big] 
     &\le  C_R \mathbb{E}\Big[\tau\sup_{t \in [\mathfrak{t}_{m-1}^R, \mathfrak{t}_m^R]}\Big(\big\|\nabla\Phi(W(t) - W(\mathfrak{t}_m^R))\big\|_{\mathbb L^2}\Big)\|({\bf e}_m^R-{\bf e}_{m-1}^{R})\|_{\mathbb L^2}\Big] \notag \\
     & \le \delta\mathbb{E}\Big[\|{\bf e}_m^R-{\bf e}_{m-1}^{R}\|_{\mathbb L^2}^2\Big] + C_R(\delta)\tau^2\mathbb{E}\Big[\sup_{t \in [\mathfrak{t}_{m-1}^R, \mathfrak{t}_m^R]}\big\|\nabla\Phi(W(t) - W(\mathfrak{t}_m^R))\big\|_{\mathbb L^2}^2\Big]\notag \\ & \le \delta \mathbb{E}\Big[\|({\bf e}_m^R-{\bf e}_{m-1}^{R})\|_{\mathbb L^2}^2\Big] + C_{\varepsilon, R, \Phi}(\delta)\tau^{3-\varepsilon}\,, 
\end{align}
where we have used Doob's maximal inequality and  $\Phi\in L_2(\mathfrak U; \mathbb H^{2}(\mt))$. 
Similarly,
\begin{align*}
    &\mathbb{E}\big[\mathcal{E}_2\big]  = \!\mathbb{E}\Big[\int_{\mathfrak{t}_{m-1}^R}^{\mathfrak{t}_{m}^R}\int_{\mt}\! \big(\bar{\bf {Y}}_{m-1}^R - \big({\bf y}(\mathfrak{t}_{m}^R) + \Phi W(\mathfrak{t}_{m-1}^R)\big)\big) \otimes \Phi( W(t) -W(\mathfrak{t}_m^R)) :\nabla{\bf e}_{m-1}^R\, \mathrm{dx}\,\mathrm{dt} \Big] \notag\\
    &\le \mathbb E \bigg[\tau \underset{t \in [\mathfrak{t}_{m-1}^R, \mathfrak{t}_{m-1}^R] }{\sup}\Big(\|\bar{\bf {Y}}_m^R - \big({\bf y}(\mathfrak{t}_{m-1}^R) + \Phi W(\mathfrak{t}_{m-1}^R)\big)\|_{\mathbb L^2} \|\Phi( W(t) -W(\mathfrak{t}_m^R))\|_{\mathbb H^{2}} \Big)\|\nabla {\bf e}_{m-1}^R\|_{\mathbb L^2}\bigg]\notag \\
    & \le \delta\tau\mathbb{E}\big[\|\nabla {\bf e}_{m-1}^R\|_{\mathbb L^2}^2\big] + C_R(\delta)\tau \mathbb{E}\Big[\sup_{t \in [\mathfrak{t}_{m-1}^R, \mathfrak{t}_m^R]} \|\Phi( W(t) -W(\mathfrak{t}_m^R))\|_{\mathbb H^{2}}^2 \notag \\& \hspace{5cm} \times \big(\tau^2\|\partial_t{\bf y}(\cdot \wedge \mathfrak{t}_m^R)\|_{\mathbb L^2}^2 + \|\Phi(W(t) -W(\mathfrak{t}_m^R))\|_{\mathbb L^2}^2\big)\Big]\notag \\
    & \le \delta\tau\mathbb{E}\big[\|\nabla {\bf e}_{m-1}^R\|_{\mathbb L^2}^2\big] + C_R(\delta)\tau \Big(\mathbb{E}\Big[\sup_{t \in [\mathfrak{t}_{m-1}^R, \mathfrak{t}_m^R]} \|\Phi( W(t) -W(\mathfrak{t}_m^R))\|_{\mathbb H^{2}}^4\big]\Big)^{\frac{1}{2}} \notag \\& \hspace{1cm} \times \bigg[\tau^2\Big(\mathbb{E}\Big[\sup_{t \in [\mathfrak{t}_{m-1}^R, \mathfrak{t}_m^R]}\|\partial_t{\bf y}(\cdot \wedge \mathfrak{t}_m^R)\|_{\mathbb L^2}^4 \Big]\Big)^{\frac{1}{2}} + \Big(\mathbb{E}\Big[\sup_{t \in [\mathfrak{t}_{m-1}^R, \mathfrak{t}_m^R]}\|\Phi(W(t) -W(\mathfrak{t}_m^R))\|_{\mathbb L^2}^4\Big]\Big)^{\frac{1}{2}}\bigg]\notag\\
    & \le \delta\tau\mathbb{E}\big[\|\nabla {\bf e}_{m-1}^R\|_{\mathbb L^2}^2\big] + C_{\varepsilon, R, \Phi}\tau^{3-\epsilon}.
\end{align*}

\noindent
We re-arrange to get
\begin{align}
&\tau\mathcal C\big(\Phi (W(\mathfrak{t}_{m-1}^R)-Q_{m-1,R}^W)
,\,
{\bf Y}_{m}^R,\,
{\bf e}_{m}^R
\big)
\notag \\ & = \tau\mathcal C\big(\Phi (W(\mathfrak{t}_{m-1}^R)-Q_{m-1,R}^W)
,\,
{\bf Y}_{m}^R,\,
{\bf e}_{m}^R - {\bf e}_{m-1}^R 
\big)\notag \\ & \quad +  \tau\mathcal C\big(\Phi (W(\mathfrak{t}_{m-1}^R)-Q_{m-1,R}^W)
,\,
{\bf Y}_{m}^R - ({\bf y}(\mathfrak{t}_{m}^R) + \Phi W(\mathfrak{t}_{m}^R)),\, {\bf e}_{m-1}^R  
\big)\notag \\
&\quad  + \tau\mathcal C\big(\Phi (W(\mathfrak{t}_{m-1}^R)-Q_{m-1,R}^W)
,\,
{\bf y}(\mathfrak{t}_{m}^R) + \Phi W(\mathfrak{t}_{m}^R),\, {\bf e}_{m-1}^R  
\big)\, =: \sum_{i=1}^3\mathcal{G}_{i}^{m}\notag
\end{align}
Note that $\mathbb{E}\Big[\sum_{n=1}^{m}\mathcal{G}_3^n\Big] = 0.$ Recall that $Q_{m,R}^{W} = \frac{1}{\tau}\int_{\mathfrak{t}_{m-1}^R}^{\mathfrak{t}_{m}^R}\Phi W(s)\,\mathrm{ds}$. Now proceeding with similar estimations as in $\mathcal{E}_1^m$ and $\mathcal{E}_2^m$, we get
\begin{align}\label{eq:Econv-Young-2}
    &\mathbb{E}\big[\mathcal{G}_1\big] \le \delta\mathbb{E}\big[\|{\bf e}_m^R -{\bf e}_{m-1}^R\|_{\mathbb L^2}^2\big]  + C_{\varepsilon, R, \Phi}\tau^{3-\varepsilon}\,,\\
    & \mathbb{E}\big[\mathcal{G}_2\big] \le \delta\tau\mathbb{E}\big[\|\nabla {\bf e}_{m-1}^R\|_{\mathbb L^2}^2\big] + C_{\varepsilon, R, \Phi}\tau^{3-\epsilon}\,.\notag
\end{align}

\noindent
By combining \eqref{eq:err-local-raw-1}--\eqref{eq:Econv-Young-2}, and choosing
\(\delta\in(0,\nu)\), we absorb the gradient terms into the left-hand side and in the resulting inequality multiplying by \(2\) and summing from $j = 0$ to $j= k$ for $k \in \{1,\dots,m\}$ and invoking Lemma \ref{lem:R-bound-1}, we get (recall that ${\bf e}_1^R= 0$)
\begin{align}\label{eq:sum-local-k-1}
\begin{aligned}
\sup_{1 \le m \le M}\mathbb{E}\big[\|{\bf e}_{m}^R\|_{\mathbb {L}^2}^2\big]
&+
\nu\tau\mathbb{E}\Big[
\sum_{j=0}^{m}
\|\nabla{\bf e}_{j}^R\|_{\mathbb {L}^2}^2\Big]
\le C_{\varepsilon,R, \Phi}\tau^{ 2-\varepsilon} + C_r\tau\mathbb{E}\Big[\sum_{j=1}^{M}\|{\bf e}_j^R\|_{\mathbb L^2}^2\Big]\,.
\end{aligned}
\end{align}
An application of the discrete Gronwall lemma gives
\begin{align}
    \sup_{1 \le m \le M}\mathbb{E}\big[\|{\bf e}_{m}^R\|_{\mathbb {L}^2}^2\big] \le  C_{\varepsilon,R, \Phi}\tau^{ 2-\varepsilon} \,,
\end{align}
which conclude the proof of the theorem.\hfill$\Box$

\section{The Crank--Nicolson Scheme}\label{theo-2}
As in the previous section, we consider the random PDE
\eqref{eq:SNSy} and aim at a temporal discretization. However, we employ now a scheme of Crank--Nicolson type to obtain a higher convergence rate.


\subsection{Time--discretization scheme}\label{sec: time-discretization}

Let $\{t_n=n\tau\}_{n=0}^N$ be a uniform partition of $[0,T]$ with time
step $\tau>0$. For $n\ge0$, we define the discrete midpoint and the BDF2
extrapolated values by
\[
{{\bf y}}_{n+\frac12} := \tfrac12({\bf y}_{n+1}+{\bf y}_n), 
\qquad 
{\bf y}_{\star,n+\frac12}:=\tfrac32\,{\bf y}_n-\tfrac12\,{\bf y}_{n-1},
\quad n\ge 0,\qquad {\bf y}_{-1}:={\bf y}_0.
\]
On each sub-interval $I_n=[t_n,t_{n+1}]$ we introduce the
fine mesh
\[
t_{n,\ell}=t_n+\ell\,\tau^2,\qquad \ell=0,\dots,M,\qquad M=\tau^{-1}.
\]

We now describe the various quadrature and approximation ingredients
used in the construction of the scheme.

\begin{itemize}[leftmargin=2em]
  \item[(i)] \textbf{Brownian quadrature on a fine mesh.}
  We approximate the time average of the Brownian motion on $I_n$,
  \[
    \mathcal Q_n^W:=\frac1{\tau}\int_{t_n}^{t_{n+1}} W(s)\,\mathrm{d}s,
  \]
  by the Riemann sum
  \[
    \mathcal I_n^W:=\sum_{\ell=1}^{M}\tau\, W(t_{n,\ell}).
  \]
  The mean-square quadrature error satisfies (see Lemma~\ref{lem:BM-quadrature-sup})
  \begin{equation}\label{eq:QWvsIW}
  \mathbb E\!\left[\sup_{0\le\,n\le N-1}\|\Phi\mathcal Q_n^W-\Phi\mathcal I_n^W\|_{\mathbb{L}^2}^2\right]
  \le\ C_{\Phi, \varepsilon}\,\tau^{3-\varepsilon},
  \end{equation}
  where $\varepsilon>0$ is arbitrary.

  \item[(ii)] \textbf{Matrix-valued Brownian triple integral.}
  We define the (matrix-valued) triple integral
  \begin{equation*}
  \mathcal{Q}^{W^2}_n
  :=\frac{1}{\tau^3}\int_{t_n}^{t_{n+1}}\!\int_{t_n}^{t_{n+1}}\!\int_{t_n}^{t_{n+1}}
      \bigl(\Phi W(t)-\Phi W(s)\bigr)\otimes\bigl(\Phi W(t)-\Phi W(r)\bigr)\,
      \mathrm{d}s\,\mathrm{d}t\,\mathrm{d}r.
  \end{equation*}
  Its discrete approximation on the fine mesh is
  \begin{align*}
  \mathcal{I}^{W^2}_{n}
  &:= \tau^3\sum_{\ell=1}^{M}\sum_{k=1}^{M}\sum_{j=1}^{M}
      \bigl(\Phi W(t_{n,\ell})-\Phi W(t_{n,k})\bigr)\otimes
      \bigl(\Phi W(t_{n,\ell})-\Phi W(t_{n,j})\bigr)\\
      &=\tau\sum_{\ell=1}^{M}
      \bigl(\Phi W(t_{n,\ell})-\Phi \mathcal{I}_n^W\bigr)\otimes
      \bigl(\Phi W(t_{n,\ell})-\Phi \mathcal{I}_n^W\bigr),
  \end{align*}
  which satisfies the mean-square bound (see Lemma~\ref{lem:triple-BM-sup})
  \begin{equation}\label{eq:QW2vsIW2}
  \mathbb{E}\big[\sup_{0\le n\le N-1}\|\mathcal{Q}_n^{W^2}-\mathcal{I}_n^{W^2}\|_{{\mathbb{L}^2}}^2\big]
  \le C_{\varepsilon, \Phi}\,\tau^{3-\varepsilon}
  \end{equation}
  for any $\varepsilon>0$
  \item[(iii)] \textbf{Integral involving the convective term.}
  We denote the time average of the convective term on $I_n$ by
  \[
  \mathcal{Q}_n^{{\bf y}^2}:=\frac{1}{\tau}\int_{t_n}^{t_{n+1}}
  \diver\big(({\bf y}(t)+\Phi W(t))\otimes ({\bf y}(t)+\Phi W(t))\big)\,\mathrm{d}t.
  \]
  Its discrete approximation is chosen as
  \[
  \mathcal{I}_{n}^{{\bf y}^2}
  :=\diver\Big(\big(\bar{{\bf y}}_{\star,n+\frac{1}{2}}+\Phi\mathcal{I}_n^W\big)
               \otimes\big(\bar{{\bf y}}_{n+\frac{1}{2}}+\Phi \mathcal{I}_n^W\big)\Big)
    + \diver \mathcal{I}_{n}^{W^2},
  \]
  where
  \[
  \bar{{\bf y}}_{n+\frac12}:=\frac{{\bf y}(t_n)+{\bf y}(t_{n+1})}{2},
  \qquad 
  \bar{\bf y}_{\star,n+\frac12}:=\frac{3}{2}{\bf y}(t_n)-\frac{1}{2}{\bf y}(t_{n-1}),
  \qquad {\bf y}(t_{-1}):={\bf y}_0.
  \]

  \item[(iv)] \textbf{Integral involving the diffusive term.}
  We denote the averaged diffusive term on $I_n$ by
  \[
  \mathcal{Q}_n^{\Delta {\bf y}}
  :=\frac{1}{\tau}\int_{t_n}^{t_{n+1}}\Delta \big[{\bf y}(t)+\Phi W(t)\big]\,\mathrm{d}t,
  \]
  and approximate it by
  \[
  \mathcal{I}_{n}^{\Delta {\bf y}}
  :=\Delta\big[\bar{{\bf y}}_{n+\frac12}+\Phi\mathcal{I}_n^W\big].
  \]
\end{itemize}

The above ingredients motivate the following time--semi--discretization
of the random PDE~\eqref{req:stochastic-NS}.


\begin{scheme}
\medskip\noindent
\textbf{Linear Crank--Nicolson scheme for the random PDE~\eqref{req:stochastic-NS}.} For all $n\ge 0$, given ${\bf y}_n,{\bf y}_{n-1}\in\mathbb V$ (with ${\bf y}_{-1}={\bf y}_0$), find
$({\bf y}_{n+1},p_{n+1})\in\mathbb H^1(\mt)\times\Qspace(\mt)$ such that, for all
$({\boldsymbol{\varphi}},q)\in\mathbb H^1(\mt)\times\Qspace(\mt)$,
\begin{align}\label{scheme:main0}
\begin{cases}
\displaystyle
\Big(\frac{{\bf y}_{n+1}-{\bf y}_n}{\tau},{\boldsymbol{\varphi}}\Big)
+ \mathcal C\!\big({\bf y}_{\star,n+\frac12}+\Phi\mathcal I_n^W,\;
                   {{\bf y}}_{n+\frac12}+\Phi\mathcal I_n^W,\;
                   {\boldsymbol{\varphi}}\big)
 -\big(\mathcal{I}_n^{W^2},\nabla{\boldsymbol{\varphi}}\big)
\\[1ex]\qquad\qquad\qquad\qquad\qquad\qquad
+ \nu\,\big(\nabla({{\bf y}}_{n+\frac12}+\Phi\mathcal I_n^W),\nabla{\boldsymbol{\varphi}}\big)
 - (p_{n+1},\mathrm{div}\,{\boldsymbol{\varphi}})
 =0,\\[1ex]
(\mathrm{div}\, {\bf y}_{n+1},q)=0.
\end{cases}
\end{align}
\end{scheme}

Note that the convective term in \eqref{scheme:main0} is linear in
${\bf y}_{n+1}$; the nonlinearity appears only through the extrapolated,
divergence-free advecting field ${\bf y}_{\star,n+\frac12}+\Phi\mathcal I_n^W$.

\subsection{A numerical scheme for the SPDE~\eqref{eq:SNS}}
\label{snse_scheme}

We now rewrite the time-discrete random PDE scheme in terms of the
original stochastic variable $X$. For all $n\ge 0$ we define the
Brownian quadrature correction terms
\[
\mathcal{J}_{\star,n+\frac{1}{2}}
:=\mathcal{I}_n^W-\frac{3}{2}W(t_n)+\frac{1}{2}W(t_{n-1}),
\qquad W(t_{-1})=W(0)=0,
\]
\[
\mathcal{J}_{n+\frac12}
:=\mathcal{I}_n^W-\frac{1}{2}\big(W(t_{n+1})+W(t_{n})\big),
\qquad
\Delta_{n+1}W:=W(t_{n+1})-W(t_{n}).
\]
We then define the discrete transformation
\begin{equation}\label{discrete transformation}
{\bf u}_{n}:={\bf y}_n+\Phi W(t_n),
\qquad n\ge 0.
\end{equation}
By construction, the sequence $\{({\bf u}_n,p_n)\}_{n\ge 1}$ satisfies the
following time--semi--discrete scheme for the original SPDE
\eqref{eq:SNS}.

\begin{scheme}
\medskip\noindent
\textbf{Modified Crank--Nicolson scheme for the SPDE~\eqref{eq:SNS}.}  For all $n\ge 0$, given ${\bf u}_n,{\bf u}_{n-1}\in\Vspace$ (with ${\bf u}_{-1}={\bf u}_0$),
find $({\bf u}_{n+1},p_{n+1})\in\mathbb H^1(\mt)\times\Qspace(\mt)$ such that, for all
$({\boldsymbol{\varphi}},q)\in\mathbb H^1(\mt)\times\Qspace(\mt)$,
\begin{align}\label{scheme:main}
\begin{cases}
\displaystyle
\Big(\frac{{\bf u}_{n+1}-{\bf u}_n}{\tau},{\boldsymbol{\varphi}}\Big)
+ \mathcal C\!\big({\bf u}_{\star,n+\frac12}+\Phi\mathcal J_{\star,n+\frac{1}{2}},\;
                   {\bf u}_{n+\frac12}+\Phi\mathcal J_{n+\frac{1}{2}},\;
                   {\boldsymbol{\varphi}}\big)
 -\big(\mathcal{I}_n^{W^2},\nabla{\boldsymbol{\varphi}}\big)
\\[1ex]\qquad\qquad
+ \nu\,\big(\nabla({\bf u}_{n+\frac12}+\Phi\mathcal J_{n+\frac{1}{2}}),\nabla{\boldsymbol{\varphi}}\big)
 - (p_{n+1},\mathrm{div}\,{\boldsymbol{\varphi}})
 =\bigl(\Phi,{\boldsymbol{\varphi}}\bigr)\frac{\Delta_{n+1}W}{\tau},\\[1ex]
(\mathrm{div}\, {\bf u}_{n+1},q)=0,
\end{cases}
\end{align}
where, analogously to ${{\bf y}}_{n+\frac12}$ and ${\bf y}_{\star,n+\frac12}$, we set
\[
{\bf u}_{n+\frac12}:=\tfrac12({\bf u}_{n+1}+{\bf u}_n),
\qquad
{\bf u}_{\star,n+\frac12}:=\tfrac32 {\bf u}_n-\tfrac12 {\bf u}_{n-1},
\qquad {\bf u}_{-1}:={\bf u}_0.
\]
\end{scheme}

\medskip

For practical implementation, it is convenient to summarize the scheme
\eqref{scheme:main} in algorithmic form; see Algorithm~\ref{alg:MCN-SPDE}.

\begin{algorithm}[!htbp]
\caption{Modified Crank--Nicolson scheme for the SPDE~\eqref{eq:SNS}}
\label{alg:MCN-SPDE}
\begin{algorithmic}[1]
\State \textbf{Input:} Final time $T>0$, time step $\tau>0$, $N:=T/\tau$,
       initial velocity ${\bf u}_0\in\Vspace$.
\State \textbf{Initialization:} Set ${\bf u}_{-1}:={\bf u}_0$, $W(t_0):=0$.
\For{$n=0,\dots,N-1$}
  \State Generate the Brownian increment
         $\Delta_{n+1}W:=W(t_{n+1})-W(t_n)$.
  \State Construct the fine mesh
         $t_{n,\ell}=t_n+\ell\,\tau^2$, $\ell=0,\dots,M$, $M=\tau^{-1}$.
  \State Compute the Brownian quadratures
  \[
    \mathcal I_n^W:=\sum_{\ell=1}^{M}\tau\,W(t_{n,\ell}),\qquad
    \mathcal{I}_n^{W^2}
    := \tau\sum_{\ell=1}^{M}
       \bigl(\Phi W(t_{n,\ell})-\Phi\mathcal{I}_n^W\bigr)\otimes
       \bigl(\Phi W(t_{n,\ell})-\Phi \mathcal{I}_n^W\bigr).
  \]
  \State Define the correction terms
  \[
    \mathcal{J}_{\star, n+\frac{1}{2}}
    :=\mathcal{I}_n^W-\frac{3}{2}W(t_n)+\frac{1}{2}W(t_{n-1}),\qquad
    \mathcal{J}_{n+\frac12}
    :=\mathcal{I}_n^W-\frac{1}{2}\big(W(t_{n+1})+W(t_{n})\big).
  \]
  \State Form the midpoint and extrapolated velocities
  \[
    {\bf u}_{n+\frac12}:=\tfrac12({\bf u}_{n+1}+{\bf u}_n),\qquad
    {\bf u}_{\star,n+\frac12}:=\tfrac32 {\bf u}_n-\tfrac12 {\bf u}_{n-1}.
  \]
  \State Find $({\bf u}_{n+1},p_{n+1})\in\Vspace(\mt)\times\Qspace(\mt)$ such that, for all
         $(\boldsymbol{\varphi},q)\in\Vspace(\mt)\times\Qspace(\mt)$,
  \[
  \begin{aligned}
    \Big(\tfrac{{\bf u}_{n+1}-{\bf u}_n}{\tau},{\boldsymbol{\varphi}}\Big)
    &+ \mathcal C\!\big({\bf u}_{\star,n+\frac12}+\Phi\mathcal J_{\star,n+\frac{1}{2}},\;
                        {\bf u}_{n+\frac12}+\Phi\mathcal J_{n+\frac{1}{2}},\;
                        {\boldsymbol{\varphi}}\big)
     -\big(\mathcal{I}_n^{W^2},\nabla{\boldsymbol{\varphi}}\big)\\
    &\quad+ \nu\,\big(\nabla({\bf u}_{n+\frac12}+\Phi\mathcal J_{n+\frac{1}{2}}),\nabla{\boldsymbol{\varphi}}\big)
     - (p_{n+1},\mathrm{div}\,{\boldsymbol{\varphi}})
     =\bigl(\Phi,{\boldsymbol{\varphi}}\bigr)\frac{\Delta_{n+1}W}{\tau},\\[0.5ex]
    (\mathrm{div}\, {\bf u}_{n+1},q)&=0.
  \end{aligned}
  \]
\EndFor
\State \textbf{Output:} Approximations $\{({\bf u}_n,p_n)\}_{n=1}^N$ to the
       velocity and pressure of \eqref{eq:SNS}.
\end{algorithmic}
\end{algorithm}
\subsection{Strong rate of convergence}

For $R\gg1$ we recall the definitions of $\mathfrak t_R^{\tt d}$, $\mathfrak t_m^R$, $\tau_m^R$, $\bfy_m^R$ and $\mathfrak m_R$ from \eqref{eq:tRd-1}, \eqref{eq:tmR} and \eqref{tdiscrtilde}.
Analogously, we introduce the stopped quadrature quantities
\(\mathcal Q_{n,R}^W\), \(\mathcal I_{n,R}^W\),
\(\mathcal Q_{n,R}^{W^2}\), and \(\mathcal I_{n,R}^{W^2}\) as the
corresponding quadratures computed along the stopped Brownian path
\(W(\cdot\wedge\mathfrak t_R^{\tt d})\).

In the following error analysis  we require that the pathwise solution \({\bf y}\) of
\eqref{req:stochastic-NS} satisfies
\begin{equation}\label{regularity}
\begin{aligned}
{\bf y}(\cdot\wedge\mathfrak t_R)
&\in
L^2
\big(\Omega;C([0,T];\mathbb H^2)\big)
\cap
{L}^4\big(\Omega;C^1([0,T];\mathbb V)\big),
\\
{\bf y}(\cdot\wedge\mathfrak t_R)
&\in
{L}^2\big(\Omega;C^{1,1/2-\varepsilon}([0,T];\mathbb V)\big).
\end{aligned}
\end{equation}
As a consequence of Lemma~\ref{lem:regadditive}, this regularity holds
under sufficiently smooth initial data and noise, for instance if
\[
{\bf u}_0\in
L^4(\Omega;\mathbb H^4)
\cap
L^{36}(\Omega;\mathbb V),
\quad
\Phi\in L_2(\mathfrak U;\mathbb H^4).
\]
Our main result is the following mean-square error estimate for the velocity. 
\begin{theorem}\label{thm:main-local}
Assume that $\bfu_0\in L^4(\Omega;\mathbb H^4(\mt))
\cap
L^{36}(\Omega;\mathbb V(\mt))$ and that $\Phi\in L_2(\mathfrak U;\mathbb V\cap \mathbb H^4(\mt))$. Let $\bfy$ be the solution
to \eqref{eq:SNSy}, and $(\bfy_n)_{n=1}^N$ be the solution to \eqref{scheme:main0}.
 
 Then, for every
\(\varepsilon>0\) and every \(R>0\), there exists a constant
\(C_{\varepsilon,R,T}>0\), independent of \(\tau\) and \(N\), such that
\begin{align}\label{eq:main-local}
\begin{aligned}
&\mathbb E\left[
\sup_{0\le n\le N-1}
\|{\bf y}(\mathfrak t_{n+1}^R)-{\bf y}_{n+1}^R\|_{\mathbb L^2}^2
\right]
\\
&\qquad
+
\nu\tau\,
\mathbb E\left[
\sup_{0\le k\le N-1}
\sum_{n=0}^{k}
\|\nabla{\bf e}_{n+\frac12}^R\|_{\mathbb L^2}^2
\right]
\le
C_{\varepsilon,R,T}\tau^{3-\varepsilon}.
\end{aligned}
\end{align}
\end{theorem}
The remainder of this subsection is devoted to the proof of Theorem~\ref{thm:main-local} presented
in the following subsections. 
\subsubsection{Consistency residual} For all $n\in\{0,1,\dots,N-1\}$, we introduce the shorthand notation
\[ {\bf {Y}}^R(t):={\bf y}(t\wedge\mathfrak{t}_{R}^d)+\Phi W(t\wedge\mathfrak{t}_R^d),\qquad \bar{\bf y}_{n+\frac{1}{2}}^R:=\frac{1}{\tau}\int_{\mathfrak{t}_n^R}^{\mathfrak{t}^R_{n+1}}{\bf y}(t)\,\mathrm{d}t,\qquad \bar{{\bf Y}}_{n+\frac12}^R:= \bar{\bf y}_{n+\frac{1}{2}}^R+\Phi\,\mathcal Q_{n,R}^W, \]
and 
\[
{{\bf Y}}^R_{n+\frac12}
:=
\frac{
{\bf y}(t_n\wedge\mathfrak t_R^{\tt d})
+
{\bf y}(t_{n+1}\wedge\mathfrak t_R^{\tt d})
}{2}
+
\Phi\mathcal I_{n,R}^W,
\qquad
\bar{\bf y}^{\,R}_{\star,n+\frac12}
:=
\frac{3}{2}{\bf y}(t_n\wedge\mathfrak t_R^{\tt d})
-
\frac{1}{2}{\bf y}(t_{n-1}\wedge\mathfrak t_R^{\tt d}) .
\]
Here we use the convention
\[
  {\bf y}(t_{-1}\wedge\mathfrak t_R^{\tt d})={\bf y}_0 .
\]
For each time level \(n=0,\dots,N-1\), we define the \emph{consistency residual} \(\mathcal{R}_{n+\frac12}\in \Vdual\) by: 
\begin{align}\label{eq:R-def} \langle \mathcal{R}^R_{n+\frac12},\boldsymbol{\varphi}\rangle &:=\frac{1}{\tau}\!\int_{\mathfrak{t}_{n}^R}^{\mathfrak{t}_{n+1}^R}\! \Big(\mathcal C({\bf {Y}}(t),{\bf {Y}}(t),\boldsymbol{\varphi})-\mathcal C(\bar{{\bf Y}}^R_{n+\frac12},{{\bf Y}}^R_{n+\frac12},\boldsymbol{\varphi})\Big)\,\mathrm{d}t\notag\\ 
&\quad- \mathcal C\!\bigl((\bar{\bf y}^R_{\star,n+\frac12}+\Phi \mathcal I_{n,R}^W)-\bar{{\bf Y}}^R_{n+\frac12},\, {{\bf Y}}^R_{n+\frac12},\, \boldsymbol{\varphi}\bigr)\notag\\
&\quad+ \nu\,\frac{1}{\tau}\!\int_{\mathfrak{t}_{n}^R}^{\mathfrak{t}_{n+1}^R}\! \big(\nabla {\bf {Y}}(t)-\nabla {{\bf Y}}^R_{n+\frac12},\nabla\boldsymbol{\varphi}\big)\,\mathrm{d}t
\end{align} 
 for any \(\boldsymbol{\varphi}\in \mathbb{V}\).
By using the bilinearity of $\mathcal{C}$, we obtain
\begin{align}\nonumber \langle \mathcal{R}_{n+\frac12}^R,\boldsymbol{\varphi}\rangle &:=\frac{1}{\tau}\!\int_{\mathfrak{t}_{n}^R}^{\mathfrak{t}_{n+1}^R}\!\mathcal C\bigl({\bf {Y}}(t)-\bar{{\bf Y}}^R_{n+\frac12},{\bf {Y}}(t)-\bar{{\bf Y}}^R_{n+\frac12},\boldsymbol{\varphi}\bigr)\,{\rm d}t\\&\quad+\underbrace{\frac{1}{\tau}\!\int_{\mathfrak{t}_{n}^R}^{\mathfrak{t}_{n+1}^R}\! \mathcal C\big( {\bf Y}(t)-\bar{{\bf Y}}^R_{n+\frac12},\bar{{\bf Y}}^R_{n+\frac12},\boldsymbol{\varphi})\big)\,\mathrm{d}t}_{=0}\notag\\ &\quad +\frac{1}{\tau}\!\int_{\mathfrak{t}_{n}^R}^{\mathfrak{t}_{n+1}^R}\! \mathcal C\big(\bar{{\bf Y}}^R_{n+\frac12},\,{\bf {Y}}(t)-{{\bf Y}}^R_{n+\frac12},\,\boldsymbol{\varphi}\big)\,\mathrm{d}t-\mathcal C\!\bigl((\bar{\bf y}^R_{\star,n+\frac12}+\Phi \mathcal I_{n,R}^W)-\bar{{\bf Y}}^R_{n+\frac12},\, {{\bf Y}}^R_{n+\frac12},\, \boldsymbol{\varphi}\bigr)\notag\\ &\quad+ \nu\,\frac{1}{\tau}\!\int_{\mathfrak{t}_{n}^R}^{\mathfrak{t}_{n+1}^R}\! \big(\nabla {\bf {Y}}(t)-\nabla {{\bf Y}}^R_{n+\frac12},\nabla\boldsymbol{\varphi}\big)\,\mathrm{d}t. 
\end{align}
By using the bilinearity of \(\mathcal C\) and the decomposition on $[\mathfrak{t}_{n}^R,\mathfrak{t}_{n+1}^R],$ 
\[ {\bf {Y}}(t)-\bar {{\bf Y}}^R_{n+\frac12} =\big({\bf y}(t)-\bar {{\bf y}}^R_{n+\frac12}\big)+\Phi\big(W(t)-\mathcal Q_{n,R}^W\big), \] 
a straightforward expansion yields 
\[ \langle\mathcal R^R_{n+\frac12},\boldsymbol{\varphi}\rangle =\sum_{i=1}^{7}\langle\mathcal{T}_{i,R,n},\boldsymbol{\varphi}\rangle, \]
where 
\begin{align*} \big\langle\mathcal{T}_{1,R,n},\boldsymbol{\varphi}\big\rangle &:=\frac{1}{\tau}\!\int_{\mathfrak{t}_{n}^R}^{\mathfrak{t}_{n+1}^R}\! \mathcal C\big({\bf y}(t)-\bar{{\bf y}}^R_{n+\frac12},\,{\bf y}(t)-\bar{{\bf y}}^R_{n+\frac12},\,\boldsymbol{\varphi}\big)\,\mathrm{d}t,\\[1mm]
\big\langle\mathcal{T}_{2,R,n},\boldsymbol{\varphi}\big\rangle &:=\frac{1}{\tau}\!\int_{\mathfrak{t}_{n}^R}^{\mathfrak{t}_{n+1}^R}\! \mathcal C\big(\Phi W(t)-\Phi\mathcal{Q}^{W}_{n,R},\, \Phi W(t)-\Phi\mathcal{Q}^{W}_{n,R},\, \boldsymbol{\varphi}\big)\,\mathrm{d}t,\\[1mm]
\big\langle\mathcal{T}_{3,R,n},\boldsymbol{\varphi}\big\rangle &:=\frac{1}{\tau}\!\int_{\mathfrak{t}_{n}^R}^{\mathfrak{t}_{n+1}^R}\! \mathcal C\big({\bf y}(t)-\bar{{\bf y}}^R_{n+\frac12},\, \Phi W(t)-\Phi\mathcal{Q}^{W}_{n,R},\, \boldsymbol{\varphi}\big)\,\mathrm{d}t,\\[1mm]
\big\langle\mathcal{T}_{4,R,n},\boldsymbol{\varphi}\big\rangle &:=\frac{1}{\tau}\!\int_{\mathfrak{t}_{n}^R}^{\mathfrak{t}_{n+1}^R}\! \mathcal C\big(\Phi W(t)-\Phi\mathcal{Q}^{W}_{n,R},\, {\bf y}(t)-\bar{{\bf y}}^R_{n+\frac12},\, \boldsymbol{\varphi}\big)\,\mathrm{d}t,\\[1mm]
\big\langle\mathcal{T}_{5,R,n},\boldsymbol{\varphi}\big\rangle &:=\frac{1}{\tau}\!\int_{\mathfrak{t}_{n}^R}^{\mathfrak{t}_{n+1}^R}\! \mathcal C\big(\bar{{\bf Y}}^R_{n+\frac12},\,{\bf {Y}}(t)-{{\bf Y}}_{n+\frac12}^R,\,\boldsymbol{\varphi}\big)\,\mathrm{d}t,\\[1mm]
\big\langle\mathcal{T}_{6,R,n},\boldsymbol{\varphi}\big\rangle &:=- \mathcal C\!\bigl((\bar{\bf y}_{\star,n+\frac12}^R+\Phi \mathcal I_{n,R}^W)-\bar{{\bf Y}}^R_{n+\frac12},\, {{\bf Y}}^R_{n+\frac12},\, \boldsymbol{\varphi}\bigr),\\[1mm] \big\langle\mathcal{T}_{7,R,n},\boldsymbol{\varphi}\big\rangle &:=\nu\,\frac{1}{\tau}\!\int_{\mathfrak{t}_{n}^R}^{\mathfrak{t}_{n+1}^R}\! \big(\nabla {\bf {Y}}(t)-\nabla {{\bf Y}}_{n+\frac12}^R,\nabla\boldsymbol{\varphi}\big)\,\mathrm{d}t. \end{align*}
The modified residual is defined by
\begin{align}\label{modified-residual-2}
\left<\widetilde{\mathcal R}_{n+\frac12}^R,\boldsymbol{\varphi}\right>
:=
\left<\mathcal R_{n+\frac12}^R,\boldsymbol{\varphi}\right>
+
\left(\mathcal Q_{n,R}^{W^2},\nabla\boldsymbol{\varphi}\right).
\end{align}
We obtain the following.
\begin{lemma}[Residual bound]\label{lem:R-bound}
For each small $\varepsilon>0$, there exists a constant \(C_{\varepsilon,R}>0\), independent of
\(\tau\) and \(N\), such that
\begin{equation}\label{eq:R-bound-sum}
\mathbb E\left[
\tau\sum_{n=0}^{N-1}
\left\|
\widetilde{\mathcal R}_{n+\frac12}^R
\right\|_{\mathbb V^{-1}}^{2}
\right]
\le
C_{\varepsilon,R}\tau^{3-\varepsilon}.
\end{equation}
\end{lemma}

\begin{proof}
Arguing as in \cite{BBCP}, it is enough to estimate the terms
\(\mathcal T_{i,R,n}\), \(i\neq2\), in the
\(\mathbb V^{-1}\)-norm, where also as in \cite{BBCP}
\begin{equation}\label{eq:T3T4-sup}
\mathbb E\left[
\sup_{0\le n\le N-1}
\left(\left\|
\mathcal T_{1,R,n}
\right\|_{\mathbb V^{-1}}
+
\left\|
\mathcal T_{3,R,n}
\right\|_{\mathbb V^{-1}}
+
\left\|
\mathcal T_{4,R,n}
\right\|_{\mathbb V^{-1}}
\right)^2
\right]
\le
C_{\varepsilon,R}\tau^{3-\varepsilon}.
\end{equation}
For \(\mathcal T_{5,R,n}\), we use the localization bound
\[
\|\bar{\bf Y}^R_{n+\frac12}\|_{\mathbb H^2(\mt)^3}\le C R.
\]
Therefore,
\[
\left\|
\mathcal T_{5,R,n}
\right\|_{\mathbb V^{-1}}
\le
C R
\left\|
\frac1{\tau}
\int_{\mathfrak t_n^R}^{\mathfrak t_{n+1}^R}
\left(
{\bf Y}(t)-{\bf Y}^R_{n+\frac12}
\right)\,{\rm d}t
\right\|_{\mathbb L^2}.
\]
By splitting \({\bf Y}={\bf y}+\Phi W\), the deterministic part is the
trapezoidal error for \({\bf y}\). Since
\[
{\bf y}(\cdot\wedge\mathfrak t_R)
\in
L^2\big(\Omega;C^{1,1/2-\varepsilon}([0,T];L^2)\big),
\]
this part is bounded by \(C\tau^{3/2-\varepsilon}\) in
\(L^2(\Omega)\) due to Lemma~\ref{lem:deltastar-average}. The stochastic part is exactly the
Brownian quadrature error
\[
\Phi(\mathcal Q_{n,R}^{W}-\mathcal I_{n,R}^{W}),
\]
which satisfies the following uniform estimate with the help of ~\eqref{eq:BM-quadrature-sup}:
\[\mathbb{E}\bigg[\sup_{0\le n\le N-1}\|\Phi(\mathcal Q_{n,R}^{W}-\mathcal I_{n,R}^{W})\|_{\mathbb{L}^2}^2\bigg]\le \mathbb{E}\big[\sup_{0\le n\le N-1}\|\Phi(\mathcal Q_{n}^{W}-\mathcal I_{n}^{W})\|_{\mathbb{L}^2}^2\big]\le C_{\Phi,\varepsilon}\tau^{3-\varepsilon}\]
Hence
\begin{equation}\label{eq:T5-sup}
\mathbb E\left[
\sup_{0\le n\le N-1}
\left\|
\mathcal T_{5,R,n}
\right\|_{\mathbb V^{-1}}^2
\right]
\le
C_{\varepsilon,R}\tau^{3-\varepsilon}.
\end{equation}
We now estimate \(\mathcal T_{6,R,n}\). We have
\[
\left<
\mathcal T_{6,R,n},\boldsymbol{\varphi}
\right>
=
-
\mathcal C
\left(
(\bar{\bf y}^R_{\star,n+\frac12}+\Phi\mathcal I_{n,R}^{W})
-\bar{\bf Y}^R_{n+\frac12},
{\bf Y}^R_{n+\frac12},
\boldsymbol{\varphi}
\right).
\]
By using again the localization bound for \({\bf Y}_{n+\frac12}\), we obtain
\[
\left\|
\mathcal T_{6,R,n}
\right\|_{\mathbb V^{-1}}
\le
C R
\left\|
(\bar{\bf y}^R_{\star,n+\frac12}-\bar{\bf y}^R_{n+\frac12})
+
\Phi(\mathcal I_{n,R}^{W}-\mathcal Q_{n,R}^{W})
\right\|_{\mathbb L^2}.
\]
For \(n\ge1\), by Lemma~\ref{lem:deltastar-average} the extrapolation error satisfies
\[
\left\|
\bar{\bf y}^R_{\star,n+\frac12}-\bar{\bf y}^R_{n+\frac12}
\right\|_{\mathbb L^2}
\le
C\tau^{3/2-\varepsilon}
\|{\bf y}(\cdot\wedge\mathfrak t_R)\|_{C^{1,1/2-\varepsilon}([0,T];\mathbb L^2)}.
\]
Together with the Brownian quadrature estimate~\eqref{eq:BM-quadrature-sup}, this gives
\begin{equation}\label{eq:T6-sup}
\mathbb E\left[
\sup_{1\le n\le N-1}
\left\|
\mathcal T_{6,R,n}
\right\|_{\mathbb V^{-1}}^2
\right]
\le
C_{\varepsilon,R}\tau^{3-\varepsilon}.
\end{equation}
For \(n=0\), the extrapolation starts with
\({\bf y}(t_{-1})={\bf y}(t_0)\), and hence only the first-order estimate
\[
\left\|
\bar{\bf y}_{\star,\frac12}^R-\bar{\bf y}_{\frac12}^R
\right\|_{\mathbb L^2(\mt)^3}
\le
C\tau
\|{\bf y}(\cdot\wedge\mathfrak t_R)\|_{C^1([0,T];\mathbb L^2)}
\]
is available. Therefore,
\begin{equation}\label{eq:T6-first}
\mathbb E\left[
\left\|
\mathcal T_{6,R,0}
\right\|_{\mathbb V^{-1}}^2
\right]
\le
C_R\tau^2.
\end{equation}
Finally, we consider \(\mathcal T_{7,R,n}\). By definition,
\[
\left\|
\mathcal T_{7,R,n}
\right\|_{\mathbb V^{-1}}
\le
\nu
\left\|
\frac1{\tau}
\int_{\mathfrak t_n^R}^{\mathfrak t_{n+1}^R}
\left(
\nabla{\bf Y}(t)-\nabla{\bf Y}^R_{n+\frac12}
\right)
\,{\rm d}t
\right\|_{\mathbb L^2}.
\]
The deterministic part is again a trapezoidal error, now at gradient
level, and the stochastic part is the Brownian quadrature error~\eqref{eq:BM-quadrature-sup} for
\(\nabla\Phi W\). By the assumed
\(C^{1,1/2-\varepsilon}\)-regularity of
\({\bf y}(\cdot\wedge\mathfrak t_R)\) and the uniform Brownian quadrature
estimate~\eqref{eq:BM-quadrature-sup}, we get
\begin{equation}\label{eq:T7-sup}
\mathbb E\left[
\sup_{0\le n\le N-1}
\left\|
\mathcal T_{7,R,n}
\right\|_{\mathbb V^{-1}}^2
\right]
\le
C_{\varepsilon,R}\tau^{3-\varepsilon}.
\end{equation}
By combining
\eqref{eq:T3T4-sup},
\eqref{eq:T5-sup},
\eqref{eq:T6-sup}, and
\eqref{eq:T7-sup}, and using
\[
\widetilde{\mathcal R}_{n+\frac12}^R
=
\sum_{\substack{i=1\\ i\neq2}}^7
\mathcal T_{i,R,n},
\]
we obtain

\[
\mathbb E\left[
\tau\sum_{n=0}^{N-1}
\left\|
\widetilde{\mathcal R}_{n+\frac12}^R
\right\|_{\mathbb V^{-1}}^2
\right]
\le
C_R\tau^3
+
C_{\varepsilon,R}\tau^{3-\varepsilon}
\le
C_{\varepsilon,R}\tau^{3-\varepsilon}.
\]
This proves \eqref{eq:R-bound-sum} and completes the proof.
\end{proof}
\subsubsection{Error inequality}

We now derive the localized error inequality in conservative form and
combine it with Lemma~\ref{lem:R-bound}. For
\(n=0,1,\dots,N-1\), we define
\[
{\bf e}_n^R:={\bf y}(\mathfrak t_n^R)-{\bf y}_n^R,
\qquad
{\bf e}_{n+\frac12}^R
:=\frac12\big({\bf e}_{n+1}^R+{\bf e}_n^R\big).
\]
For \(n\ge1\), we set
\[
{\bf e}_{\star,n+\frac12}^R
:=\frac32{\bf e}_n^R-\frac12{\bf e}_{n-1}^R.
\]
For \(n=0\), we use the convention
\[
{\bf e}_{-1}^R:={\bf e}_0^R,
\qquad
{\bf e}_{\star,\frac12}^R
:=\frac32{\bf e}_0^R-\frac12{\bf e}_{-1}^R
={\bf e}_0^R.
\]

\begin{lemma}[Error inequality]\label{lem:balanced-local}
 There exists a constant
\(C>0\), independent of \(n\) and \(\tau\), but possibly depending
on \(R,\nu\), and \(T\), such that for all $n\in \{0,\dots,N-1\}$
\begin{align}\label{eq:balanced-local}
\begin{aligned}
&\frac12\,
\mathbb E\bigg[
\sup_{0\le i\le n}
\|{\bf e}_{i+1}^R\|_{\mathbb L^2}^2
\bigg]
+
\frac{\nu\tau}{2}\,
\mathbb E\bigg[
\sup_{0\le k\le n}
\sum_{i=0}^{k}
\|\nabla{\bf e}_{i+\frac12}^R\|_{\mathbb L^2}^2
\bigg]
\\
&\qquad\le
C\,
\tau\sum_{i=0}^{N-1}\mathbb E\bigg[\|\widetilde{\mathcal R}_{i+\frac12}^R\|_{\mathbb V^{-1}}^2
\bigg]+
C\,
\mathbb E\bigg[
\sup_{0\le i\le N-1}
\|\mathcal I_{i}^{W^2}-\mathcal Q_{i}^{W^2}\|_{\mathbb L^{2}}^2
\bigg]
\\
&\qquad\quad
+
C(R)\,\tau
\sum_{k=0}^{n-1}
\mathbb E\bigg[
\sup_{0\le i\le k}
\|{\bf e}_{i+1}^R\|_{\mathbb L^2}^2
\bigg].
\end{aligned}
\end{align}
\end{lemma}

\begin{proof}
By arguing as in the proof of \cite[Lemma~3.2]{BBCP}, we obtain for fixed $n\in\{0,1,\dots,N-1\}$
\begin{align}\label{eq:err-local}
\begin{aligned}
\frac1{2\tau}
\left(
\|{\bf e}_{n+1}^R\|_{\mathbb L^2}^2
-
\|{\bf e}_n^R\|_{\mathbb L^2}^2
\right)
&+
\nu
\|\nabla{\bf e}_{n+\frac12}^R\|_{\mathbb L^2}^2
=
-\big\langle
\widetilde{\mathcal R}_{n+\frac12}^R,
{\bf e}_{n+\frac12}^R
\big\rangle
-
\big\langle
\mathcal N_{n+\frac12}^R,
{\bf e}_{n+\frac12}^R
\big\rangle
\\
&\quad
-
\mathcal E_{\rm conv}^{\,n+\frac12}
({\bf e}_{n+\frac12}^R),
\end{aligned}
\end{align}
where \[
\left\langle
\mathcal N_{n+\frac12}^R,\varphi
\right\rangle
:=
\left(
\mathcal I_{n,R}^{W^2}
-
\mathcal Q_{n,R}^{W^2},
\nabla\varphi
\right).
\]
\noindent
We now estimate each term on the right-hand side. First, by duality and Young's inequality, for every
\(\delta\in(0,\nu)\),
\begin{align}\label{eq:Rtilde-bound}
\big|
\big\langle
\widetilde{\mathcal R}_{n+\frac12}^R,
{\bf e}_{n+\frac12}^R
\big\rangle
\big|
&\le
\|\widetilde{\mathcal R}_{n+\frac12}^R\|_{\mathbb V^{-1}}
\,
\|\nabla{\bf e}_{n+\frac12}^R\|_{\mathbb L^2}
\nonumber\\
&\le
\frac{\delta}{4}
\|\nabla{\bf e}_{n+\frac12}^R\|_{\mathbb L^2}^2
+
C(\delta)
\|\widetilde{\mathcal R}_{n+\frac12}^R\|_{\mathbb V^{-1}}^2.
\end{align}
Similarly,
\begin{align}\label{eq:N-bound}
\big|
\big\langle
\mathcal N_{n+\frac12}^R,
{\bf e}_{n+\frac12}^R
\big\rangle
\big|
&\le
\|\mathcal I_{n,R}^{W^2}-\mathcal Q_{n,R}^{W^2}\|_{\mathbb L^{2}}
\,
\|\nabla{\bf e}_{n+\frac12}^R\|_{\mathbb L^2}
\nonumber\\
&\le
\frac{\delta}{4}
\|\nabla{\bf e}_{n+\frac12}^R\|_{\mathbb L^2}^2
+
C(\delta)
\|\mathcal I_{n,R}^{W^2}-\mathcal Q_{n,R}^{W^2}\|_{\mathbb L^{2}}^2.
\end{align}
It remains to estimate the convection defect. By definition,
\[
\mathcal E_{\rm conv}^{\,n+\frac12}
({\bf e}_{n+\frac12}^R)
=
\mathcal C(
{\bf e}_{\star,n+\frac12}^R,
{\bf Y}_{n+\frac12}^R,
{\bf e}_{n+\frac12}^R
).
\]
By using the continuity estimate for the trilinear form and the localization
bound
\[
\|{\bf Y}_{n+\frac12}^R\|_{\mathbb H^2}\le C R,
\]

\noindent
we obtain
\begin{align}\label{eq:Econv-bound-local}
\big|
\mathcal E_{\rm conv}^{\,n+\frac12}
({\bf e}_{n+\frac12}^R)
\big|
&\le
C\,
\|{\bf e}_{\star,n+\frac12}^R\|_{\mathbb L^2}
\,
\|{\bf Y}_{n+\frac12}^R\|_{\mathbb H^2}
\,
\|\nabla{\bf e}_{n+\frac12}^R\|_{\mathbb L^2}
\nonumber\\
&\le
C R\,
\|{\bf e}_{\star,n+\frac12}^R\|_{\mathbb L^2}
\,
\|\nabla{\bf e}_{n+\frac12}^R\|_{\mathbb L^2}.
\end{align}
Hence, again by Young's inequality,
\begin{equation}\label{eq:Econv-Young}
\big|
\mathcal E_{\rm conv}^{\,n+\frac12}
({\bf e}_{n+\frac12}^R)
\big|
\le
\frac{\delta}{4}
\|\nabla{\bf e}_{n+\frac12}^R\|_{\mathbb L^2}^2
+
C(\delta)R^2
\|{\bf e}_{\star,n+\frac12}^R\|_{\mathbb L^2}^2.
\end{equation}
Moreover, by the definition of \({\bf e}_{\star,n+\frac12}^R\),
\[
\|{\bf e}_{\star,n+\frac12}^R\|_{\mathbb L^2}^2
\le
C
\left(
\|{\bf e}_n^R\|_{\mathbb L^2}^2
+
\|{\bf e}_{n-1}^R\|_{\mathbb L^2}^2
\right),
\]
with the convention \({\bf e}_{-1}^R={\bf e}_0^R\). Therefore,
\begin{equation}\label{eq:estar-bound}
\|{\bf e}_{\star,n+\frac12}^R\|_{\mathbb L^2}^2
\le
C
\sup_{0\le i\le n}
\|{\bf e}_i^R\|_{\mathbb L^2}^2.
\end{equation}
By combining \eqref{eq:err-local}--\eqref{eq:estar-bound}, and choosing
\(\delta\in(0,\nu)\), we absorb the gradient terms into the left-hand side.
Thus
\begin{align}\label{eq:one-step-local}
\begin{aligned}
\frac1{2\tau}
\left(
\|{\bf e}_{n+1}^R\|_{\mathbb L^2}^2
-
\|{\bf e}_n^R\|_{\mathbb L^2}^2
\right)
&+
\frac{\nu}{2}
\|\nabla{\bf e}_{n+\frac12}^R\|_{\mathbb L^2}^2\le
C
\|\widetilde{\mathcal R}_{n+\frac12}^R\|_{\mathbb V^{-1}}^2
\\
&\quad
+
C
\|\mathcal I_{n,R}^{W^2}-\mathcal Q_{n,R}^{W^2}\|_{\mathbb L^{2}}^2
+
C(R)
\sup_{0\le i\le n}
\|{\bf e}_i^R\|_{\mathbb L^2}^2 .
\end{aligned}
\end{align}
Multiplying \eqref{eq:one-step-local} by \(2\tau\), we get
\begin{align}\label{eq:one-step-local-tau}
\begin{aligned}
\|{\bf e}_{n+1}^R\|_{\mathbb L^2}^2
-
\|{\bf e}_n^R\|_{\mathbb L^2}^2
+
\nu\tau
\|\nabla{\bf e}_{n+\frac12}^R\|_{\mathbb L^2}^2&\le
C\tau
\|\widetilde{\mathcal R}_{n+\frac12}^R\|_{\mathbb V^{-1}}^2
+
C\tau
\|\mathcal I_{n,R}^{W^2}-\mathcal Q_{n,R}^{W^2}\|_{\mathbb L^2}^2
\\
&\quad
+
C(R)\tau
\sup_{0\le i\le n}
\|{\bf e}_i^R\|_{\mathbb L^2}^2 .
\end{aligned}
\end{align}
Now let \(k\in\{0,\dots,n\}\). By summing
\eqref{eq:one-step-local-tau} from \(m=0\) to \(m=k\), and using
\({\bf e}_0^R=0\), gives
\begin{align}\label{eq:sum-local-k}
\begin{aligned}
\|{\bf e}_{k+1}^R\|_{\mathbb L^2}^2
&+
\nu\tau
\sum_{m=0}^{k}
\|\nabla{\bf e}_{m+\frac12}^R\|_{\mathbb L^2}^2
\le
C\tau
\sum_{m=0}^{k}
\|\widetilde{\mathcal R}_{m+\frac12}^R\|_{\mathbb V^{-1}}^2
\\
&\quad
+
C\tau
\sum_{m=0}^{k}
\|\mathcal I_{m,R}^{W^2}-\mathcal Q_{m,R}^{W^2}\|_{\mathbb L^{2}}^2
+
C(R)\tau
\sum_{m=0}^{k}
\sup_{0\le i\le m}
\|{\bf e}_i^R\|_{\mathbb L^2}^2 .
\end{aligned}
\end{align}
Taking the supremum over \(k=0,\dots,n\), we obtain
\begin{align}\label{eq:sup-local-k}
\begin{aligned}
\sup_{0\le k\le n}
\|{\bf e}_{k+1}^R\|_{\mathbb L^2}^2
&+
\nu\tau
\sup_{0\le k\le n}
\sum_{m=0}^{k}
\|\nabla{\bf e}_{m+\frac12}^R\|_{\mathbb L^2}^2\le
\tau
\sum_{m=0}^{N-1}
\|\widetilde{\mathcal R}_{m+\frac12}^R\|_{\mathbb V^{-1}}^2
\\
&\quad
+
C\tau n
\sup_{0\le m\le n}
\|\mathcal I_{m}^{W^2}-\mathcal Q_{m}^{W^2}\|_{\mathbb L^{2}}^2
+
C(R)\tau
\sum_{m=0}^{n-1}
\sup_{0\le i\le m}
\|{\bf e}_i^R\|_{\mathbb L^2}^2 .
\end{aligned}
\end{align}
Since \(n\tau\le T\), the factors \(\tau n\) are bounded by \(T\).
Absorbing this into the constants, taking expectations, and reducing the
coefficient on the left-hand side if necessary, we arrive at
\eqref{eq:balanced-local}. This completes the proof.
\end{proof}
\subsubsection{Error estimate}
By using Lemmas~\ref{lem:R-bound} and ~\ref{lem:balanced-local}, and the
uniform Brownian quadrature estimates \eqref{eq:triple-BM-sup}, we obtain, for every \(n=0,\dots,N-1\),
\[
\begin{aligned}
&\mathbb E\bigg[
\sup_{0\le i\le n}
\|{\bf e}_{i+1}^R\|_{\mathbb L^2}^2
\bigg]
+
\nu\tau\,
\mathbb E\bigg[
\sup_{0\le k\le n}
\sum_{i=0}^{k}
\|\nabla{\bf e}_{i+\frac12}^R\|_{\mathbb L^2}^2
\bigg]
\\
&\qquad\le
C_{\varepsilon,R}\tau^{3-\varepsilon}
+
C(R)\tau
\sum_{k=0}^{n-1}
\mathbb E\bigg[
\sup_{0\le i\le k}
\|{\bf e}_{i+1}^R\|_{\mathbb L^2}^2
\bigg].
\end{aligned}
\]
An application of the discrete Gronwall lemma gives
\[
\mathbb E\bigg[
\sup_{0\le i\le n}
\|{\bf e}_{i+1}^R\|_{\mathbb L^2}^2
\bigg]
\le
C_{\varepsilon,R,T}\tau^{3-\varepsilon}.
\]
Substituting this estimate back into the previous inequality yields
\[
\nu\tau\,
\mathbb E\bigg[
\sup_{0\le k\le n}
\sum_{i=0}^{k}
\|\nabla{\bf e}_{i+\frac12}^R\|_{\mathbb L^2}^2
\bigg]
\le
C_{\varepsilon,R,T}\tau^{3-\varepsilon}.
\]
By taking \(n=N-1\), we conclude that
\begin{align}\label{eq:final-local-rate}
\begin{aligned}
&\mathbb E\bigg[
\sup_{0\le i\le N-1}
\|{\bf e}_{i+1}^R\|_{\mathbb L^2}^2
\bigg]
+
\nu\tau\,
\mathbb E\bigg[
\sup_{0\le k\le N-1}
\sum_{i=0}^{k}
\|\nabla{\bf e}_{i+\frac12}^R\|_{\mathbb L^2}^2
\bigg]\le
C_{\varepsilon,R,T}\tau^{3-\varepsilon}.
\end{aligned}
\end{align}
This completes the proof of Theorem~\ref{thm:main-local}.
\subsection{Strong rate of convergence for the pressure}
\label{subsec:pressure-rate}

We next derive the corresponding localized estimate for the pressure.
As in the velocity estimate, the argument is performed for the stopped
three-dimensional solution.  The pressure is not compared pointwise in
time; instead, we compare the discrete pressure with the time average of
the exact pressure on each time interval.

\noindent
We use the continuous inf--sup condition: there exists a constant
\(\beta>0\) such that
\begin{equation}\label{eq:inf-sup-cont-3d}
   \beta \|q\|_{\mathbb L^2}
   \le
   \sup_{{\bf v}\in\mathbb V\setminus\{0\}}
   \frac{(q,\diver{\bf v})}
        {\|{\bf v}\|_{\mathbb H^1}},
   \qquad q\in \Qspace .
\end{equation}

\noindent
Let \(\mathfrak t_R^{\tt d}\) be the discrete stopping time introduced in
\eqref{eq:tRd-1}.  For the pressure estimate, following
\cite[Section~3.4]{BBCP}, we introduce a localized probability set which
controls only the discrete advecting velocity appearing in the linearized
convective term.  Namely, for \(R>0\), define
\begin{equation}\label{eq:pressure-local-set}
   \Omega_R^\tau
   :=
   \left\{\omega\in\Omega: 
      \sup_{0\le n\le N-1}
      \left\|
         {\bf y}_{\star,n+\frac12}^R+\Phi\mathcal I_{n,R}^{W}
      \right\|_{\mathbb V}^{2}
      \le R
   \right\},
   \qquad
   \mathcal B_R^\tau:=\mathbf 1_{\Omega_R^\tau}.
\end{equation}
Equivalently, we shall use the pathwise bound
\begin{equation}\label{eq:pressure-pathwise-bound}
   \sup_{0\le n\le N-1}
   \mathbf 1_{\Omega_R^\tau}
   \left\|
      {\bf y}_{\star,n+\frac12}^R+\Phi\mathcal I_{n,R}^{W}
   \right\|_{\mathbb V}^{2}
   \le R.
\end{equation}
The localized pressure error is
\begin{equation}\label{eq:pressure-error-stopped-3d}
   \pi_{n+1}^R
   :=
   \mathbf 1_{\Omega_R^\tau}
   \left(
      \frac1{\tau}\int_{t_n}^{t_{n+1}}p(t)\,\mathrm dt
      -p_{n+1}
   \right).
\end{equation}

\begin{lemma}[Discrete time derivative estimate]
\label{lem:discrete-time-derivative-3d}
Under the assumptions of Theorem~\ref{thm:main-local}, for every
\(\varepsilon>0\) and every \(R>0\), there exists a constant
\(C_{\varepsilon,R,T}>0\), independent of \(\tau\) and \(N\), such that
\begin{equation}\label{eq:discrete-time-derivative-3d}
   \tau\sum_{n=0}^{N-1}
   \mathbb E\left[
      \mathbf 1_{\Omega_R^\tau}
      \left\|
         \frac{{\bf e}_{n+1}^R-{\bf e}_n^R}{\tau}
      \right\|_{\Vdual}^2
   \right]
   \le
   C_{\varepsilon,R,T}\tau^{3-\varepsilon}.
\end{equation}
\end{lemma}

\begin{proof}
The proof follows \cite[Lemma~3.3]{BBCP}.  The only modification is that
the two-dimensional localization is replaced by
\(\mathbf 1_{\Omega_R^\tau}\).  The advecting velocity is controlled pathwise by
\eqref{eq:pressure-pathwise-bound}, while the remaining terms are
controlled by the stopped velocity estimate \eqref{eq:main-local} and the
residual estimate \eqref{eq:R-bound-sum}.  This proves
\eqref{eq:discrete-time-derivative-3d}.
\end{proof}

\begin{theorem}[Localised pressure error estimate]
\label{thm:pressure-local-3d}
Let the assumptions of Theorem~\ref{thm:main-local} hold.  Then, for every
\(\varepsilon>0\) and every \(R>0\), there exists a constant
\(C_{\varepsilon,R,T}>0\), independent of \(\tau\) and \(N\), such that
\begin{equation}\label{eq:pressure-rate-3d}
   \tau\sum_{n=0}^{N-1}
   \mathbb E\left[
      \|\pi_{n+1}^R\|_{\mathbb L^2}^2
   \right]
   \le
   C_{\varepsilon,R,T}\tau^{3-\varepsilon}.
\end{equation}
\end{theorem}

\begin{proof}
The proof follows the inf--sup argument of \cite[Theorem~3.2]{BBCP}.  On
the localized set \(\Omega_R^\tau\), the pressure error is bounded by the
discrete time-derivative error, the stopped velocity error terms, and the
modified residual. These are controlled by
Lemma~\ref{lem:discrete-time-derivative-3d},
Theorem~\ref{thm:main-local}, and the residual estimate~\eqref{eq:R-bound-sum}.  This yields \eqref{eq:pressure-rate-3d}.
\end{proof}
\section{Simulations}\label{sec: simulations}
To approximate the random PDE \eqref{req:stochastic-NS}, we employ fully discrete finite element 
methods based on both the semi-implicit Euler and Crank-Nicolson time 
discretisations as discussed in Sections~\ref{sec:IE} and \ref{theo-2}, respectively.
Here we also allow a deterministic forcing $\mathbf f$ which can depend on space and time. The computational 
domain is taken to be the unit cube $\mathcal{O} = ¢(0,1)^3$, which is partitioned 
into a uniform tetrahedral mesh with mesh size $h = 1/L$. The mesh is constructed 
by dividing $\mathcal{O}$ into $L \times L \times L$ congruent sub-cubes and 
subsequently decomposing each sub-cube into tetrahedra. More precisely, let 
$\mathcal{T}_h$ denote the resulting quasi-uniform, shape-regular triangulation into tetrahedra of 
$\mathcal{O}$, with maximal mesh size
\[
h = \max_{K \in \mathcal{T}_h} \operatorname{diam}(K).
\]
For the spatial discretisation, we employ the lowest-order Taylor-Hood finite 
element pair, which is known to be inf-sup stable; see, 
e.g., \cite{GR}. Denoting by $P_\ell(K)$ the 
space of polynomials of degree at most $\ell$ on an element $K \in \mathcal{T}_h$, 
the discrete velocity and pressure spaces are defined as
\[
\mathbb{V}_h
:=
\bigl\{
\mathbf{v}_h \in H_0^1(\mathcal{O}) :
\mathbf{v}_h|_K \in P_2(K),
\ \forall\, K \in \mathcal{T}_h
\bigr\},
\]
\[
\mathbb{Q}_h
:=
\bigl\{
q_h \in L_0^2(\mathcal{O}) :
q_h|_K \in P_1(K),
\ \forall\, K \in \mathcal{T}_h
\bigr\},
\]
where $L_0^2(\mathcal{O}) := \{ q \in L^2(\mathcal{O}) : \int_{\mathcal{O}} q\, dx = 0 \}$ and $H^1_0(\mathcal O)$ denotes the space of $H^1$-functions over $\mathcal O$ vanishing at the boundary. The pair 
$(\mathbb{V}_h, \mathbb{Q}_h)$ is called stable if the so-called Ladyshenskaya-Babuska-Brezzi (LLB) condition is satisfied i.e., there exists a constant $\varpi > 0$ (independent of $h$) such that 
\[
\inf_{q_h \in \mathbb{Q}_h \setminus \{0\}}
\sup_{\mathbf{v}_h \in \mathbb{V}_h \setminus \{0\}}
\frac{\int_{\mathcal{O}}\text{div}\mathbf{v}_h q_h\,dx}{\|\nabla \mathbf{v}_h\|_{\mathbb{L}^2}\,\|q_h\|_{\mathbb{L}^2}}
\geq \varpi > 0,
\]
The LLB condition is a necessary condition for showing the well-posedness of full discrete formulation. In its absence, it is impossible to uniquely reconstruct the pressure once the velocity has been constructed.

The fully discrete approximations corresponding to the semi-implicit Euler (IE) is given as follows: For $m=1,\ldots,N$, given
$ \mathbf{y}_h^{m-1}\in \mathbb V_h$, find
$(\mathbf{y}_h^{m},p_h^{m})\in \mathbb V_h\times \mathbb Q_h$
such that, for all
$(\boldsymbol{\varphi},q)\in \mathbb V_h\times \mathbb Q_h$,
\begin{equation}
\label{eq:IE-FEM}
\begin{cases}
\displaystyle
\left(
\frac{\mathbf{y}_h^{m}-\mathbf{y}_h^{m-1}}{\tau},
\boldsymbol{\varphi}
\right)
+
\mathcal{C}^{*}
\!\left(\mathbf{y}_h^{m-1}
+\Phi(W^{m-1}),
\mathbf{y}_h^{m}
+\Phi(W^{m}),
\boldsymbol{\varphi}
\right)
\\[1ex]
\displaystyle\qquad
+\,
\nu
\left(
\nabla\bigl(\mathbf{y}_h^{m}
+\Phi(W^{m})\bigr),
\nabla\boldsymbol{\varphi}
\right)
-
\left(p_h^{m},
\text{div}\,\boldsymbol{\varphi}\right)
=
\left(\mathbf{f}(t_{m}),
\boldsymbol{\varphi}\right),
\\[2ex]
\displaystyle
\left(\text{div}\,\mathbf{y}_h^{m},q\right)=0.
\end{cases}
\end{equation}
where $
W^{m}=W(t_m),$ and 
nonlinear convection term is discretised using its skew-symmetric form,
\[
\mathcal{C}^{*}(\mathbf{u},\mathbf{v},\boldsymbol{\varphi})
=
\frac12\mathcal{C}(\mathbf{u},\mathbf{v},\boldsymbol{\varphi})
-
\frac12\mathcal{C}(\mathbf{v},\mathbf{u},\boldsymbol{\varphi})
\]
thanks to the divergence-free velocity fields and ensures that the fully discrete formulation satisfies the appropriate discrete energy balance.
The fully discrete approximation corresponding to Crank-Nicolson (CN) scheme is given as follow: For $m=1,\ldots,N$, given
$\mathbf{y}_h^{m-1},\mathbf{y}_h^{m-2}\in \mathbb V_h$, find
$(\mathbf{y}_h^{m},p_h^{m})\in \mathbb V_h\times \mathbb Q_h$
such that, for all
$(\boldsymbol{\varphi},q)\in \mathbb V_h\times \mathbb Q_h$,
\begin{equation}\label{eq:CN-FEM}
\begin{cases}
\displaystyle
\left(
\frac{\mathbf{y}_h^{m}-\mathbf{y}_h^{m-1}}{\tau},
\boldsymbol{\varphi}
\right)
+
\mathcal{C}\!\left(
\mathbf{y}_{h,\star}^{m-\frac12}
+\Phi\mathcal{I}_{m-1}^W,
\mathbf{y}_h^{m-\frac12}
+\Phi\mathcal{I}^W_{m-1},
\boldsymbol{\varphi}
\right)
\\[1.5ex]
\displaystyle\qquad
-
\left(
\mathcal{I}_{m-1}^{W^2},\nabla\boldsymbol{\varphi}
\right)
+
\nu
\left(
\nabla\!\left(
\mathbf{y}^{m-\frac12}_h
+\Phi\mathcal{I}^W_{m-1}
\right),
\nabla\boldsymbol{\varphi}
\right)
-
\left(
p_h^{m},
\text{div}\,\boldsymbol{\varphi}
\right)
\\[1.5ex]
\displaystyle\qquad
=
\left(
\frac{1}{\tau}
\int_{t_{m-1}}^{t_{m}}
\mathbf{f}(s)\,ds,
\boldsymbol{\varphi}
\right),
\\[2ex]
\displaystyle
\left(
\text{div}\,\mathbf{y}_h^{m},
q
\right)=0.
\end{cases}
\end{equation}

To discretise the probability space, we employ a Monte Carlo approach with 
$N_s \in \mathbb{N}$ independent samples where both the temporal schemes require different types of stochastic increments. For the IE scheme, we use classical Wiener increments. For each sample 
$\ell = 1,\ldots, N_s$, the Brownian increments 
$(\Delta_m W(\omega_\ell))_{m=1}^{M}$ are approximated by
\[
Z_m^{\ell} = \sqrt{\tau}\,\xi_m^{\ell},
\qquad \xi_m^{\ell} \stackrel{\mathrm{i.i.d.}}{\sim} \mathcal{N}(0,1),
\qquad m = 1,\ldots,M,
\]
generated via a pseudo-random number generator. On the other hand, the CN scheme requires averaged Wiener increments to achieve its higher 
convergence rate of order. For each sample 
$\ell = 1,\ldots,N_s$, we define
\[
\widetilde{Z}_m^{\ell}
= \frac{1}{2}\bigl(\Delta_m W^{\ell} + \Delta_{m-1} W^{\ell}\bigr)
  + \frac{1}{\sqrt{12}}\bigl(\zeta_m^{\ell} - \zeta_{m-1}^{\ell}\bigr),
\qquad m = 1,\ldots,M,
\]
where $\Delta_m W^{\ell} = \sqrt{\tau}\,\xi_m^{\ell}$ are classical increments 
and $\zeta_m^{\ell} = \sqrt{\tau}\,\eta_m^{\ell}$ with 
$\eta_m^{\ell} \stackrel{\mathrm{i.i.d.}}{\sim} \mathcal{N}(0,1)$, independent 
of $(\xi_m^{\ell})$.

When comparing solutions at different temporal resolutions, coarse and fine 
approximations must be driven by the same pseudo-random event. Let 
$M_f = r M_c$ for some $r \in \mathbb{N}$, where $M_c = 2^k$ and 
$M_f = 2^K$ are the number of time steps at levels $k < K$, with step 
sizes $\tau_c = T/M_c$ and $\tau_f = T/M_f = \tau_c/r$, respectively. 
Increments are generated first at the finest level and all coarser 
levels are recovered via the reconstruction
\[
Z_m^{\ell,c} = \sum_{n=1}^{r} Z_{(m-1)r+n}^{\ell,f},
\qquad m = 1,\ldots,M_c, \quad \ell = 1,\ldots,N_s.
\]
Since $Z_{(m-1)r+n}^{\ell,f} \sim \mathcal{N}(0,\tau_f)$ independently, 
the sum satisfies $Z_m^{\ell,c} \sim \mathcal{N}(0,\tau_c)$, consistent 
with Brownian motion. This reconstruction is applied identically for both 
the IE and CN schemes. For CN, the Brownian bridge within each coarse 
interval is then constructed independently, conditioned on $Z_m^{\ell,c}$.
\subsection{Academic Experiment}
We demonstrate the convergence orders for the IE scheme~\eqref{eq:IE-FEM}
and the CN scheme~\eqref{eq:CN-FEM} in three spatial dimensions using an exact solution
\begin{align}
    \mathbf{y}^{\mathrm{ex}}(t,\mathbf{x})
        &= \bigl(2\cos(6t) + 4W_1(t)\bigr)\,\mathbf{g}(\mathbf{x}), \\
    p^{\mathrm{ex}}(t,\mathbf{x})
        &= t\bigl(|\mathbf{x}|^2 - 1\bigr),
\end{align}
where $\mathbf{g}(\mathbf{x}) = (x_1^3,\,-3x_1^2 x_2,\, x_1 x_2)^\top$ satisfies
$\text{div}\,\mathbf{g} = 0$ and the noise coefficient is
$\boldsymbol{\Phi} = 4\mathbf{g}$, so that
$\boldsymbol{\Phi} W_1(t) = 4\mathbf{g}(\mathbf{x})\,W_1(t)$.
We approximate~\eqref{req:stochastic-NS} on $(0,T)\times \mathcal{O}$ with $T = 1$ and the unit cube
$\mathcal{O} = (0,1)^3$, with right-hand side
\begin{equation}
    \mathbf{f}(t,\mathbf{x})
    = -12\sin(6t)\,\mathbf{g}
    + \bigl(2\cos(6t) + 4W_1(t)\bigr)^2 (\mathbf{g}\cdot\nabla)\mathbf{g}
    - \bigl(2\cos(6t) + 4W_1(t)\bigr)\Delta\mathbf{g}
    + t\,\nabla\bigl(|\mathbf{x}|^2 - 1\bigr),
\end{equation}
and Dirichlet boundary condition
$\mathbf{y}\big|_{\partial \mathcal{O}} = \mathbf{y}^{\mathrm{ex}}$.
The domain $\mathcal{O}$ is partitioned into a uniform tetrahedral mesh with mesh size $h = 1/16$.
Errors are estimated from $N_s = 1000$ independent realisations of the driving Wiener
process for time step sizes $\tau \in \{2^{-2}, 2^{-3}, 2^{-4}, 2^{-5}, 2^{-6}\}$.

In Figure~\ref{fig:exact_IE_combined}, we display the mean-square approximation errors
\begin{equation}\label{eq:errors}
    \mathbb{E}\!\Big[\max_{0\le n\le N}
        \bigl\|\mathbf{y}(t_n) - \mathbf{y}_n\bigr\|^2_{L^2(\mathcal{O})}\Big]
    \quad\text{and}\quad
    \tau\sum_{n=1}^{N}
    \mathbb{E}\!\left[\|\pi_n\|^2_{L^2(\mathcal{O})}\right]
\end{equation}
for the velocity and pressure of the IE scheme, and in Figure~\ref{fig:exact_CN_combined} the
corresponding errors for the CN scheme. The plots in Figures~\ref{fig:exact_IE_combined} and~\ref{fig:exact_CN_combined} validate the theoretical convergence
rates: the IE scheme achieves the expected rate of order $1$ for the velocity, and the
CN scheme attains the predicted rate of order $3/2$ for both velocity and pressure.
For the pressure approximation of the IE scheme, the plot suggests a rate that is
consistent (even better) with order $1$;
we attribute this to the smoothness of the exact pressure $p^{\mathrm{ex}}(t,\mathbf{x})
= t(|\mathbf{x}|^2 - 1)$ in this academic example.

In the two-dimensional setting of~\cite{BBCP}, the CN scheme yields significantly smaller
errors than the IE scheme: at the largest time step $\tau = 0.01$, the velocity error
of the CN scheme is approximately $2$ times smaller than that of the IE scheme,
with this advantage persisting across all computed time steps. However, in our present experiment, the situation is reversed: the IE scheme
produces smaller absolute errors throughout the computed range
$\tau \in \{2^{-2},\ldots,2^{-6}\}$, with the IE velocity error being approximately
$2.9$ times smaller at $\tau = 2^{-2}$ and $1.9$ times smaller at $\tau = 2^{-6}$.
This does not contradict either the theoretical convergence rates or the experimental
findings for both the schemes. Rather, it reflects the fact that the error constant
$C_{\mathrm{CN}}$ of the CN scheme is approximately $15$ times larger than the error constant 
$C_{\mathrm{IE}}$ of IE scheme, in this three-dimensional case in present setting . Consequently, the crossover
step size $
\tau^* = \left(\frac{C_{\mathrm{IE}}}{C_{\mathrm{CN}}}\right)^2 \approx 2^{-8},
$
lies just beyond our finest accessible time step. Thus, in three spatial
dimensions the computational cost prevents us from reaching $\tau^*$, so the IE
scheme appears superior over the range considered here.
\begin{figure}[htbp]
\centering
\begin{minipage}[t]{0.48\textwidth}
    \centering
    \includegraphics[width=\linewidth]{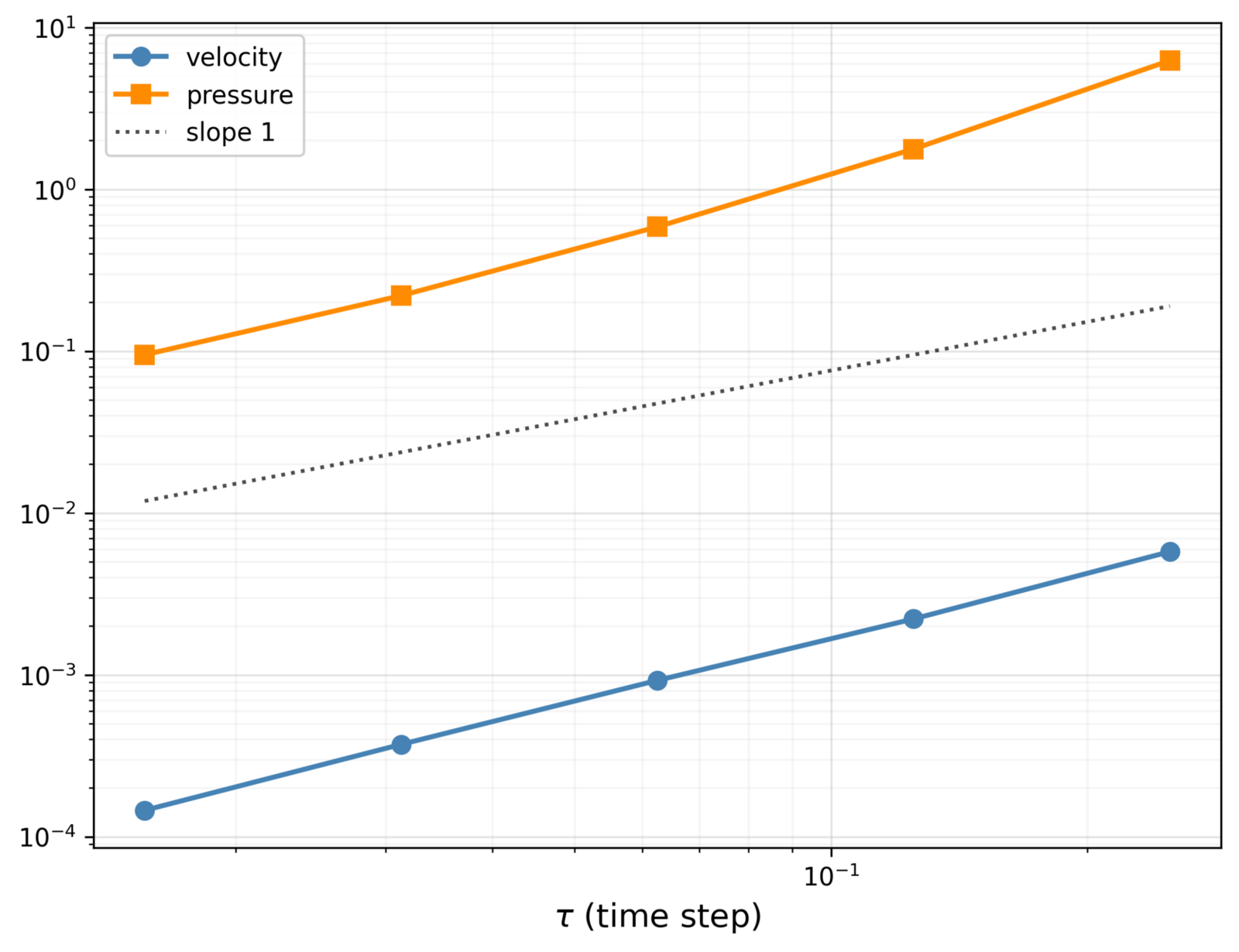}
    \captionsetup{width=1\textwidth}
    \caption{Academic Experiment: approximation error for IE scheme both velocity and pressure}
\label{fig:exact_IE_combined}
\end{minipage}
\hfill
\begin{minipage}[t]{0.48\textwidth}
    \centering
    \includegraphics[width=\linewidth]{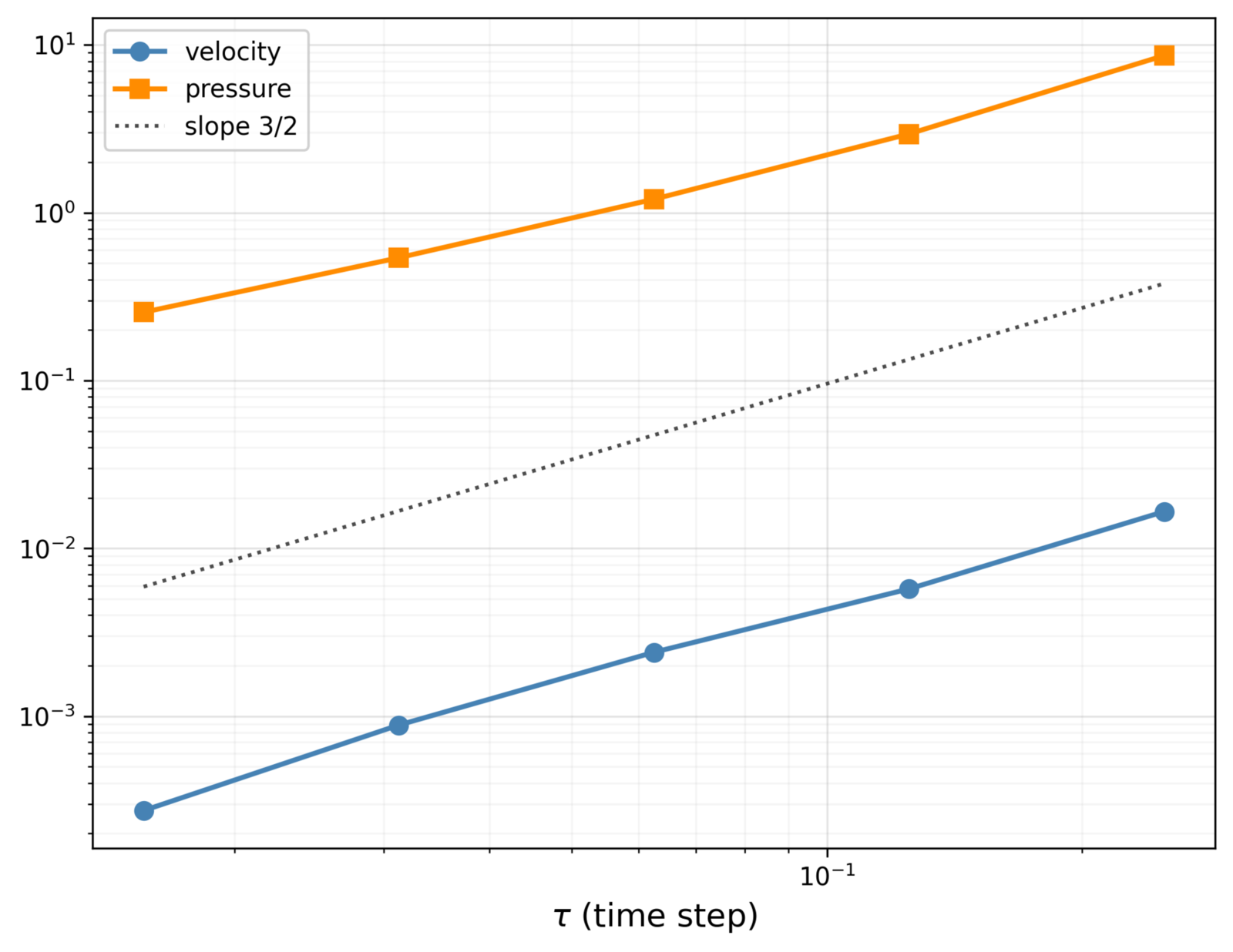}
    \captionsetup{width=1\textwidth}
    \caption{Academic Experiment: approximation error for CN scheme both velocity and pressure}
\label{fig:exact_CN_combined}
\end{minipage}
\end{figure}
\begin{table}[htbp]
\centering
\caption{Mean-square velocity errors
and experimental orders of convergence~(EOC) for the IE and CN schemes;
$h=1/16$, $N_s=1000$ realisations.}
\label{tab:vel_eoc}
\begin{tabular}{c rr rr}
\toprule
& \multicolumn{2}{c}{IE} & \multicolumn{2}{c}{CN} \\
\cmidrule(lr){2-3}\cmidrule(lr){4-5}
$\tau$ & \multicolumn{1}{c}{Error} & \multicolumn{1}{c}{EOC}
       & \multicolumn{1}{c}{Error} & \multicolumn{1}{c}{EOC} \\
\midrule
$2^{-2}$ & $5.77\times10^{-3}$ & $---$ & $1.66\times10^{-2}$ & $---$ \\
$2^{-3}$ & $2.22\times10^{-3}$ & $1.28$ & $5.74\times10^{-3}$ & $1.53$ \\
$2^{-4}$ & $9.24\times10^{-4}$ & $1.26$ & $2.41\times10^{-3}$ & $1.25$ \\
$2^{-5}$ & $3.72\times10^{-4}$ & $1.31$ & $8.86\times10^{-4}$ & $1.44$ \\
$2^{-6}$ & $1.46\times10^{-4}$ & $1.35$ & $2.74\times10^{-4}$ & $1.69$ \\
\bottomrule
\end{tabular}
\end{table}
\subsection{Lid-driven cavity flow in 3D}\label{sec:Lid-3d}
We test the long-time behaviour of schemes~\eqref{eq:IE-FEM}--\eqref{eq:CN-FEM} on
the three-dimensional lid-driven cavity problem.
The domain is $\mathcal{O}=(0,1)^3$; we take $u_0=0$, $\mathbf f=0$, and kinematic viscosity
$\nu=\tfrac{1}{500}$ ($\mathrm{Re}=500$).
Boundary conditions are $\mathbf u\big|_{\{z=1\}}=(1,0,0)^\top$ (moving lid) and
$\mathbf u=0$ on all remaining walls. The additive noise is given by
\[
  \Phi\,\mathrm{d}W(t)
  = \mu\!\sum_{c\in\mathcal{G}^3}\boldsymbol{\varphi}_c(x)\,\mathrm{d}\widetilde{\beta}_c(t),
  \qquad \mu\geq 0,
\]
where $\{\widetilde{\beta}_c\}_{c\in\mathcal{G}^3}$ are independent Brownian motions and
$\mathcal{G}=\bigl\{\tfrac{1}{8},\tfrac{3}{8},\tfrac{5}{8},\tfrac{7}{8}\bigr\}$,
giving $|\mathcal{G}|^3=64$ modes.
The modes $\varphi_c$ are constructed as the three-dimensional analogue of the
stream-function approach used in the two-dimensional case~\cite{BBCP} and \cite[Example 5.9]{BrPrWi}. In two dimensions one sets $\boldsymbol{\varphi}_c=\nabla^\perp\psi_c$ for a scalar
stream function $\psi_c$ whereas in three dimensions we employ the vector-potential identity
$\boldsymbol{\varphi}_c=\nabla\times (0,0,\psi_c)^\top$, so that
$\operatorname{div}\boldsymbol{\varphi}_c=0$ holds identically.
Setting $\lambda=\tfrac{1}{8}$, $\mathcal{G}_c=\bigl[x_0-\tfrac{\lambda}{2},x_0+\tfrac{\lambda}{2}\bigr]^3$
for $c=(x_0,y_0,z_0)\in\mathcal{G}^3$, and writing
$(\hat x,\hat y,\hat z)=(x-c)/\lambda+\tfrac{1}{2}$ for the rescaled local coordinates,
we take
\[
  \psi_c(\hat x,\hat y,\hat z)
  = C\,\hat x^2(1-\hat x)^2\,\hat y^2(1-\hat y)^2\,\hat z^2(1-\hat z)^2.
\]
An explicit computation of $\nabla\times(0,0,\psi_c)^\top$ then yields
\[
  \boldsymbol{\varphi}_c(x)
  = \mathbf{1}_{\mathcal{O}_c}(x)\;h(\hat z)
    \begin{pmatrix}
      \tilde{g}(\hat x,\hat y)\\[3pt]
     -\tilde{g}(\hat y,\hat x)\\[2pt]
      0
    \end{pmatrix},
  \quad
  \tilde{g}(s,r)=200\,s^2(1-s)^2\,r(r-1)(2r-1),
  \quad
  h(s)=s^2(1-s)^2.
\]
The factor $h(\hat z)$ ensures $\boldsymbol{\varphi}_c$ vanishes on $\{z=0\}$ and $\{z=1\}$, while
the $s^2(1-s)^2$ prefactors in $\tilde g$ enforce vanishing at the remaining walls;
together, $\boldsymbol{\varphi}_c$ is compactly supported in $\mathcal{O}_c\subset \mathcal{O}$ and satisfies the the boundary condition on $\partial \mathcal{O}$.
All results are computed on $(0,T]$ with $T \in \{20, 50, 100\}$, time step $\tau=0.05$, and a
uniform tetrahedral mesh with $h=1/16$ and with $N=500$ independent Monte Carlo samples.

\vspace{0.1cm}
\noindent\underline{$U$-velocity profile}: The $U$-velocity profile along the vertical centreline of the cavity has been
widely adopted as a benchmark quantity for the three-dimensional lid-driven
cavity problem; see \cite{FTTF, KHT} for the deterministic setup. The computed solutions from the present formulation (for $\nu= \{0,10,40\}$ at T=100) are shown in Figures \ref{fig:lidcavity_IE_mean_T100_centerline} and \ref{fig:lidcavity_CN_mean_centerline} match almost exactly with the deterministic case \cite{FTTF, KHT}.
\begin{figure}[htbp]
\centering
\begin{minipage}[t]{0.45\textwidth}
    \centering
    \includegraphics[width=\linewidth]{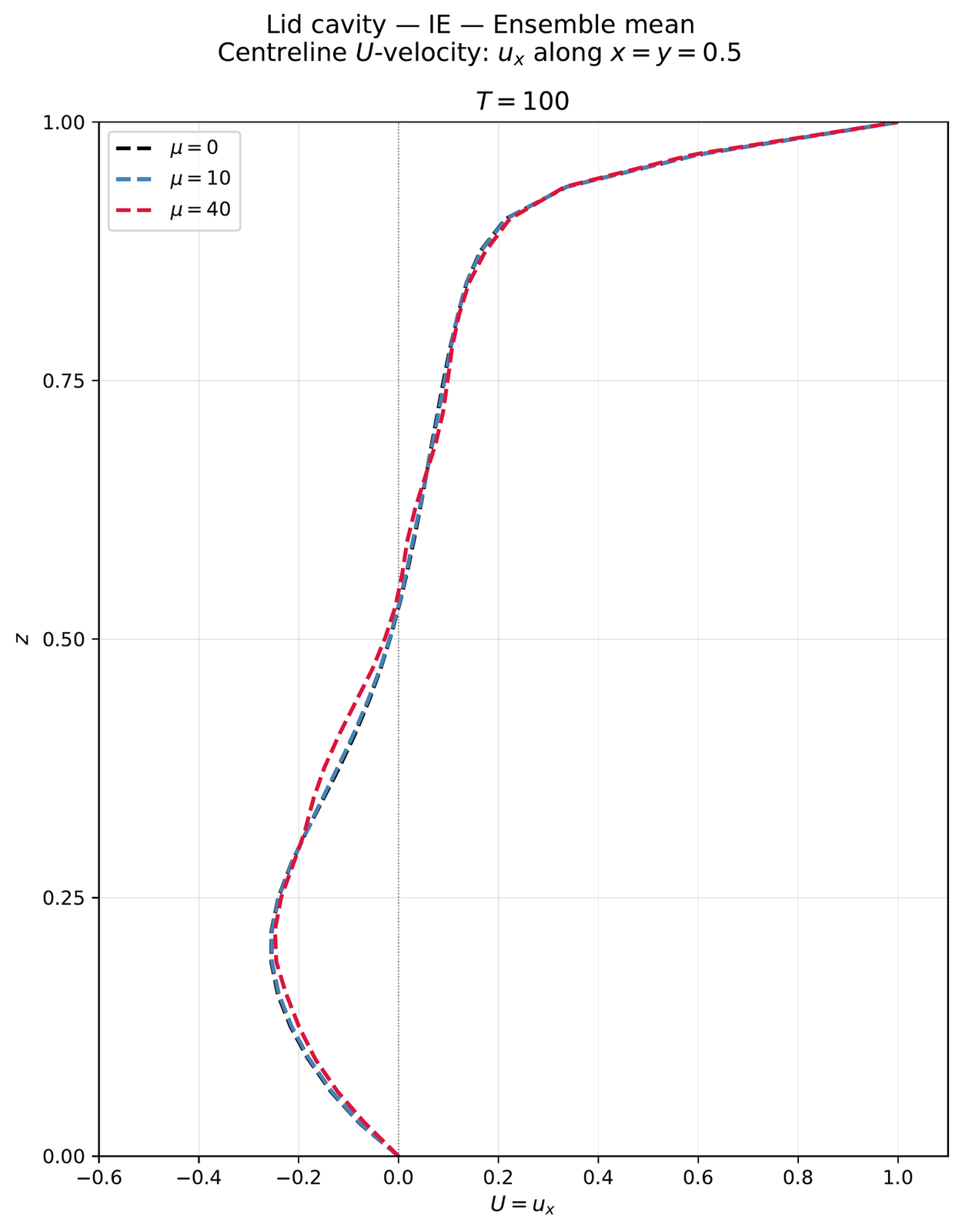}
    \captionsetup{width=1\textwidth}
    \caption{3D-mid plane centerline $U$-velocity: for IE scheme, Re=500}
\label{fig:lidcavity_IE_mean_T100_centerline}
\end{minipage}
\hfill
\begin{minipage}[t]{0.45\textwidth}
    \centering
    \includegraphics[width=\linewidth]{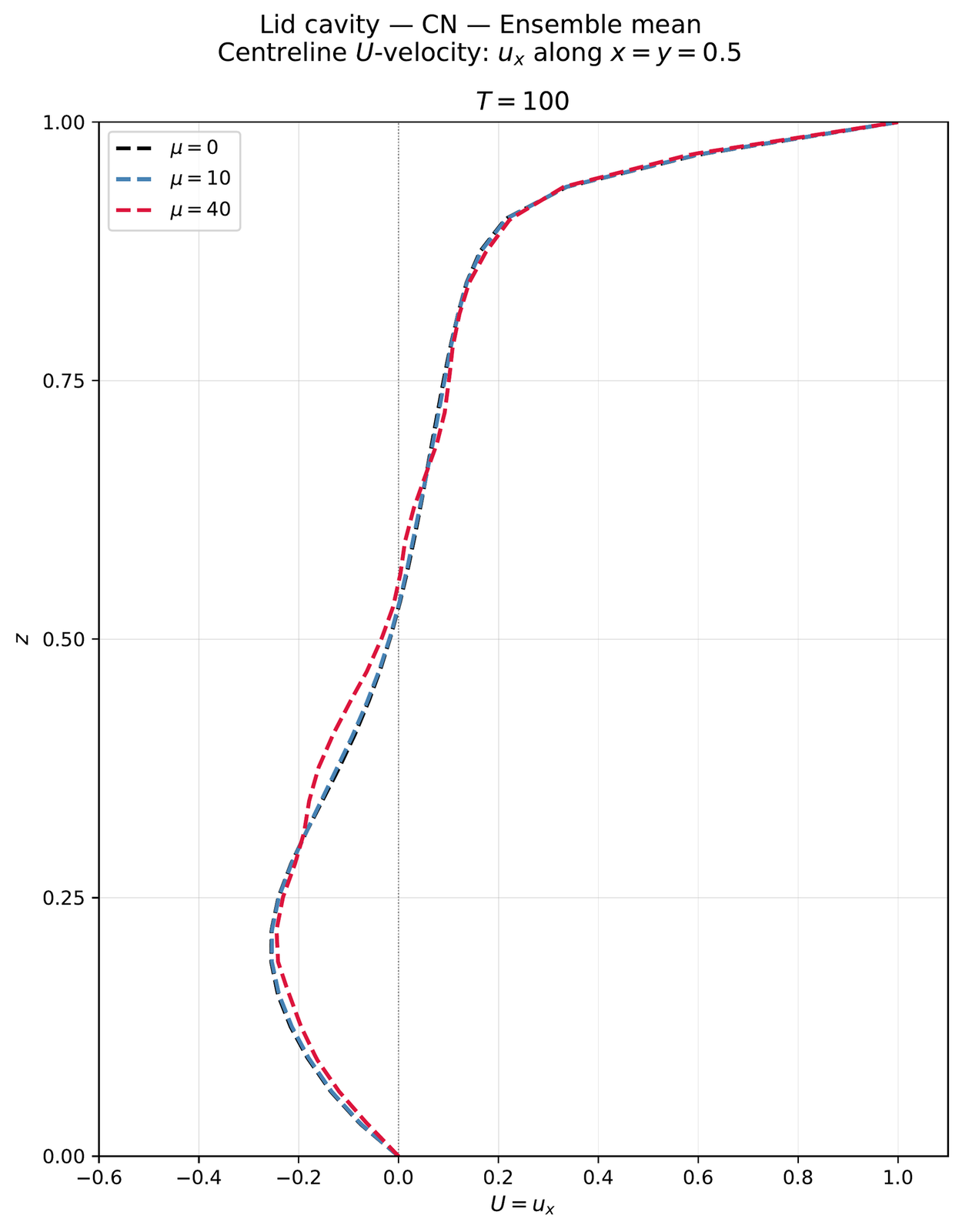}
    \captionsetup{width=1\textwidth}
    \caption{For CN scheme, Re=500}
\label{fig:lidcavity_CN_mean_centerline}
\end{minipage}
\end{figure}

\noindent\underline{2D streamlines and Vorticity:} Figures \ref{fig:lidcavity_IE_mean_T020_slices}, \ref{fig:lidcavity_CN_mean_T020_slices}, and \ref{fig:IE_mean_T100_streamlines_vorticity}--\ref{fig:CN_realization_T100_streamlines_vorticity} display the streamline of the velocity and vorticity across XY-plane $z=0.5$, XZ-plane $y=0.5$,  YZ-plane $z=0.5$ with odd columns representing velocity and even ones, vorticity. The streamlines are colored according to the amplitude of the velocity where as for vorticity the color convey the rotational nature of the fluid. In streamlines, the bright \textit{yellow-white} portion mark the  high-speed regions, concentrated along the lid ($z=1$) and the outer rim of the primary vortex core, 
the \textit{orange-red} part represents the moderate speed and trace inward movement of fluid towards the vortex center where as the 
dark \textit{purple-black} marks the near-stagnation regions at the vortex core and at the corner zones where fluid motion restricted by the wall stresses.  For vorticity \textit{blue} marks the clockwise rotation associated with the primary lid-driven vortex, \textit{red} displays the counter-clockwise rotation and characteristic
of secondary corner vortices forming due to wall friction, while \textit{white} represents the stagnation regions and shear layers.
\begin{figure}[htbp]
    \centering
    \includegraphics[width=\textwidth]{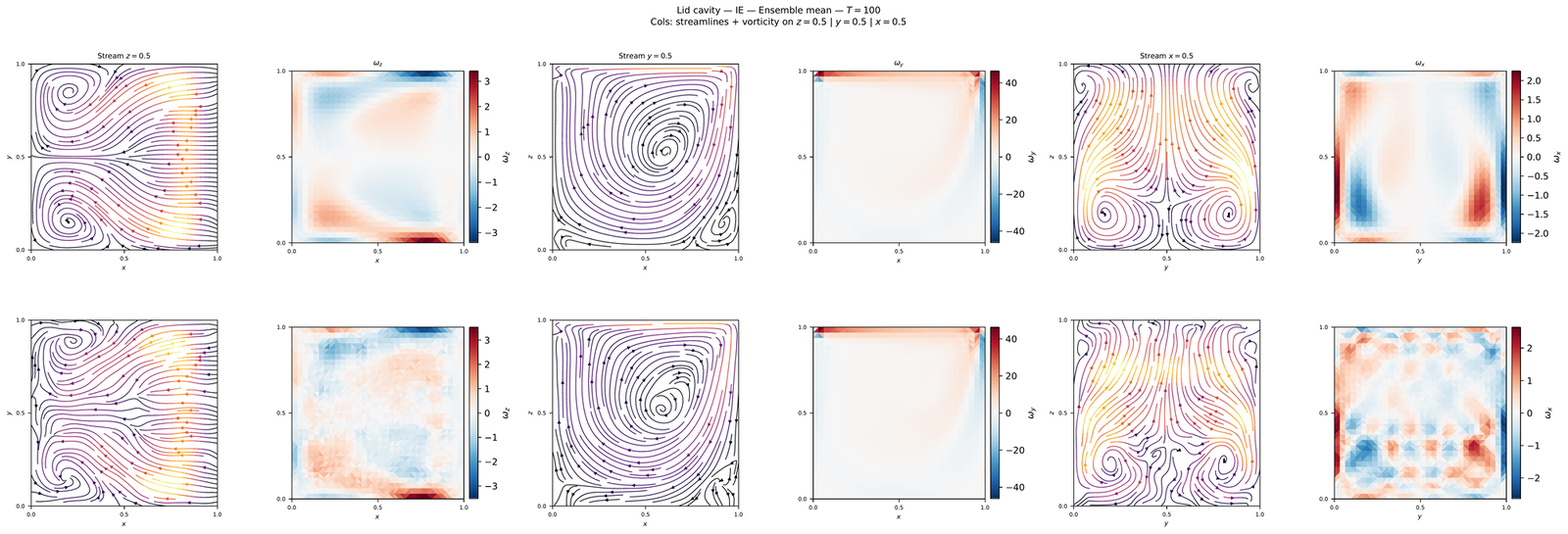}
    \caption{Streamlines and vorticity of solution at $T=100$ for IE scheme}\label{fig:IE_mean_T100_streamlines_vorticity}
\end{figure}
\begin{figure}[htbp]
    \centering
    \includegraphics[width=\textwidth]{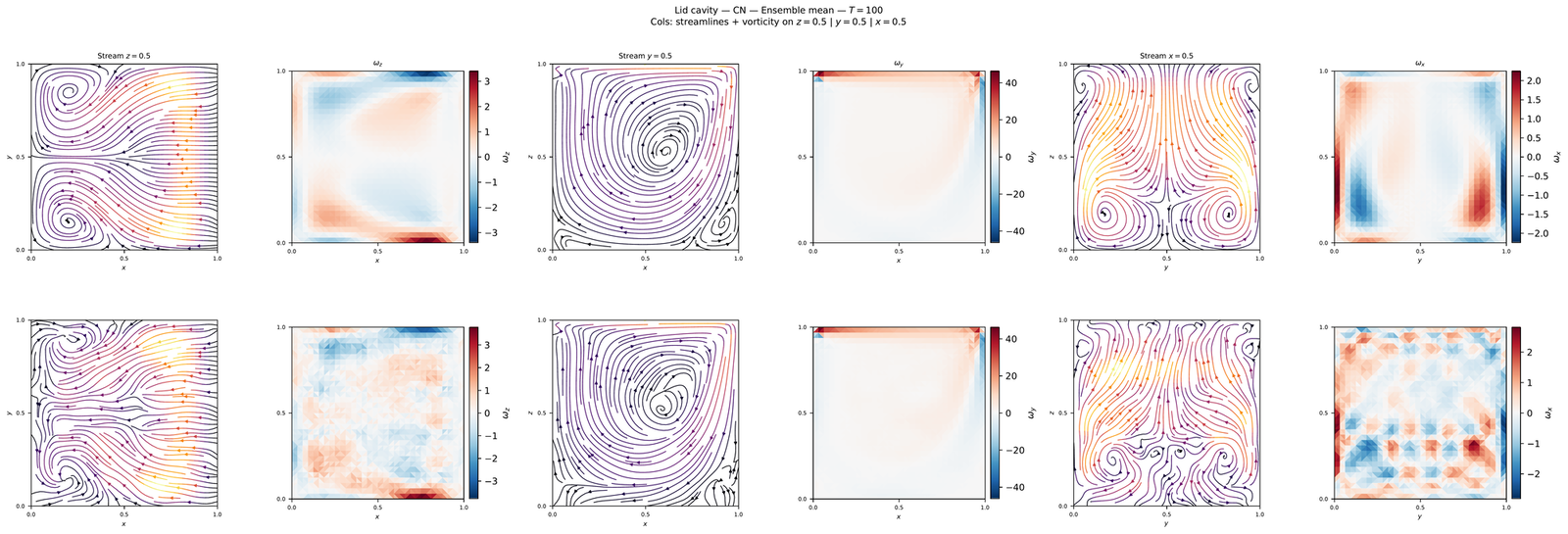}
    \caption{Streamlines and vorticity of solution at $T=100$ for CN scheme}\label{fig:CN_mean_T100_streamlines_vorticity}
\end{figure}
\begin{figure}[htbp]
    \centering
    \includegraphics[width=\textwidth]{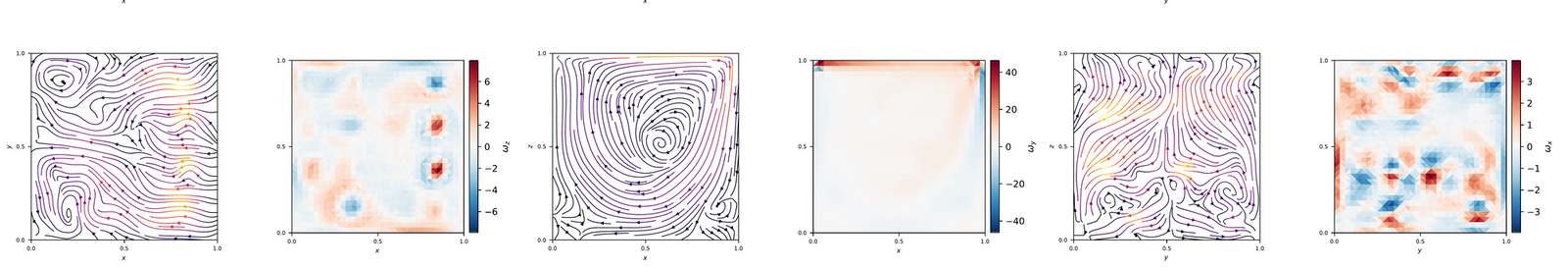}
    \caption{Streamlines and vorticity of solution at $T=100$, for one realisation of the noise  for IE scheme}\label{fig:IE_realization_T100_streamlines_vorticity}
\end{figure}
\begin{figure}[htbp]
    \centering
    \includegraphics[width=\textwidth]{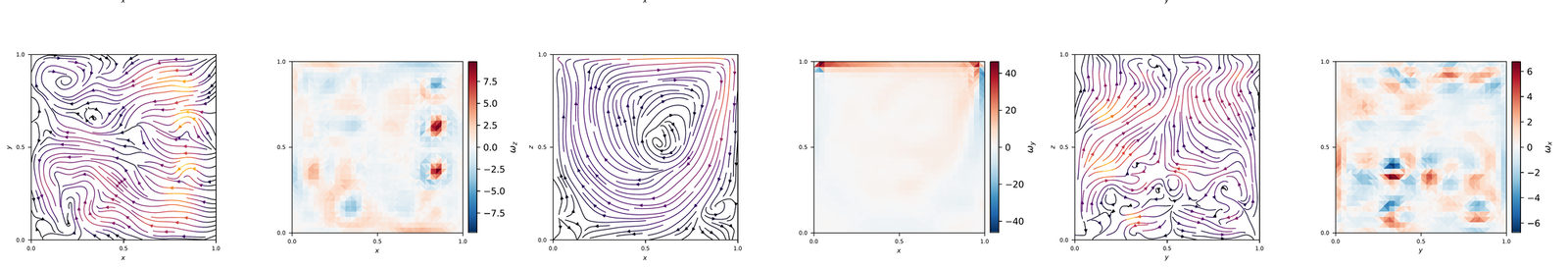}
    \caption{Streamlines and vorticity of solution at $T=100$ for for one realisation of the noise for CN scheme}\label{fig:CN_realization_T100_streamlines_vorticity}
\end{figure}

Figure \ref{fig:IE_mean_T100_streamlines_vorticity} (IE) and \ref{fig:CN_mean_T100_streamlines_vorticity} (CN) demonstrate the streamline of velocity and vorticity 
at $T=100$ for $\mu = 0$ (deterministic solution) at row 1 and $\mu = 40$ at row 2. It is observed that the flow has reached a statistically stationary regime at $T=100$.
On the plane $z=0.5$  dominant primary recirculation vorteices occupy
the upper and lower portion of the cavity and a downward return flow along the side walls. On the plane $x=0.5$ the cross-sectional structure reveals the
three-dimensional nature of the flow: the vortex core is elongated
in the spanwise direction with weaker tertiary recirculation at the end walls.
As $\mu$ increases from $0$ to $40$, the stochastic forcing continuously
injects energy into the flow, causing the vortex core to shift,
the streamline topology to become more complex and vorticity magnitude to intensify. Moreover, we observed that mean streamlines remain topologically consistent and the large-scale rotational structure of the flow is preserved (i.e. mean vorticity remains stable) across
both the IE and CN schemes while for single realization the Figures \ref{fig:IE_realization_T100_streamlines_vorticity} (IE) and \ref{fig:CN_realization_T100_streamlines_vorticity} (CN) convey the stochastic distortions of the vortex and significant fluctuation in rotational structure of the flow for $\mu = 40$ at $T = 100$.

In the introduction, the Figures \ref{fig:lidcavity_IE_mean_T020_slices} (IE) and \ref{fig:lidcavity_CN_mean_T020_slices} (CN) mark the mean value of the numerical solution at $T=20$ computed over 1000 realisation of noise $\mu = 10$ in row 2 and $\mu = 40$ in row 3 along with the deterministic solution (i.e. $\mu = 0$ in row 1), respectively. The deterministic solution and the solution for $\mu = 10$ show  similar results for streamline flow and vorticity where as larger noise intensity (i.e., $\mu = 40$) differs significantly across the schemes.

\vspace{0.1cm}
\noindent\underline{Volumetric streamline:} Figures~\ref{fig:lidcavity_IE_mean_T020_streamlines_mu0_mu40_back_oblique} (IE) and \ref{fig:lidcavity_CN_mean_T020_streamlines_mu0_mu40_back_oblique} (CN), in introduction display
spaghetti-style volumetric streamlines in the back view (column 1) and oblique view (column 2),
for $\mu = 0$ in row 1 and   $\mu = 40$ in row 2,  at $T = 20$.  At $\mu=0$ the flow exhibits a classical single-vortex topology; accompanied by two compact counter-rotating
corner vortices in the lower cavity, visible in the oblique view. As $\mu$ increases to $40$, the continuous injection of kinetic energy
by the stochastic forcing progressively destabilizes the vortex structure.

Figures \ref{fig:paper_stream_IE_mean_mu40_T050_T100} (IE) and \ref{fig:stream_CN_mean_mu40_T050_T100}  (CN) mark the ensemble mean of volumetric streamline for $\mu =40$ at time $T=50$ at row 1 and $T=100$ at row 2, with broadened vortex core and
weakened corner vortices, while single realisations in Figures \ref{fig:stream_IE_realization_mu40_T050_T100} (IE) and  \ref{fig:CN_realization_mu40_T050_T100} (CN) exhibit
multiple competing vortex cores and heavily tangled trajectories,
with the number of identifiable vortex structures increasing remarkably.
\begin{figure}[htbp]
\centering
\begin{minipage}[t]{0.48\textwidth}
    \centering
    \includegraphics[width=\linewidth]{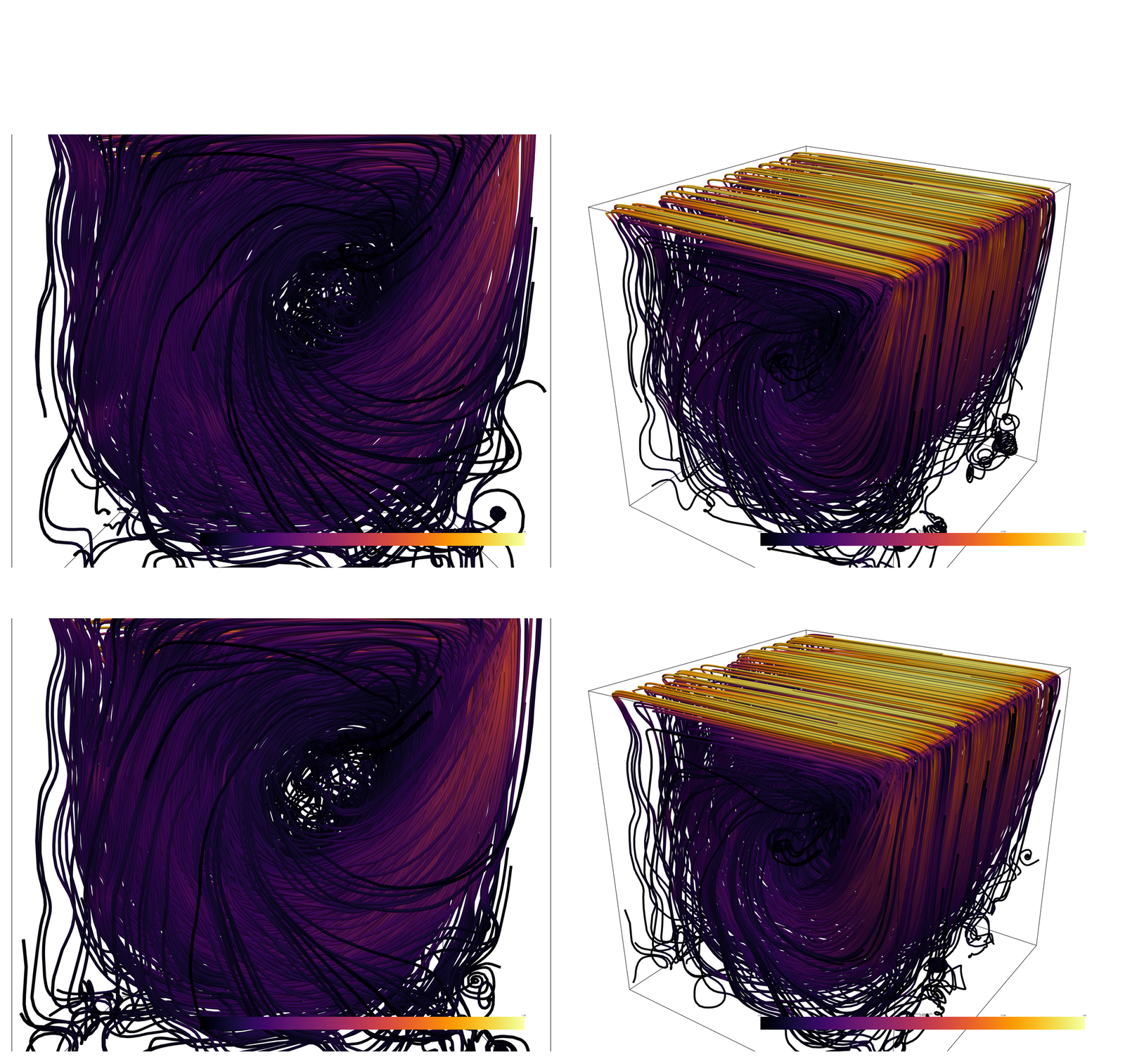}
    \captionsetup{width=1\textwidth}
    \caption{Volumetric streamlines: for $\mu = 40$ for IE scheme}
\label{fig:paper_stream_IE_mean_mu40_T050_T100}
\end{minipage}
\hfill
\begin{minipage}[t]{0.48\textwidth}
    \centering
    \includegraphics[width=\linewidth]{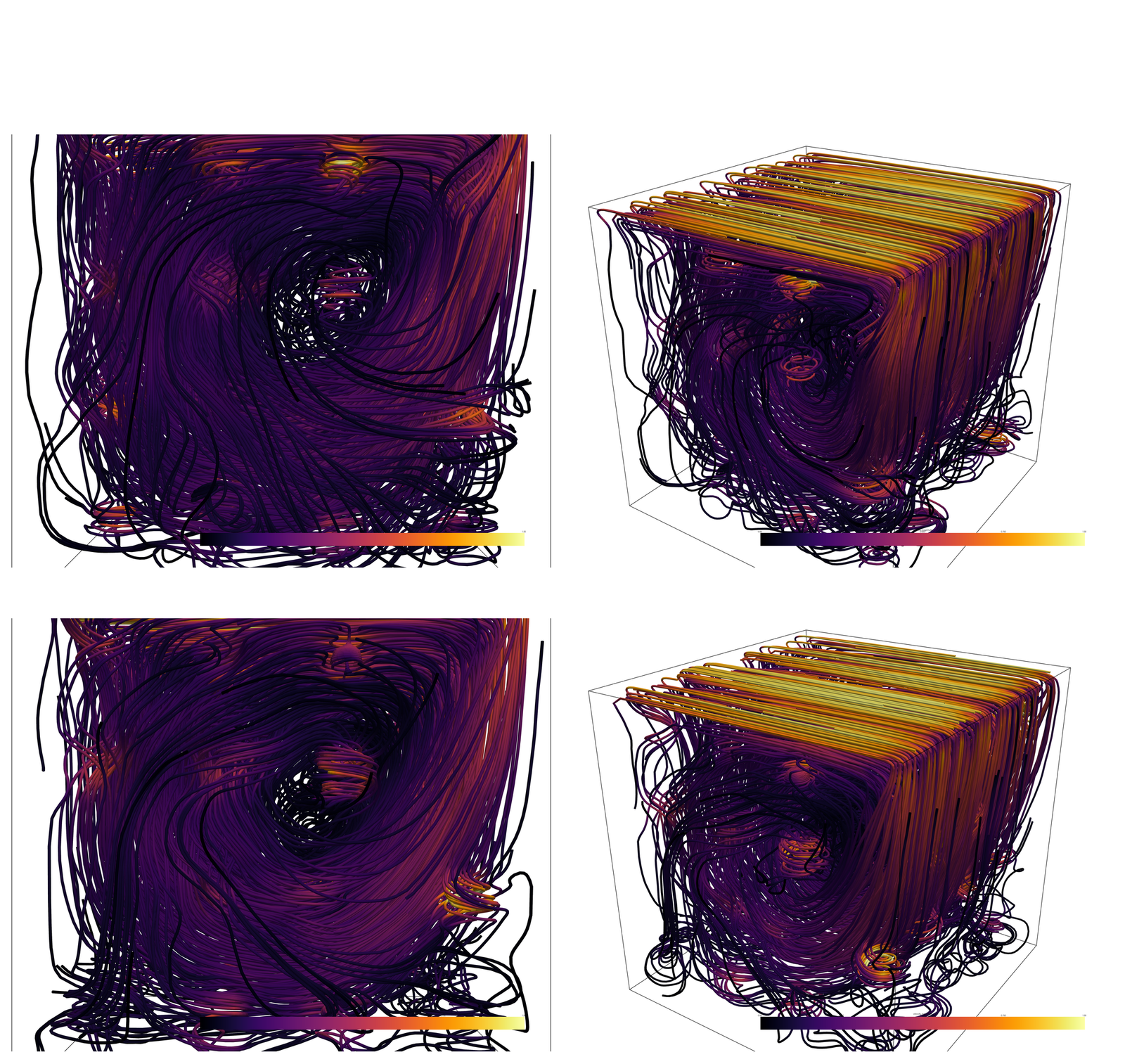}
    \captionsetup{width=1\textwidth}
    \caption{Volumetric streamlines: for $\mu = 40$ for IE scheme with one realisation of noise}
\label{fig:stream_IE_realization_mu40_T050_T100}
\end{minipage}
\end{figure}
\begin{figure}[htbp]
\centering
\begin{minipage}[t]{0.48\textwidth}
    \centering
    \includegraphics[width=\linewidth]{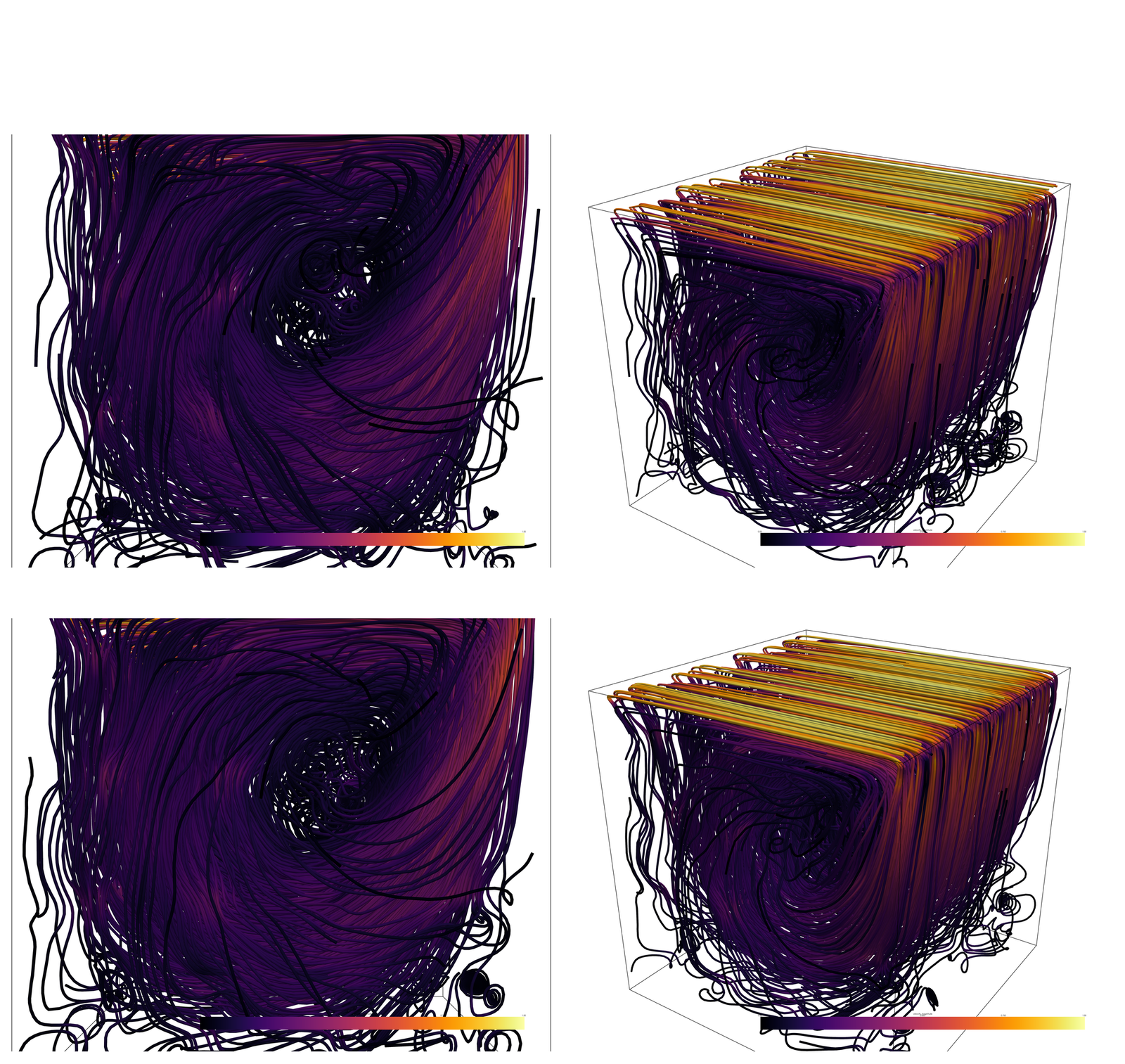}
    \captionsetup{width=1\textwidth}
    \caption{Volumetric streamlines: for $\mu = 40$ for CN scheme}
\label{fig:stream_CN_mean_mu40_T050_T100}
\end{minipage}
\hfill
\begin{minipage}[t]{0.48\textwidth}
    \centering
    \includegraphics[width=\linewidth]{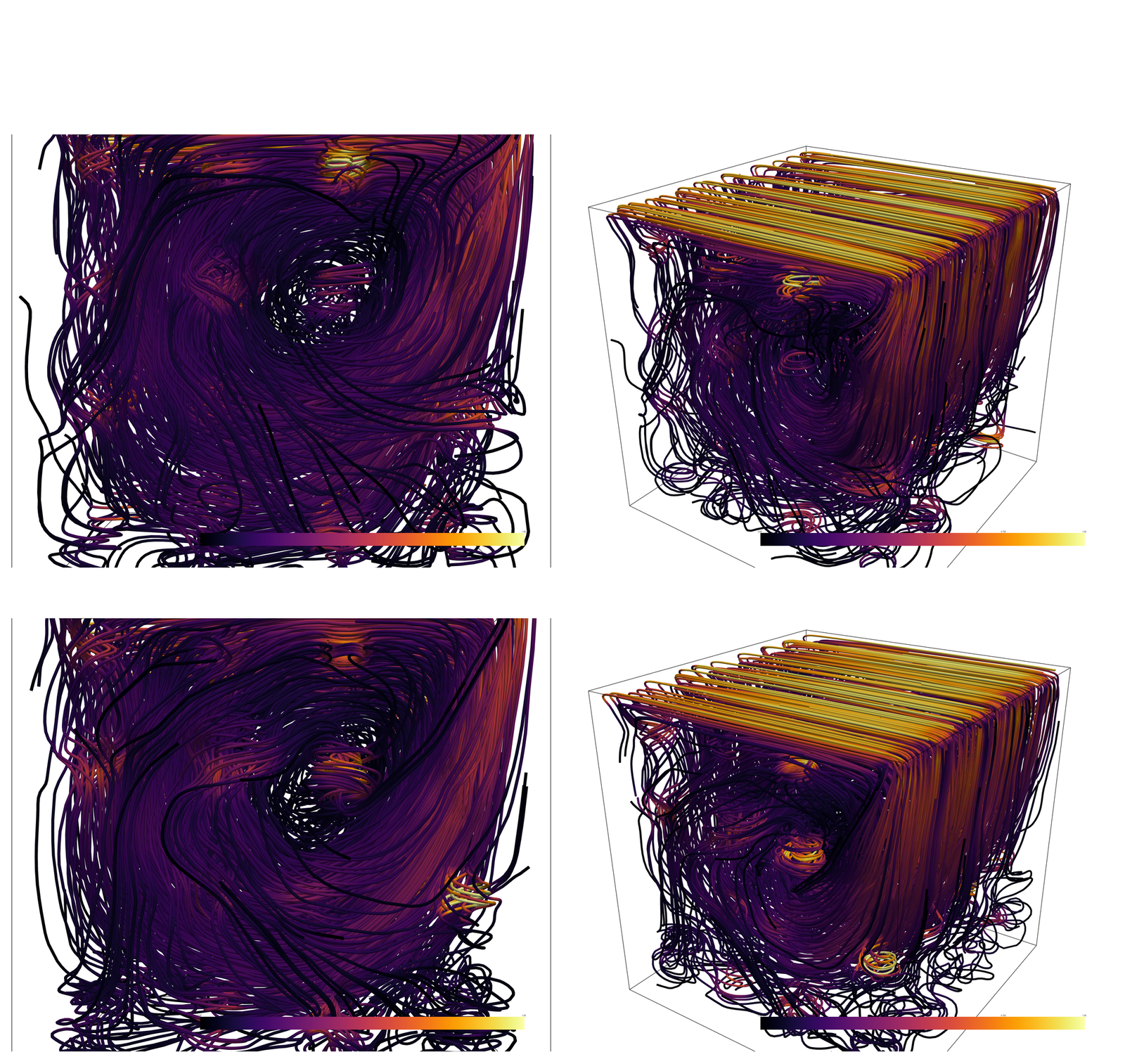}
    \captionsetup{width=1\textwidth}
    \caption{Volumetric streamlines: for $\mu = 40$ for CN scheme with one realisation of noise}
\label{fig:CN_realization_mu40_T050_T100}
\end{minipage}
\end{figure}

\vspace{0.1cm}
\noindent\underline{Iso-surface:} Figures~\ref{fig:lidcavity_IE_mean_T020_isosurface_mu0_mu40_back_oblique} (IE) and  \ref{fig:lidcavity_CN_mean_T020_isosurface_mu0_mu40_back_oblique} (CN) show iso-surfaces of the magnitude of the velocity at $10\%$,
$20\%$, and $30\%$ of the lid speed, coloured dark purple, teal, and
yellow-green (viridis), respectively (where the figures for $\mu = 0$ are placed in row 1 and $\mu = 40$ in row 2 with back view in column 1 and oblique view in column 2) . The dark-purple outermost shell marks the spatial extent over which
the lid momentum has penetrated the fluid; the teal shell delineates
the transition from driven to near-stagnant flow; while the yellow-green
innermost shell identifies the high-speed core of the primary vortex directly energised by the lid. At $\mu=0$, $T=100$, the three nested shells are smooth and connected, reflecting the organised single-vortex structure.

In Figure \ref{fig:iso_IE_mean_mu40_T050_T100} (IE) and \ref{fig:iso_CN_mean_mu40_T050_T100} (CN), we mark velocity magnitude iso-surface for $\mu=40$ at $T=50$ in row 1 and $T=100$ at row 2. The ensemble-mean shells dilate outward at both $T=50$
and $T=100$, reflecting elevated mean kinetic energy, while remaining topologically connected.
For single realizations in Figures \ref{fig:iso_IE_realization_mu40_T050_T100} (IE) and \ref{fig:iso_CN_realization_mu40_T050_T100} (CN), the shells fragment into multiple
disconnected lobes; increasing at the innermost level, with
fragmentation intensifying from $T=50$ to $T=100$.
IE and CN mean iso-surfaces are visually identical; where as in single realization fragmentation patterns differ due to independent sampling.
\begin{figure}[htbp]
\centering
\begin{minipage}[t]{0.48\textwidth}
    \centering
    \includegraphics[width=\linewidth]{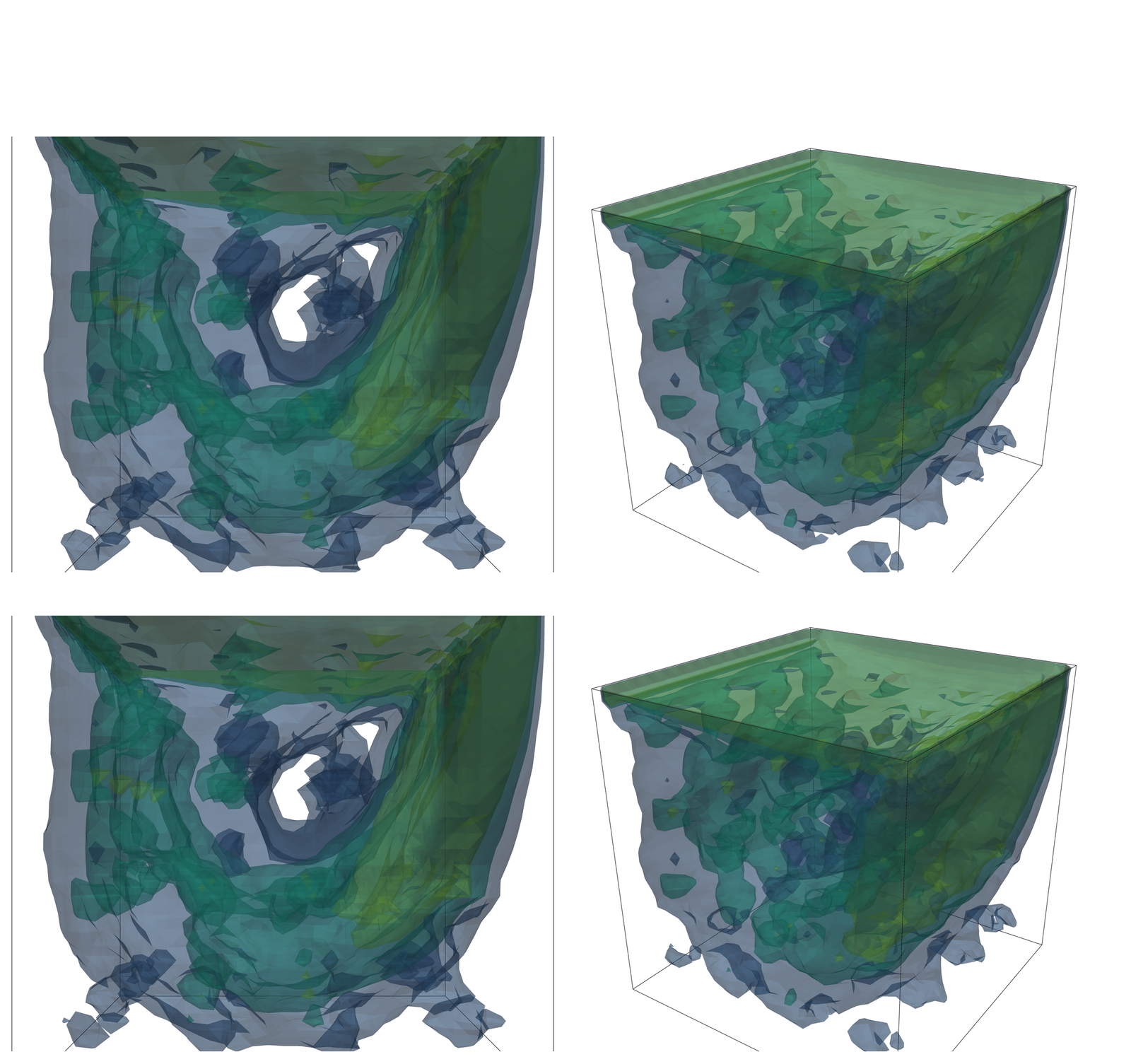}
    \captionsetup{width=1\textwidth}
    \caption{Velocity magnitude Iso surface for IE scheme with $\mu = 40$}
\label{fig:iso_IE_mean_mu40_T050_T100}
\end{minipage}
\hfill
\begin{minipage}[t]{0.48\textwidth}
    \centering
    \includegraphics[width=\linewidth]{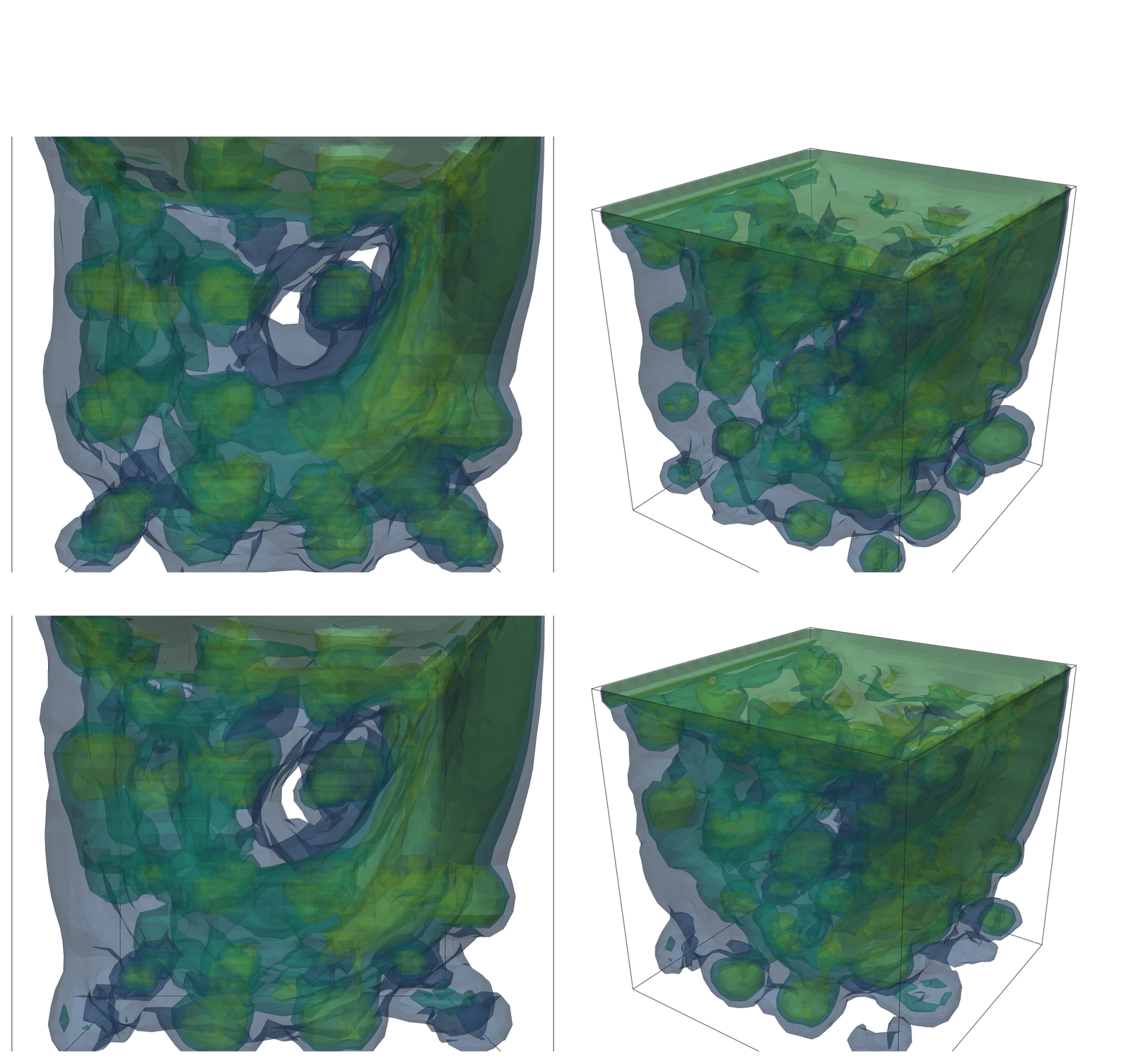}
    \captionsetup{width=1\textwidth}
    \caption{Velocity magnitude Iso surface for one realisation of noise  of IE scheme with $\mu = 40$}
\label{fig:iso_IE_realization_mu40_T050_T100}
\end{minipage}
\end{figure}
\begin{figure}[htbp]
\centering
\begin{minipage}[t]{0.48\textwidth}
    \centering
    \includegraphics[width=\linewidth]{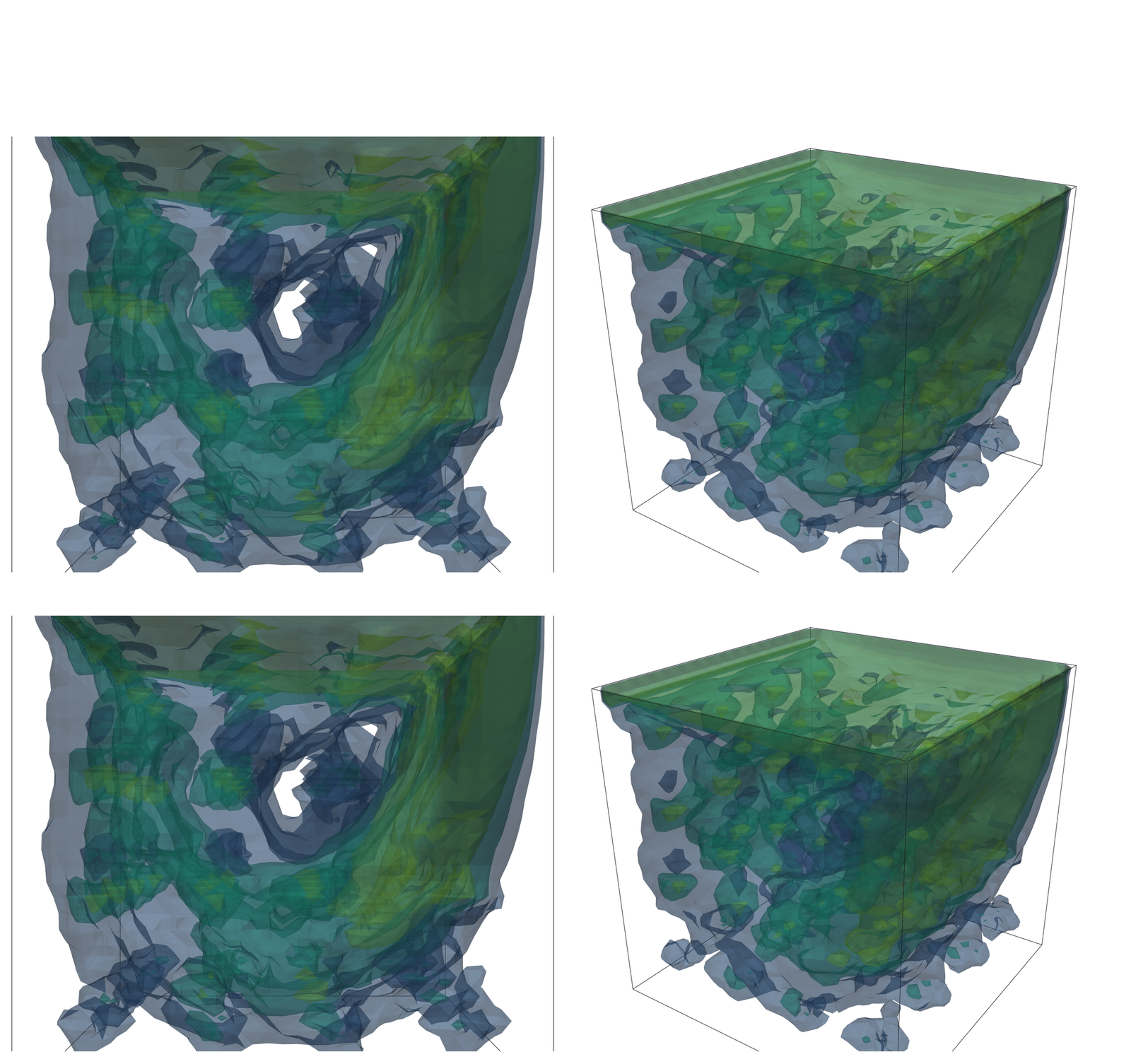}
    \captionsetup{width=1\textwidth}
    \caption{Iso surface for CN scheme with $\mu = 40$}
\label{fig:iso_CN_mean_mu40_T050_T100}
\end{minipage}
\hfill
\begin{minipage}[t]{0.48\textwidth}
    \centering
    \includegraphics[width=\linewidth]{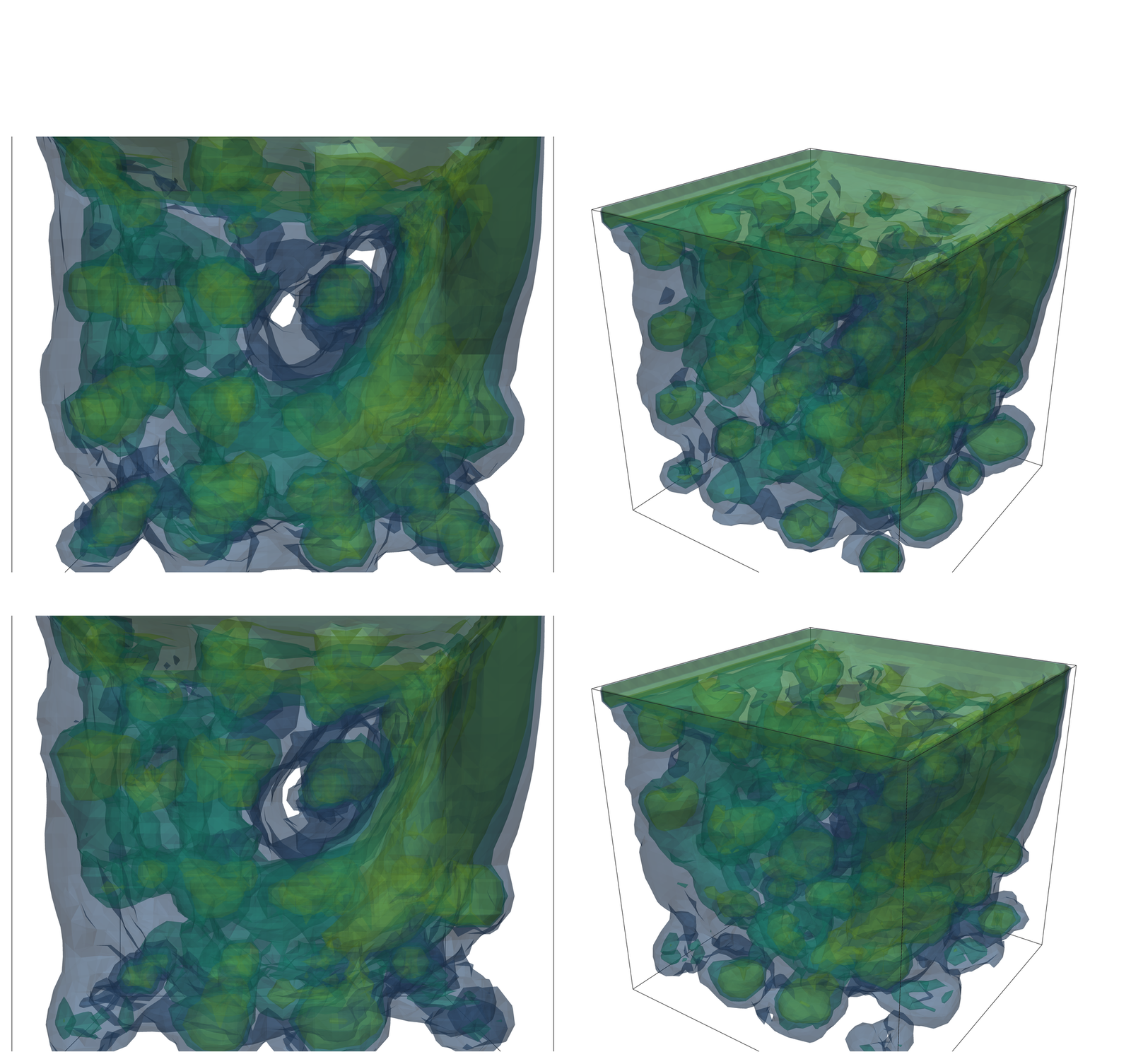}
    \captionsetup{width=1\textwidth}
    \caption{Iso surface for one realisation of noise  of CN scheme with $\mu = 40$}
\label{fig:iso_CN_realization_mu40_T050_T100}
\end{minipage}
\end{figure}
\subsection{Computational cost}\label{sec:CC}
All computations were carried out on the \textbf{amplitUDE HPC cluster}
(University of Duisburg-Essen) using MPI-based Monte Carlo
parallelisation. The spatial discretisation and time-stepping were
implemented in \textbf{Firedrake}, a high-level
finite element framework that generates optimized parallel code via
automated code generation. Each MPI process independently advances one
realsation of the noise, with no inter-process communication during
time-stepping. Since the realisations are statistically independent,
the $N_{\mathrm{MC}}$ samples are distributed across processes by
round-robin assignment, and the wall-clock time is governed by the
process carrying the largest sample count. The linear systems arising
at each time step were solved using \textbf{MUMPS} (MUltifrontal
Massively Parallel direct Solver), invoked independently on each MPI
process.

For the rate of convergence experiment ($N_s = 1000$, five time-step levels),
the Monte Carlo ensemble was distributed across $32$ compute nodes using $704$
MPI processes.
The IE and CN schemes were executed sequentially within
each process, yielding a total wall-clock time of approximately $12\,\mathrm{h}\;21\,\mathrm{min}$,
or $6\,\mathrm{h}\;11\,\mathrm{min}$ per scheme approx. This corresponds to $198$ node-hours and $4353$ CPU-hours per scheme ($396$
node-hours and $8706$ CPU-hours in total).

For 3D lid-driven cavity experiment ($N_s = 500$, $T = 10,50,100$, and $2000$ time steps per sample), each noise level ($\mu \in \{0, 10, 40\}$) was run on $9$ compute nodes with
$504$ MPI processes ($56$ per node), so that each process handles at most one
Monte Carlo sample.
Wall-clock times per noise level were approximately $17\,\mathrm{h}\;52\,\mathrm{min}$ for IE
and $20\,\mathrm{h}\;57\,\mathrm{min}$ for CN, yielding $161$ and $189$ node-hours
(respectively $9{,}005$ and $10{,}559$ CPU-hours) per noise level.

\appendix

\medskip
\begin{center}
{Compliance with ethical standards}
\end{center}
\noindent\textbf{Conflict of Interest.} The authors declare that they have no conflict of interest.

\vspace{0.1cm }
\noindent\textbf{Data Availability.} Data sharing is not applicable to this article as no datasets were generated.
or analysed during the current study.

\section{Technical results}\label{appx}

\subsection{Peano-kernel proof of the midpoint extrapolation error}
\label{app:peano}

We recall the following deterministic estimate from
\cite[Lemma~A.1]{BBCP}.  

\begin{lemma}[Midpoint extrapolation error]
\label{lem:deltastar-average}
Let \((\mathbb X;\| \cdot\|_{\mathbb X})\) be a Banach space, let \(\beta\in(0,1)\), and let
\({\bf y}\in C^1([0,T];\mathbb X)\) with
\(\partial_t{\bf y}\in C^\beta([0,T];\mathbb X)\).  For \(n=0,\dots,N-1\), define
\[
\delta_{\star,\,n+\frac12}
:=
\frac{1}{\tau}\int_{t_n}^{t_{n+1}}{\bf y}(s)\,\mathrm ds
-
\left(
\frac32{\bf y}(t_n)-\frac12{\bf y}(t_{n-1})
\right),
\qquad {\bf y}(t_{-1}):={\bf y}(t_0).
\]
Then there exists a constant \(C>0\), independent of \(n\) and \(\tau\),
such that
\begin{equation}\label{eq:deltastar-average}
\|\delta_{\star,\,n+\frac12}\|_{\mathbb X}
\le
C\,[\partial_t{\bf y}]_{C^\beta([0,T];\mathbb X)}\,\tau^{1+\beta}.
\end{equation}
\end{lemma}

\subsection{Approximation of Brownian integrals}
The next result is related to \cite[Eq.~(3.10)]{CP1}. However, \cite[Eq.~(3.10)]{CP1} provides a mean-square estimate for a single time interval, whereas here we require a uniform-in-time estimate, with the supremum over all time steps taken inside the expectation. Consequently, the argument used in \cite[Eq.~(3.10)]{CP1}, which relies only on independence and second-moment estimates for Brownian increments, is not sufficient in the present setting. We instead use higher moment bounds for Brownian increments, combined with a summability argument over the time grid. For this reason, we include the complete proof.
\begin{lemma}[Uniform-in-time Brownian quadrature estimate]
\label{lem:BM-quadrature-sup}
Let \(T>0\), $M\in\mathbb N$ and set $\tau=\frac{1}{M}$ and \(N=T/\tau\). Assume that \(M=\tau^{-1}\in\mathbb N\).
Let \(W\) be a cylindrical Wiener process on a separable Hilbert space
\(\mathfrak U\). Let \(s\in\{1,2\}\), and assume that
\[
  \Phi\in L_2(\mathfrak U;\mathbb H^s).
\]
For \(n=0,\dots,N-1\), define
\[
  \mathcal Q_n^W
  :=
  \frac1{\tau}\int_{t_n}^{t_{n+1}}W(r)\,{\rm d}r,
  \qquad
  \Phi\mathcal I_n^W
  :=
  \sum_{\ell=1}^{M}\tau\, W(t_{n,\ell}),
\]
where
\[
  t_{n,\ell}:=t_n+\ell\tau^2,
  \qquad \ell=0,\dots,M.
\]
Then, for every \(\varepsilon>0\), there exists a constant
\(C_{\varepsilon,\Phi,T}>0\), independent of \(\tau\), such that
\begin{equation}\label{eq:BM-quadrature-sup}
  \mathbb E\left[
  \sup_{0\le n\le N-1}
  \|\Phi(\mathcal Q_n^W-\mathcal I_n^W)\|_{\mathbb H^s}^{2}
  \right]
  \le
  C_{\varepsilon,\Phi,T}\tau^{3-\varepsilon}.
\end{equation}
\end{lemma}

\begin{proof}
We first prove a higher-moment estimate on one fixed interval. By the
stationarity of Brownian increments, it is enough to consider the interval
\([0,\tau]\). Let
\[
  h:=\tau^2,
  \qquad
  t_\ell:=\ell h,
  \qquad \ell=0,\dots,M.
\]
Then \(Mh=\tau\), since \(M=\tau^{-1}\). On \([0,\tau]\), we have
\[
  \Phi\mathcal Q^W
  :=
  \frac1{\tau}\int_0^\tau \Phi W(r)\,{\rm d}r,
  \qquad
  \Phi\mathcal I^W
  :=
  \sum_{\ell=1}^{M}\tau\,\Phi W(t_\ell).
\]
Since \(h=\tau^2\), we can write
\[
  \Phi\mathcal I^W
  =
  \frac1{\tau}
  \sum_{\ell=1}^{M}h\,\Phi W(t_\ell).
\]
Therefore,
\[
\begin{aligned}
  \Phi(\mathcal Q^W-\mathcal I^W)
  &=
  \frac1{\tau}
  \sum_{\ell=1}^{M}
  \left(
  \int_{t_{\ell-1}}^{t_\ell}\Phi W(r)\,{\rm d}r
  -
  h\,\Phi W(t_\ell)
  \right) \\
  &=
  \frac1{\tau}
  \sum_{\ell=1}^{M}
  \int_{t_{\ell-1}}^{t_\ell}
  \bigl(\Phi W(r)-\Phi W(t_\ell)\bigr)\,{\rm d}r .
\end{aligned}
\]
For \(r\in[t_{\ell-1},t_\ell]\), we have
\[
  \Phi W(r)-\Phi W(t_\ell)
  =
  -\int_r^{t_\ell}\Phi\,{\rm d}W(\rho).
\]
Hence, by the stochastic Fubini theorem,
\[
\begin{aligned}
  \int_{t_{\ell-1}}^{t_\ell}
  \bigl(\Phi W(r)-\Phi W(t_\ell)\bigr)\,{\rm d}r
  &=
  -\int_{t_{\ell-1}}^{t_\ell}
  \int_r^{t_\ell}\Phi\,{\rm d}W(\rho)\,{\rm d}r \\
  &=
  -\int_{t_{\ell-1}}^{t_\ell}
  (\rho-t_{\ell-1})\,\Phi\,{\rm d}W(\rho).
\end{aligned}
\]
Consequently,
\[
  \Phi(\mathcal Q^W-\mathcal I^W)
  =
  -\frac1{\tau}
  \sum_{\ell=1}^{M}
  \int_{t_{\ell-1}}^{t_\ell}
  (\rho-t_{\ell-1})\,\Phi\,{\rm d}W(\rho).
\]
Equivalently,
\[
  \Phi(\mathcal Q^W-\mathcal I^W)
  =
  \int_0^\tau g(\rho)\,\Phi\,{\rm d}W(\rho),
\]
where
\[
  g(\rho):=
  -\frac1{\tau}(\rho-t_{\ell-1}),
  \qquad
  \rho\in[t_{\ell-1},t_\ell].
\]
Since the integrand is deterministic, \(\Phi(\mathcal Q^W-\mathcal I^W)\)
is a centered Gaussian random variable with values in
\(\mathbb H^s(\mathbb T^3)\). By It\^o-isometry,
\[
\begin{aligned}
  \mathbb E
  \big[\|\Phi(\mathcal Q^W-\mathcal I^W)\|_{\mathbb H^s}^{2}\big]
  &=
  \int_0^\tau
  |g(\rho)|^2
  \|\Phi\|_{L_2(\mathfrak U;\mathbb H^s)}^2
  \,{\rm d}\rho \\
  &=
  \frac1{\tau^2}
  \sum_{\ell=1}^{M}
  \int_{t_{\ell-1}}^{t_\ell}
  (\rho-t_{\ell-1})^2\,{\rm d}\rho\,
  \|\Phi\|_{L_2(\mathfrak U;\mathbb H^s)}^2 \\
  &=
  \frac1{\tau^2}\,M\,\frac{h^3}{3}\,
  \|\Phi\|_{L_2(\mathfrak U;\mathbb H^s)}^2= \frac{\tau^3}{3}
  \|\Phi\|_{L_2(\mathfrak U;\mathbb H^s)}^2.
\end{aligned}
\]
Let \(q\in\mathbb N\). Since
\(\Phi(\mathcal Q^W-\mathcal I^W)\) is a centered Gaussian random variable
in the Hilbert space \(\mathbb H^s(\mathbb T^3)\), the Gaussian moment
estimate gives
\[
  \mathbb E
  \big[\|\Phi(\mathcal Q^W-\mathcal I^W)\|_{\mathbb H^s}^{2q}\big]
  \le
  C_q
  \left(
  \mathbb E
  \|\Phi(\mathcal Q^W-\mathcal I^W)\|_{\mathbb H^s}^{2}
  \right)^q.
\]
Hence
\[
  \mathbb E
  \big[\|\Phi(\mathcal Q^W-\mathcal I^W)\|_{\mathbb H^s}^{2q}\big]
  \le
  C_{q,\Phi}\tau^{3q}.
\]
Returning to the general interval \([t_n,t_{n+1}]\), stationarity of
Brownian increments gives
\begin{equation}\label{eq:single-BM-high-moment-cyl}
  \mathbb E
  \big[\|\Phi(\mathcal Q_n^W-\mathcal I_n^W)\|_{\mathbb H^s}^{2q}\big]
  \le
  C_{q,\Phi}\tau^{3q},
  \qquad n=0,\dots,N-1.
\end{equation}
We now pass to the supremum in \(n\).
For \(q\ge1\), we have
\[
  \sup_{0\le n\le N-1} \|\Phi(\mathcal Q_n^W-\mathcal I_n^W)\|_{\mathbb H^s}^{2}
  \le
  \left(
  \sum_{n=0}^{N-1} \|\Phi(\mathcal Q_n^W-\mathcal I_n^W)\|_{\mathbb H^s}^{2q}
  \right)^{1/q}.
\]
Therefore, by Jensen's inequality and \eqref{eq:single-BM-high-moment-cyl},
\[
\begin{aligned}
  \mathbb E\left[
  \sup_{0\le n\le N-1}
  \|\Phi(\mathcal Q_n^W-\mathcal I_n^W)\|_{\mathbb H^s}^{2}
  \right]
  &\le
  \left(
  \sum_{n=0}^{N-1}
  \mathbb E
  \big[\|\Phi(\mathcal Q_n^W-\mathcal I_n^W)\|_{\mathbb H^s}^{2q}\big]
  \right)^{1/q} \\
  &\le
  \left(
  N C_{q,\Phi}\tau^{3q}
  \right)^{1/q}.
\end{aligned}
\]
Since \(N=T/\tau\), this yields
\[
  \mathbb E\left[
  \sup_{0\le n\le N-1}
  \|\Phi(\mathcal Q_n^W-\mathcal I_n^W)\|_{\mathbb H^s}^{2}
  \right]
  \le
  C_{q,\Phi,T}\tau^{3-\frac1q}.
\]
Finally, given \(\varepsilon>0\), choose \(q\in\mathbb N\) such that
\(q>1/\varepsilon\). Then
\[
  3-\frac1q\ge 3-\varepsilon,
\]
and therefore
\[
  \mathbb E\left[
  \sup_{0\le n\le N-1}
  \|\Phi(\mathcal Q_n^W-\mathcal I_n^W)\|_{\mathbb H^s}^{2}
  \right]
  \le
  C_{\varepsilon,\Phi,T}\tau^{3-\varepsilon}.
\]
This proves \eqref{eq:BM-quadrature-sup}.
\end{proof}
The next lemma is analogous to \cite[Lemma~A.2]{BBCP}.  We give the
proof in full detail because the present stopped error analysis requires a
uniform-in-time control of the Brownian quadrature error.  More precisely,
we need an estimate for the pathwise supremum over all time intervals,
taken inside the expectation.  This is stronger than the single-interval
mean-square estimate used in \cite[Lemma~A.2]{BBCP}, where independence
and fourth-moment bounds for Brownian increments are sufficient.  Here we
use higher moment estimates for the increments and then apply a
summability argument over the time grid.
\begin{lemma}[Uniform-in-time approximation of the Brownian triple integral]
\label{lem:triple-BM-sup}
Let \(T>0\), let \(N=T/\tau\), and assume that \(M=\tau^{-1}\in\mathbb N\).
Let \(W\) be a cylindrical Wiener process on a separable Hilbert space
\(\mathfrak U\), and assume that
\[
  \Phi\in L_2(\mathfrak U;\mathbb H^2).
\]
Set
\[
  Z(t):=\Phi W(t),\qquad t\in[0,T].
\]
For \(n=0,\dots,N-1\), define
\[
  \mathcal Q_n^{W^2}
  :=
  \frac{1}{\tau^3}
  \int_{t_n}^{t_{n+1}}\int_{t_n}^{t_{n+1}}\int_{t_n}^{t_{n+1}}
  \bigl(Z(t)-Z(s)\bigr)\otimes
  \bigl(Z(t)-Z(r)\bigr)
  \,{\rm d}s\,{\rm d}t\,{\rm d}r .
\]
Let
\[
  h:=\tau^2,\qquad
  t_{n,\ell}:=t_n+\ell h,\qquad \ell=0,\dots,M,
\]
and define
\[
  \mathcal I_n^{W^2}
  :=
  \frac{h^3}{\tau^3}
  \sum_{\ell=1}^{M}\sum_{k=1}^{M}\sum_{j=1}^{M}
  \bigl(Z(t_{n,\ell})-Z(t_{n,k})\bigr)\otimes
  \bigl(Z(t_{n,\ell})-Z(t_{n,j})\bigr).
\]
Then, for every \(\varepsilon>0\), there exists a constant
\(C_{\varepsilon,\Phi,T}>0\), independent of \(\tau\), such that
\begin{equation}\label{eq:triple-BM-sup}
  \mathbb E\Big[
  \sup_{0\le n\le N-1}
  \|\mathcal Q_n^{W^2}-\mathcal I_n^{W^2}\|_{\mathbb L^2}^{2}
  \Big]
  \le
  C_{\varepsilon,\Phi,T}\,\tau^{3-\varepsilon}.
\end{equation}
\end{lemma}

\begin{proof}
We first prove a higher-moment estimate on one fixed interval
\([t_n,t_{n+1}]\). By stationarity of Brownian increments, it is enough
to consider the interval \([0,\tau]\). For simplicity, we drop the index
\(n\).
Define
\[
  B
  :=
  \int_0^\tau\int_0^\tau\int_0^\tau
  F(t,s,r)\,{\rm d}s\,{\rm d}t\,{\rm d}r,
\]
where
\[
  F(t,s,r)
  :=
  \bigl(Z(t)-Z(s)\bigr)\otimes
  \bigl(Z(t)-Z(r)\bigr).
\]
Similarly, define the discrete approximation
\[
  B^{\rm disc}
  :=
  h^3
  \sum_{\ell=1}^{M}\sum_{k=1}^{M}\sum_{j=1}^{M}
  F(t_\ell,t_k,t_j).
\]
Then
\[
  B=\tau^3\mathcal Q_n^{W^2},
  \qquad
  B^{\rm disc}=\tau^3\mathcal I_n^{W^2}.
\]

We partition \([0,\tau]^3\) into the boxes
\[
  R_{\ell k j}
  :=
  (t_{\ell-1},t_\ell]\times
  (t_{k-1},t_k]\times
  (t_{j-1},t_j],
  \qquad \ell,k,j=1,\dots,M.
\]
For each box, set
\[
  {\bf e}_{\ell k j}
  :=
  \int_{R_{\ell k j}}
  \bigl(F(t,s,r)-F(t_\ell,t_k,t_j)\bigr)
  \,{\rm d}s\,{\rm d}t\,{\rm d}r .
\]
Then
\[
  B-B^{\rm disc}
  =
  \sum_{\ell,k,j=1}^{M}{\bf e}_{\ell k j}.
\]
We next estimate the local errors in higher moments. Let \(q\ge1\).
For two points \((t,s,r),(t',s',r')\in[0,\tau]^3\), set
\[
  U:=Z(t)-Z(s),\qquad
  V:=Z(t)-Z(r),
\]
and
\[
  U':=Z(t')-Z(s'),\qquad
  V':=Z(t')-Z(r').
\]
Then
\[
  F(t,s,r)-F(t',s',r')
  =
  U\otimes V-U'\otimes V'
\]
and hence
\[
  F(t,s,r)-F(t',s',r')
  =
  (U-U')\otimes V
  +
  U'\otimes(V-V').
\]
Using the product estimate
\[
  \|{\bf a}\otimes{\bf b}\|_{\mathbb L^2}
  \le
  \|{\bf a}\|_{\mathbb L^2}
  \|{\bf b}\|_{\mathbb L^\infty},
\]
the Sobolev embedding
$  \mathbb H^2(\mathbb T^3)\hookrightarrow \mathbb L^\infty(\mathbb T^3),$
and the Gaussian moment estimate
\[
  \mathbb E
  \|Z(t)-Z(s)\|_{\mathbb H^2}^{p}
  \le
  C_{p,\Phi}|t-s|^{p/2},
  \qquad p\ge2,
\]
we obtain
\begin{align}
  &\mathbb E
  \Big[
  \|F(t,s,r)-F(t',s',r')\|_{\mathbb L^2(\mathbb T^3)^9}^{2q}
  \Big]
  \notag\\
  &\qquad\le
  C_{q,\Phi}\,
  \tau^{q}
  \bigl(
  |t-t'|+|s-s'|+|r-r'|
  \bigr)^{q}.
  \label{eq:F-diff-high-moment}
\end{align}
Indeed, the factors \(U,V,U',V'\) are Brownian increments over intervals
of length at most \(\tau\), while \(U-U'\) and \(V-V'\) are controlled by
increments over intervals of total length
\[
  |t-t'|+|s-s'|+|r-r'|.
\]
Now choose
\[
  (t',s',r')=(t_\ell,t_k,t_j).
\]
If \((t,s,r)\in R_{\ell k j}\), then
\[
  |t-t_\ell|+|s-t_k|+|r-t_j|
  \le 3h.
\]
Therefore, by \eqref{eq:F-diff-high-moment},
\[
  \mathbb E
  \Big[
  \|F(t,s,r)-F(t_\ell,t_k,t_j)\|_{\mathbb L^2(\mathbb T^3)^9}^{2q}
  \Big]
  \le
  C_{q,\Phi}\,\tau^q h^q .
\]
Since \(|R_{\ell k j}|=h^3\), Jensen's inequality gives
\begin{align}\label{eq:local-error-high-moment}
\begin{aligned}
  \mathbb E
  \Big[
  \|{\bf e}_{\ell k j}\|_{\mathbb L^2}^{2q}
  \Big]
  &\le
  |R_{\ell k j}|^{2q-1}
  \int_{R_{\ell k j}}
  \mathbb E
  \Big[
  \|F(t,s,r)-F(t_\ell,t_k,t_j)\|_{\mathbb L^2}^{2q}
  \Big]
  \,{\rm d}s\,{\rm d}t\,{\rm d}r  \\
  &\le
  h^{3(2q-1)} h^3 C_{q,\Phi}\tau^q h^q  \\
  &=
  C_{q,\Phi}\tau^{15q}.
  \end{aligned}
\end{align}
There are \(M^3=\tau^{-3}\) boxes. Hence, using the triangle inequality in
\(L^{2q}(\Omega;\mathbb L^2(\mathbb T^3))\), we obtain
\begin{align*}
  \|B-B^{\rm disc}\|_{L^{2q}(\Omega;\mathbb L^2)}
  &\le
  \sum_{\ell,k,j=1}^{M}
  \|{\bf e}_{\ell k j}\|_{L^{2q}(\Omega;\mathbb L^2)}
  \\
  &\le
  M^3 C_{q,\Phi}\tau^{15/2}
  =
  C_{q,\Phi}\tau^{9/2}.
\end{align*}
Therefore,
\[
  \mathbb E
  \Big[
  \|B-B^{\rm disc}\|_{\mathbb L^2}^{2q}
  \Big]
  \le
  C_{q,\Phi}\tau^{9q}.
\]
Since
\[
  B-B^{\rm disc}
  =
  \tau^3
  \bigl(
  \mathcal Q_n^{W^2}-\mathcal I_n^{W^2}
  \bigr),
\]
we get
\begin{equation}\label{eq:single-step-high-moment}
  \mathbb E
  \Big[
  \|\mathcal Q_n^{W^2}-\mathcal I_n^{W^2}\|_{\mathbb L^2}^{2q}
  \Big]
  \le
  C_{q,\Phi}\tau^{3q}.
\end{equation}
We now pass to the supremum over all time steps. By \eqref{eq:single-step-high-moment},
\begin{align*}
  &\mathbb E
  \Big[
  \sup_{0\le n\le N-1}
  \|\mathcal Q_n^{W^2}-\mathcal I_n^{W^2}\|_{\mathbb L^2(\mathbb T^3)^9}^{2}
  \Big]
  \\
  &\qquad\le
  \left(
  \mathbb E
  \Big[
  \sup_{0\le n\le N-1}
  \|\mathcal Q_n^{W^2}-\mathcal I_n^{W^2}\|_{\mathbb L^2}^{2q}
  \Big]
  \right)^{1/q}
  \\
  &\qquad\le
  \left(
  \sum_{n=0}^{N-1}
  \mathbb E
  \Big[
  \|\mathcal Q_n^{W^2}-\mathcal I_n^{W^2}\|_{\mathbb L^2}^{2q}
  \Big]
  \right)^{1/q}
  \\
  &\qquad\le
  \left(
  N C_{q,\Phi}\tau^{3q}
  \right)^{1/q}.
\end{align*}
Since \(N=T/\tau\), we obtain
\[
  \mathbb E
  \Big[
  \sup_{0\le n\le N-1}
  \|\mathcal Q_n^{W^2}-\mathcal I_n^{W^2}\|_{\mathbb L^2}^{2}
  \Big]
  \le
  C_{q,\Phi,T}\tau^{3-\frac1q}.
\]
Choosing again \(q>1/\varepsilon\)
for \(\varepsilon>0\) given we can conclude as in the proof of Lemma \ref{lem:BM-quadrature-sup}.
\end{proof}

\end{document}